\documentclass[10pt,english, oneside,reqno]{smfart}

\usepackage[T1]{fontenc}
\usepackage[english]{babel}
\usepackage[utf8]{inputenc}

\usepackage[dvipsnames]{xcolor}
\usepackage{appendix}
\usepackage{lipsum}
\usepackage{amsmath,amssymb}
\usepackage{amsfonts}
\usepackage{smfthm}
\usepackage{mathrsfs}
\usepackage{systeme}
\usepackage{stmaryrd}
\usepackage{pgfplots}
\usepackage{minted}
\usepackage{enumitem}
\usepackage{nicematrix}
\usepackage{todonotes}
\usepackage{mathabx}
\usepackage{mathtools}
\usepackage{centernot}
\usepackage{stmaryrd}

\newcommand{\NN}{\mathbb{N}}
\newcommand{\RR}{\mathbb{R}}

\renewcommand{\tilde}{\widetilde}

\newcommand{\vare}{{\varepsilon}}
\newcommand{\vphi}{{\varphi}}

\DeclareMathOperator{\Jac}{Jac}

\DeclareMathOperator{\sech}{sech}
\DeclareMathOperator{\supp}{Supp}

\usepackage{lastpage}
\usepackage{blindtext}
\usepackage{algorithm}
\usepackage{algpseudocode}
\usepackage{graphicx}
\usepackage{wrapfig}
\usepackage{tikz}

\usepackage[top=2.5cm,bottom=2.5cm,margin=2.5cm]{geometry}
\usepackage{fancyhdr}
\usepackage{hyperref}
\hypersetup{colorlinks=true, linkcolor=black, filecolor=black, citecolor = black, urlcolor=gray,}
\usepackage{tikz}
\usetikzlibrary{patterns}
\usepackage{calrsfs}
\usepackage{systeme}
\usepackage{amsfonts}
\usepackage{stmaryrd}
\usepackage{amsthm}
\usepackage{mathtools}
\usepackage{amsmath}
\usepackage{dsfont} 
\usepackage{amssymb}

\usepackage{lipsum}

\usepackage{graphicx}
\graphicspath{ {./images/} }
\newtheorem{definition}{Definition}[section]
\newtheorem{theorem}{Theorem}[section]

\newtheorem{remark}{Remark}[section]
\newtheorem{lemma}[theorem]{Lemma}
\newtheorem{proposition}[theorem]{Proposition}

\usepackage{pgfplots}
\pgfplotsset{compat=1.5}
\usepackage{mathrsfs}
\usetikzlibrary{arrows}
\usepackage[
   backend=biber,        
   sorting=nyt,          
]{biblatex}
\usepackage{csquotes}
\usepackage{graphicx} 
\usepackage{tikz}
\usetikzlibrary{positioning}
\usepackage{inputenc}

\usepackage{emptypage}

\title{Full range of infinite point blow-up exponents \\ for the critical generalized KdV equation}
\author{Nailya Manatova
\\ \textsc{Laboratoire de mathématiques de Versailles}
\\ \textsc{Université de Versailles St-Quentin en Y., CNRS}
\\ \textsc{Université Paris-Saclay}
\\ nailya.manatova@uvsq.fr}

\begin{document}

\begin{abstract}
   For the quintic, mass critical generalized Korteweg-de Vries equation, for any $\nu \in (\frac 12, 1)$, we prove the existence of  solutions in the energy space that blow up in finite time $T>0$ with the blow-up rate
   $\|\partial_x u(t)\|_{L^2} \sim (T-t)^{-\nu}$ (infinite point blow-up). 
   These solutions are constructed arbitrarily close to the family of solitons and correspond to
   the concentration of a soliton traveling at $+\infty$ in space as $t\uparrow T$.

   This complements the previous results obtained in \cite{MMPIII} on infinite point exotic blow-up,
   which were valid under the technical restriction $\nu>\frac {11}{13}$.
   The value $\nu=\frac 12$ corresponds to a critical case to be treated elsewhere.
   
   At the technical level, we implement a modification of the virial-energy functional originally introduced in \cite{MMPI,MMPIII}, to allow all $\nu > \frac 12$ and simplify the proof of energy estimates.
\end{abstract}

\maketitle

\section{Introduction}

\subsection{Presentation of the problem}
We consider the $L^2$ critical generalized Korteweg-de Vries equation (gKdV)
\begin{equation}\label{gKdV_principal_eq}
    \partial_t U + \partial_x(\partial_{xx}U + U^5)=0, \qquad (t,x)\in [0,T) \times\RR.
\end{equation}
The equation being posed in $\RR$, in the following we write $H^1$ instead of $H^1(\RR)$
and similarly for other spaces of functions.

From the works of Kato \cite{Kato-1983} and of Kenig, Ponce and Vega \cite{KPV93}, the Cauchy problem is locally well-posed in the energy space $H^1$. For a given $U_0 \in H^1$, there exists a unique maximal solution $U$ of \eqref{gKdV_principal_eq} in a functional space included in $C([0,T), H^1)$, which satisfies in addition
\begin{equation}
    \text{if }\quad T < +\infty \quad \text{then} \quad \lim_{t \uparrow T} \|\partial_x U(t)\|_{L^2} = + \infty.
\end{equation}
If $T<+\infty$, then we say that the solution blows up in finite time.

For an $H^1$ solution, the mass and the energy, defined respectively by
\begin{equation}
    M(U)(t) = \int U^2 (t,x) \, dx, \qquad E(U)(t) = \frac 12 \int(\partial_x U)^2(t,x)\, dx - \frac 16 \int U^6(t,x)\, dx\,,
\end{equation}
are conserved along the evolution in time.
 
The equation is scaling and translation invariant: if $U$ is a solution, 
then for any $\lambda>0$, $\sigma \in \RR$,
\begin{equation}
    U_{\lambda,x_0}(t,x) := \lambda^{\frac 12}U(\lambda^3 t, \lambda x + \sigma)
\end{equation}
is also a solution to \eqref{gKdV_principal_eq}. The scaling symmetry leaves the norm $L^2$ invariant, so the problem is mass critical.

The family of solitary wave solutions of \eqref{gKdV_principal_eq} is given by 
\begin{equation}
    \Big\{ \lambda^{-\frac 12}_0 \,Q\big(\lambda^{-1}_0 (x - \lambda^{-2}_0 t - \sigma_0) \big)\;\big|\;  (\lambda_0,\sigma_0) \in (0,+\infty)\times \RR \Big\}
\end{equation}
where 
\[
Q(x) = \left( 3 \sech^2(2x) \right)^{\frac 14}
\]
is the unique (up to translation) positive solution of the equation
\begin{equation}\label{equation_of_Q}
    -Q'' + Q - Q^5 =0 \quad \text{on}\; \RR.
\end{equation}
The ground state $Q$ satisfies $E(Q) =0$ and is related to the following sharp Gagliardo-Nirenberg inequality (see for instance \cite{Weinstein-1983})
\begin{equation}\label{gagliardo-nirenberg}
    \frac 13\int \phi^6 \leq \bigg( \frac{\int \phi^2}{\int Q^2} \bigg)^2\, \int (\partial_x \phi)^2 \qquad \forall \phi \in H^1.
\end{equation}

From the mass and energy conservation and the above inequality, any $H^1$ initial data $U_0$ with subcritical mass $\|U_0\|_{L^2} < \|Q\|_{L^2}$ generates a global and bounded solution of \eqref{gKdV_principal_eq} in $H^1$. This result is sharp due to the existence of
a minimal mass blow-up solution.

Indeed, minimal mass blow-up for critical gKdV equation is addressed in \cite{MMPII}:
there exists a unique (up to invariances) blow-up solution $S$ of \eqref{gKdV_principal_eq} with minimal mass $\| S(t) \|_{L^2} = \|Q\|_{L^2}$. The solution $S$ blows up in finite time and satisfies $\|\partial_x S(t)\|_{L^2} \sim t^{-1}\|Q'\|_{L^2}$ as $t \downarrow 0$
(by convention, the blow-up time is chosen at $t=0$ for the particular solution $S$).

The work \cite{MMPI}, by  Martel, Merle and Rapha\"el, provides a description of the flow of the gKdV equation for initial data such that
\begin{equation}
    \|Q\|_{L^2} < \|U_0\|_{L^2} \leq \|Q\|_{L^2} + \vare, \qquad \vare \ll1
\end{equation}
\emph{i.e.} for masses slightly above the threshold $\|Q\|_{L^2}$,
and for $H^1$ initial data with a sufficient decay on the right in space.
In such case, only three scenarios can occur: \begin{enumerate}
    \item finite time blow-up with the rate $\|\partial_x U(t)\| \sim (T-t)^{-1}$ as $t \uparrow T$;
    \item global existence and local convergence to a final rescaled soliton in large time;
    \item finite time exit from the soliton neighborhood.
\end{enumerate}
The third case is not well-understood at this stage: in particular,
the long time behaviour of the solution 
after its exiting the neighbourhood of the soliton family remains an open question.
From \cite{MMPI}, it is also known that the first and third cases (blow-up and exit) 
are in a certain sense stable.
Moreover, in the work \cite{Nakanishi-16}, Martel, Merle, Nakanishi and Rapha\"el 
prove that those two stable regimes are separated by a $\mathcal{C}^1$ codimension-one manifold on which the unstable scenario occurs (convergence to a rescaled soliton). 

Because of the stability (in a certain sense) of the $(T-t)^{-1}$ blow-up,
the blow-up behaviors corresponding to $\|\partial_x U(t)\|_{L^2}\sim (T-t)^{-\nu}$ for
$\nu\neq 1$, or any other blow-up behavior, will be called "non generic" or
"exotic".
So far, for the mass critical gKdV equation, examples of non-generic blow-up rates have been constructed in the work \cite{MMPIII} of Martel, Merle, Rapha\"el. 
In particular, they have obtained finite time blow-up solutions for the continuum of rates $\nu>\frac{11}{13}$. The value $\frac{11}{13}$ is a technical threshold and it was not clear from the proof in \cite{MMPII} which other rates $\nu$ should be possible among
$\frac 13 \leq \nu \leq \frac {11}{13}$.
Recall that rates $0<\nu<\frac 13 $ are ruled out by a standard scaling argument and
the Cauchy theory. Moreover, from \cite{Martel-Merle-2002}, 
no blow-up occurs with $\nu =\frac13$ for solutions close to solitons.
Recall also that global solutions with grow-up in infinite time 
with exponential and all powers of polynomials blow-up rates
were also constructed in \cite{MMPIII}.

Since such behaviors do not enter the classification of \cite{MMPI}, the initial data
shall not satisfy the assumptions of \cite{MMPI}.
Actually, the exotic blow-up in \cite{MMPIII} are a consequence of an explicit
tail of the initial data. This tails produces a modification in the blow-up behavior
by acting as a nonlinear source term and modifying the ODE system of modulation parameters
(scale and translation). See \S \ref{S:strategy} for heuristics.

All the blow-up behaviors described above are obtained by concentration of a soliton
of the form
\[
U(t,x) \approx \frac{1}{\lambda^\frac12(t)}\, Q\left( \frac{x-\sigma(t)}{\lambda(t)}\right)
\]
where $\sigma(t)$ represents the position of the soliton and $\lambda(t)$ is related to its
scale. By a simple calculation, we have $\|\partial_x U(t)\|_{L^2}\approx \lambda^{-1}(t)$.
Moreover, it can be proved that $\sigma_t(t)\approx \lambda^{-2}(t)$,
which justifies why $\nu=\frac 12$ is critical in a
blow-up behavior of the form $\lambda(t)\approx (T-t)^\frac12$:
for $\nu>\frac 12$, blow-up occurs at an infinite point ($\lim_T \sigma=+\infty$), 
while for $\frac 13 <\nu<\frac 12$, blow-up occurs at a finite point 
($\lim_T \sigma<+\infty$).
In particular, all exotic blow-up in \cite{MMPIII} occur at an infinite point.
The two works \cite{Martel-Pilod-2024}, \cite{Martel-Pilod-2026} of Martel and Pilod provide the construction of the finite point blow-up for the mass critical gKdV equation with the blow up rate $\nu \in \{\frac 25\}\cup (\frac 37, \frac 12)$. The existence of finite point blow up in the $H^1$ space for the rates $\nu \in (\frac 13, \frac 37)\setminus \{\frac 25\}$ remains open.

\subsection{Main result}

In the present article,
we complete the study of the power finite-time blow-up
for the mass critical gKdV equation with infinite point blow-up,
by constructing initial data leading to blow-up with any rate $\nu \in (\frac 12, 1)$.
 
\begin{theorem}\label{Theorem_principal_result}
    Let $\nu \in (\frac 12, 1)$. For any $\delta >0$, there exist $T>0$ and $U_0 \in H^1(\RR)$ with $\|U_0 - Q\|_{H^1} \leq \eta$ such that the $H^1$ solution $U$ of \eqref{gKdV_principal_eq} with initial data $U_0$ exists on $[0,T)$ and blows up at the finite time $T$ with blow up rate $\nu$.
    
    Moreover, $U$ is of the form
    \begin{equation}
        U(t,x) = \frac{1}{\lambda^{\frac 12}(t)}Q\Big( \frac{x-\sigma(t)}{\lambda(t)} \Big) + \tilde{\vare}(t,x),
    \end{equation}
    with
    \begin{equation}\label{main_theorem_estimate_parameters}
        \lambda(t) \sim \frac{1}{(3\nu-1)^{\nu}}(T-t)^{\nu} \qquad \text{and}\qquad \sigma(t) \sim \frac{(3\nu-1)^{2\nu}}{2\nu-1}(T-t)^{1-2\nu} \quad \text{as}\,\, t\uparrow T.
    \end{equation}
    The remainder term verifies
    \begin{equation}\label{main_theorem_estimate_remainder}
        \sup_{[0,T)}\|\tilde{\vare}(t)\|_{L^2}\leq 2\delta \qquad \text{and}\qquad \lim_{t\uparrow T} \lambda(t)\|\partial_x \tilde{\vare}(t)\|_{L^2} = 0.
    \end{equation}
\end{theorem}

Theorem \ref{Theorem_principal_result} combined with \cite[Theorem 1]{MMPIII}
shows the existence of infinite point blow-up solutions with all possible rates $\nu>\frac 12$. These solutions are expected to be unstable, as it will be explained later, 
even if we do not provide a proof of this fact here.
Concerning infinite point blow-up, a relevant remaining open question concerns the rate $\nu =\frac12$. Work is currently in progress on this scenario involving a logarithmically slow-traveling infinite point blow-up. This will appear elsewhere.

The overall strategy of the proof of Theorem \ref{Theorem_principal_result}
follows the one in \cite{MMPIII} with two key differences:
\begin{enumerate}
    \item The explicit tail is used to reach all values of $\nu$ 
    and becomes quite small (large negative powers of $x$) as $\nu\downarrow \frac 12$.
    It is thus more delicate to make it effective to adjust the blow-up rate.
    In particular, in contrast with \cite{MMPIII}, one has to work in a weighted space
    whose weight depends on the value of $\nu$.
    \item The articles \cite{MMPI,MMPII,MMPIII} introduce a mixed virial-energy functional to prove crucial local energy estimates on the solutions.
    Here, we implement a modification of this virial-energy functional which simplifies the computations, reduces the number of bootstrap assumptions and clarifies the proof.
\end{enumerate}
We refer to \S \ref{S:strategy} for a detailed sketch of the proof.

\subsection{Some related results for other nonlinear dispersive or wave models}

The $L^2$ critical Nonlinear Schr\"odinger equation (NLS) has many similarities with the critical (gKdV) equation. For the case of the 1D NLS equation, the ground state is the same as for (gKdV).
For higher dimensions, the ground state is strictly positive, smooth and has exponential decay 
at infinity.
The criterion on global existence due to Weinstein \cite{Weinstein-1983}
is the same as for (gKdV): any $H^1$ initial data with subcritical mass leads to
a unique global and bounded solution.
For (NLS), a minimal mass blow-up solution is known explicitly with the blow-up rate $(T-t)^{-1}$, using the pseudoconformal symmetry.
This minimal mass blow-up solution was proved to be unique by Merle \cite{Merle-1993-minmass-nls} (up to invariances).
Solutions with the pseudoconformal blow-up rate $(T-t)^{-1}$ were constructed in the work of Bourgain, Wang \cite{Bourgain-Wang-1998} and subsequently in the work of Krieger, Schlag \cite{Krieger-Schlag-2009}, by perturbing the minimal mass blow-up solution with strongly degenerate profiles.
The work of Landman, Papanicolaou, Sulem, Sulem \cite{LPSS-1988-first-blow-up-nls} has conjectured the so called "$\log\log$ law" of blow-up. A first rigorous construction of such blow-up is due to Perelman in \cite{Perelman-2001-nls-d1} in 1D.
A complete description of this regime for space dimension $d\leq 5$, and in particular its stability in $H^1$, was obtained in a series of works of Merle, Rapha\"el; see for example \cite{Merle-Raphael-profile-quant-2005, Merle-Raphael-crit-nls-2005}.
Such blow-up occurs at a finite point, for some $H^1$ open set of slightly supercritical initial data, with speed $(T-t)^{-\frac12}\sqrt{\log{|\log{(T-t)}|}}$.  
The proof of this result relies on a coercivity property of a certain operator related to the ground state proved in \cite{Merle-Raphael-crit-nls-2005} for $d=1$. The lack of an explicit formula for $Q$ in higher dimensions complicates the proof of this property, but it has been verified numerically for $2\leq d\leq 4$ in \cite{Fibich-Merle-Raphael-spec-prop-2006} and for $5\leq d \leq 10$ in \cite{Yang-Roudenko-Zhao-blowup-dynamics-2018}.
The subject of exotic blow-up (other than "$\log\log$" and $\frac{1}{t}$) for (NLS) is still widely open.
See, however, \cite{Raphael-2005} for a negative result.

Recall that the dependence of the grow up rate on the tail of given initial data was previously observed for global in time growing up solutions of the energy critical harmonic heat flow in \cite{Gustafson-Nakanishi-Tsai-asympt-2008}. 

The blow-up for the energy critical problems has also attracted considerable attention. 
Krieger, Schlag and Tataru  \cite{Krieger-Schlag-Tataru-renorm-2008, Krieger-Schlag-Tataru-2009}
pioneered the construction of continuum of exotic blow-up in such a context,
both for the energy-critical wave map problem and the energy critical focusing wave equation in $d=3$. Particularly, they proved the existence of such solutions with concentration rate
$\lambda(t) \approx t^{-1-\nu}$ for any $\nu>\frac 12$. In \cite{Krieger-Schlag-2014} this condition was relaxed to $\nu>0$. Then, Donninger, Huang, Krieger and Schlag \cite{Donninger-Huang-Krieger-Schlag-2014} investigated the existence of other admissible rescaling functions for this type of blow-up. Furthermore, they constructed blow-up solutions with the scaling law $\lambda(t) = t^{-1-v(t)}$ where $v(t)$ is eternally oscillating and asymptotically constant. Such problems are also relevant for the mass critical (gKdV) equation.

Other results are known for the higher space dimensions.
Jendrej \cite{Jendrej-2017} constructed type-II blow-up solutions in the five space dimensions.
For the energy-critical one equivariant wave maps in the two sphere, Jendrej, Lawrie, Rodriguez in \cite{Jendrej-Lawrie-Rodriguez-2022} provide a sharp classification of the blow up rate for solutions constructed by prescribing the radiation. In the recent work \cite{Kim-Kihun-et-al-2024}, Kim, Kwon and Oh adapt the their method to construct finite-energy blow-up solutions for the radial self-dual Chern–Simons–Schrödinger equation with a continuum of blow-up rates.


\subsection{Notation}

Throughout the text, $X\lesssim Y$ is a shorthand for the inequality $X \leq C\,Y $ for some time-space independent constant $C$. In general manner, the constants are time-space independent, unless otherwise stated.

For a given small positive constant $0<\alpha^* \ll 1$, we denote $\delta$ a small parameter depending on $\alpha^*$ such that
\begin{equation}
    \delta(\alpha^*) \to 0 \quad \text{as}\quad \alpha^* \to 0.
\end{equation}

For $1\leq p\leq +\infty$, $L^p(\RR)$ denotes the classical real-valued Lebesgue space. For $f,g \in L^2(\RR)$, we denote their scalar product
\begin{equation}
    (f,g) = \int_{\RR} f(y)\,g(y)\, dy.
\end{equation}

We define the weighted Lebesgue spaces $L^2(\RR; e^{-\frac{|y|}{10}}dy)$ and $L^2(\RR; e^{\frac{y}{B}}dy)$, for $B>100$, equipped respectively with the norms
\begin{equation}
    \|f\|_{L^2_{sol}} = \Big( \int f^2 e^{-\frac{|y|}{10}}dy\Big)^{\frac{1}{2}} \qquad \text{and} \qquad \|f\|_{L^2_B}= \Big( \int f^2(y) e^{\frac{y}{B}} dy \Big)^{\frac{1}{2}}.
\end{equation}

In the following, for simplicity of notation, we omit $dy$ and the integration domain is $\RR$, unless otherwise stated.

In the space $L^2$, the generator of scaling symmetry is given by
\begin{equation}
    \Lambda f = \frac{1}{2}f + yf'
\end{equation}
and the linearized operator $\mathcal{L}$ around the ground state by 
\begin{equation}
    \mathcal{L}f = -f''+f-5Q^4 f.
\end{equation}

\subsection{Strategy of the proof}\label{S:strategy}
The proof proceeds in four main steps, detailed in Sections \ref{S:2} to \ref{section_construction}.
\begin{enumerate}
    \item \textit{Definition and role of the tail}\\
    Given $c_0 < 0$ and $\theta > 1$, we consider a smooth $L^2$ function $f_0$ with power decay of order $\theta$ on the right
    \begin{equation}
        f_0(x) = c_0 x^{-\theta} \qquad \text{for}\quad x_0 \gg1.
    \end{equation}
 Then, we define $f$ the solution of \eqref{gKdV_principal_eq} evolving from the initial data $f_0$. The evolution of such decaying tails is studied by local energy estimates for (gKdV). 
 We know that $f$ is global and that the tail of $f_0$ is preserved for $x \gtrsim t$. 
 The aim of this tail will be to modify in the next step the behaviour of the coupled system of modulation equations of the parameters of a blowing up soliton, by nonlinear interaction
 of the tail with the soliton.
 
    \item \textit{Modulation around a refined profile and system of parameters}\\
    We look for a solution $U$ of \eqref{gKdV_principal_eq} of the following form
    \[
    U(t,x) = v(t,x)+f(t,x),
    \]
   with
    \[
    v(t,x)=\lambda^{-\frac 12}(s)\big( Q_{b(s)}(y) + f(t(s),\sigma(s))R(y) + \vare(s,y)\big)
    \]
    and where $(s,y)$ are the rescaled time and space variables defined by
    \begin{equation}
        \frac{ds}{dt} = \frac{1}{\lambda^3},\quad y = \frac{x-\sigma(s)}{\lambda(s)}.
    \end{equation}
    Note that as $t\in [0,T)$, $s$ will be defined in a certain interval $[s_0,+\infty)$
    and $y\in \RR$.
    In the formula above for $v$, $Q_b$ is a refined soliton profile (the parameter $b(s)$ takes into account the variation of the scaling parameter $\lambda(s)$), the function $R(y)$ is a specific rapidly decaying function and $\vare$ is a remainder term to be estimated by the energy method in the next step. 
    Note that the term $f(t,\sigma) R(y)$ aims at compensating the first order of the nonlinear interaction between the soliton and the tail $f(t,x)$ added to the solution $U$. 
    This interaction is relatively weak, since the soliton is concentrated (this means that $\lambda(s)\ll 1$) and the position of the soliton $\sigma(s)$ is large, which means  that $f(t,\sigma)$ is small. 
    However, this interaction is sufficient to
    modify the blow-up behavior under certain conditions as we will formally justify now.
    
    The equation of $\vare$ and the parameters $(\lambda,\sigma,b)$ are given in \eqref{equation_of_eps_s}.
 From this equation and orthogonality conditions prescribed on $\vare$, we obtain 
 the dynamical system of the parameters $(\lambda,\sigma,b)$. 
 Formally (\emph{i.e.} neglecting error terms), we obtain
    \begin{equation}\label{formal_dynamical_system}
        \frac{\lambda_s}{\lambda}+b = 0, \qquad \sigma_s = \lambda, \qquad \frac{d}{ds}\Big( \frac{b}{\lambda^2} + \frac{4}{\int Q}\,c_0 \,\lambda^{-\frac{3}{2}}\sigma^{-\theta} \Big) = 0.
    \end{equation}
    The equations for $\lambda$ and $\sigma$ are natural in view of equation \eqref{equation_of_eps_s}. The equation of $b$ is more subtle and it will be justified later.
    Let us say for the moment that the term containing $\sigma^{-\theta}$ comes from the nonlinear interaction of the tail $f$ with the soliton. In practice, this term is computed thanks to the addition of $f(t,\sigma) R(y)$ in the formula for $v(t,x)$ above.
    The precise computations that justify this dynamical system can be found in Lemma \ref{lemma_with_all_estimates_on_r_F_etc} - Lemma \ref{lemma_with_estimates_on_hs_and_gs}.\\
    In what follows, we perform formal computations related to \eqref{formal_dynamical_system}.\\
    Integrating the last equation on some time interval $[s_0,s]$, we have 
    \begin{equation}\label{equation_with_l0}
        \frac{b(s)}{\lambda^2(s)} + \frac{4\,c_0}{\int Q}\,\lambda^{-\frac 32}(s)\,\sigma^{-\theta}(s) = l_0 .
    \end{equation}
    In the regime considered, we will check that $\frac{b(s)}{\lambda^2(s)}\to +\infty$, which
    means that the value of the integration constant $l_0$ does not matter.
    For now on, we fix $l_0 = 0$ to simplify.
    Replacing $b$ by $-\frac{\lambda_s}{\lambda}$ and then $\lambda$ by $\sigma_s$ (from \eqref{formal_dynamical_system}) yields
    \begin{equation}
        \lambda^{-\frac 12}\lambda_s  = \frac{4\,c_0}{\int Q}\, \lambda\,\sigma^{-\theta} = \frac{4\,c_0}{\int Q}\,\sigma_s\,\sigma^{-\theta}.
    \end{equation}
    We expect $\lambda(s)\to 0^+$ as $s\to +\infty$ and thus $\lambda_s\leq 0$, thus the first equality leads us to choose $c_0<0$.\\
    After integration, we get 
    \begin{equation}\label{l1}
        \lambda^{\frac 12}(s) + \frac{2\,c_0}{(\theta -1)\int Q}\,\sigma^{1-\theta}(s) = l_1.
    \end{equation}
   Since $\lambda(s)\to 0$ and $\sigma(s)\to +\infty$ as $s\to +\infty$, one has to impose
   $l_1=0$.
    Combining with the second equation in \eqref{formal_dynamical_system}, and choosing at
    $s=s_0$ the initial value $\sigma^{2\theta -1}(s_0) = (2\theta-1)\Big( \frac{2\,c_0}{(\theta -1 )\int Q}\Big)^2s
    _0$, we get 
    \begin{equation}
        \sigma_s = \Big( \frac{2\,c_0}{(\theta -1 )\int Q}\Big)^2\,\sigma^{2(1-\theta)} \qquad \iff \qquad \sigma^{2\theta -1}(s) = (2\theta-1)\Big( \frac{2\,c_0}{(\theta -1 )\int Q}\Big)^2 s.
    \end{equation}
    Inserted in the previous relations, this gives
    \begin{equation}
        \lambda(s) = (2\theta - 1)^{-\frac{2(\theta-1)}{2\theta-1}}\Big( \frac{2\,c_0}{(\theta-1)\int Q} \Big)^{\frac{2}{2\theta-1}}\,s^{-\frac{2(\theta-1)}{2\theta-1}}
    \end{equation}
    With the notation
    \begin{equation}
        \beta = \frac{2(\theta-1)}{2\theta-1}, \qquad \theta = \frac{1-\frac{\beta}{2}}{1-\beta}, \qquad c_0 = -\frac{\int Q}{2}(\theta-1)(2\theta-1)^{\theta-1} < 0,
    \end{equation}
    we have obtained
    \begin{equation}\label{formal_law_of_parameters}
        \lambda(s) = s^{-\beta}, \qquad \sigma(s)= \frac{1}{1-\beta}s^{1-\beta},\qquad b(s)= \frac{\beta}{s}.
    \end{equation}
   Reciprocally, one easily checks that \eqref{formal_law_of_parameters} are solutions of the dynamical system \eqref{formal_dynamical_system}. Moreover, substituting \eqref{formal_law_of_parameters} into \eqref{equation_with_l0} indeed justifies that the constant $l_0$ could be neglected. In contrast, if $l_1$ is chosen positive,
    one formally obtains from \eqref{l1} that the solution will converge to a fixed rescaled soliton ($\lambda\to l_1^2$).
    In this sense, the direction $l_1 = 0$ presents an instability of the exotic blow-up.
    This observation is formal and we will not pursue this issue here.
    However, it justifies the use of a topological argument to conclude the bootstrap argument in the proof of the Theorem \ref{Theorem_principal_result}.
    
    \item \textit{Energy estimates}\\
    The first two steps are similar to the work of Martel, Merle, Rapha\"el \cite{MMPIII}.
    In the third step, we introduce a modification of the virial-energy functional of the form
    \begin{equation}\label{def:h}
        \mathcal{H}(s) = s^{j}\mathcal{F}(s) + \lambda^k(s)\int \vphi_k\,\vare^2(s),
    \end{equation}
    which simplifies and extends the work in \cite{MMPI,MMPII,MMPIII}. 
    The term $\mathcal{F}$ represents a mixed virial-energy functional at the level of the equation of $\vare$. In the regime \eqref{formal_law_of_parameters}, $\mathcal{F}$ is coercive and its control allows to bound locally the remainder $\vare$.
    
    Compared to \cite{MMPIII}, the simplification comes from the fact that adding the term 
    \begin{equation*}
        \lambda^k(s)\int \vphi_k\,\vare^2(s) \qquad \vphi_k(y) = y^{k},\, y \gg 1
    \end{equation*}
    directly to the functional $\mathcal{H}$
    allows to control the delicate scaling terms that appear when differentiating $s^{j}\mathcal{F}$. In \cite{MMPIII}, the control of such terms is based on a bootstrap 
    estimate. Note that this term is at the scaling level, which means that coming back to the original variables $(t,x)$, no scaling term in $\lambda$ would appear.
    
    The functional in \eqref{def:h} also extends \cite{MMPIII} thanks to the introduction of the parameter $k$, which depends on the blow-up rate.
    Choosing a fixed $k$, as was done in \cite{MMPIII}, does not allow to reach all the values of $\nu \to \frac 12^{+}$. In the paper \cite{MMPIII}, the authors considered the interval $\frac 54< \theta < \frac{29}{18}$, which corresponds to $\frac 13<\beta< \frac{11}{20}$ and thus any blow-up rate $\nu > \frac{11}{13}$. In our work, we consider the interval $\theta>\frac 32$, therefore $\frac 12<\beta< 1$ and thus $\frac 12< \nu < 1$.
    In particular, $k \to +\infty$ as $\nu \to \frac 12^{+}$ (equivalently $\theta\to+\infty$).
    It might appear a bit counter-intuitive that a stronger modification of blow-up laws
    (as $\nu$ approaches $\frac 12$, it differs more from the stable rate $\nu=1$)
    is the consequence of a more decaying tail $\theta\to +\infty$ since the influence of the tail should be weaker.
    
    \item \textit{Construction of blowing up solutions}\\
    For well-chosen initial data arbitrarily close to the soliton in a given topology, we construct solutions 
    in the regime close to \eqref{formal_dynamical_system}, 
    using a bootstrap argument and energy estimates as described above.
    See Proposition \ref{Prop_on_blow_up} for a precise statement.\\
    Now, we justify formally that the behaviour of the parameters in \eqref{formal_law_of_parameters},  leads to the blow-up scenarios in the Theorem \ref{Theorem_principal_result} in the original time $t$.\\
    By the definition of the rescaled time and since $\beta > \frac 13$, we get that 
    \begin{equation}
        \int^{+\infty}_{s_0} \lambda^3(s)\,ds = T < +\infty.
    \end{equation}
    Therefore, the solution $U$ blows up in finite time.\\
    The law of the rescaled variable yields the following  behaviour of 
    the rescaled time $s$ in terms of $t$
    \begin{equation}
        s(t) \sim \big[(3\beta-1) \,(T-t)\big]^{-\frac{1}{3\beta-1}}  \qquad \text{for}\,\, t \to T^{-}.
    \end{equation}
    Therefore, 
    \begin{equation}
        \lambda(s(t)) \sim (T-t)^{\nu} \qquad \text{with }\,\nu = \frac{\beta}{3\beta-1}\in (\frac 12, 1).
    \end{equation}
    Combining with the consequence of the decomposition around the approximate blow-up profile, it holds for some $c>0$
    \begin{equation}
        \|\partial_x U(t)\|_{L^2_x} \sim \lambda^{-1}(s(t)) \|Q'\|_{L^2} \sim c\,(T-t)^{-\nu} .
    \end{equation}
    Those computations are detailed in the proof of the Theorem \ref{Theorem_principal_result} in  Section \ref{section_construction}.
\end{enumerate}

\vspace{0.4cm}
\textbf{Acknowledgments: }\textit{The author would like to express her gratitude to Didier Pilod for his active interest in the publication of this paper and for several helpful conversations.
Further, the author would like to thank Yvan Martel for helpful discussions and encouragements.}

\section{Definition of the tail on the right-hand side}\label{S:2}
In this section we introduce a function with an explicit power-like decay on the right 
and we recall from \cite{MMPIII} a property of persistence of such tails
for the critical gKdV equation in a suitable space-time region.

For $c_0<0$, $\theta >1$ and $x_0 \gg 1$, we consider a function
$f_0 \in C^{\infty}(\RR)\cap L^2(\RR)$ such that
\begin{equation}\label{estimaate_on_f0_derivative_on_x}
    f_0(x) := \begin{cases}
        0 \quad \text{ for } x < \frac{x_0}{4}\\
        c_0 x^{-\theta} \quad \text{ for } x > \frac{x_0}{2}
    \end{cases}
    \qquad \text{and} \qquad \Big\lvert \frac{d^k}{dx^k}f_0(x)\Big\rvert \lesssim  |c_0|\lvert x \rvert^{-\theta-k} \qquad \forall (x,k) \in \RR\times\NN.
\end{equation}
We remark that 
\begin{equation}\label{estimate_on_f0_norm_L2}
    \|f_0\|_{L^2(\RR)} \approx |c_0| x_0^{-\theta+\frac{1}{2}},\quad
    \|f_0\|_{H^1(\RR)} \approx |c_0| x_0^{-\theta+\frac{1}{2}}.
\end{equation}

Let $f$ be the solution of
\begin{equation}\label{equation_of_f}
    \begin{cases}
        \partial_t f + \partial_x(\partial_{xx}f+ f^5) = 0,\\
        f(0,x) = f_0(x).
    \end{cases}
\end{equation}

\begin{lemma}[\cite{MMPIII}]\label{lemma_on_the_decaying_tail}
    For $x_0$ large enough, the solution $f$ of \eqref{equation_of_f} is global, smooth and bounded in $H^1$.
    Moreover, $f\in C(\RR,H^s(\RR))$ for $s\geq 0$ and $\|f\|_{L^{\infty}_{t}H^s_x} \lesssim \delta(x_0^{-1})$.
    
    In addition, it holds for all $t\geq0$ and $x > \frac{t}{2}+\frac{x_0}{2}$,
    \begin{equation}
        \forall k \in \NN\cup\{0\}, \qquad \lvert \partial^k_x f(t,x) - f_0^{(k)}(x)\rvert \lesssim x^{-\theta-k-2}.
    \end{equation}
    \begin{equation}
        \lvert \partial_t f(t,x) \rvert \lesssim x^{-\theta-3}.
    \end{equation}
\end{lemma}
\begin{proof}
    The first part of the Lemma follows from the general result in \cite[Theorem 2.8]{KPV93} and \cite[Corollary 2.9]{KPV93} or in \cite[Theorem 7.2]{Linares-Ponce}.  The asymptotic of the decaying tail is proved in \cite[Lemma 2.3]{MMPIII}.
\end{proof}
\begin{remark}
As in \cite{MMPIII}, we will use $f(t,x)$ instead of $f_0(x)$ to insert in 
the approximate solution. Indeed, $f$ being a solution of the nonlinear equation \eqref{gKdV_principal_eq}, it will provide better estimates than
simply taking $f_0$, at the cost of some error terms controlled in Lemma \ref{lemma_on_the_decaying_tail} above.
\end{remark}

\section{Decomposition around the soliton}\label{section_decomposition} 
\subsection{Structure of the linearized operator}
Here we recall some properties of the linearized operator $\mathcal{L}$ around the soliton. The operator is defined as follows
\begin{equation}
    \mathcal{L}\phi = -\phi'' +\phi - 5Q^4 \phi.
\end{equation}
We introduce the following function space
\begin{equation}
    \mathcal{Y} := \{\phi \in C^{\infty}(\RR,\RR) \,|\, \forall k \in \NN, \exists C_k, r_k>0 \;\text{ s.t. }\; \lvert \phi^{(k)}(y)\rvert\leq C_k(1+|y|)^{r_k}e^{-|y|}, \forall y \in \RR \}.
\end{equation}
We recall the following standard result, (see for instance \cite{Weinstein-85} and \cite{Martel-Merle-01}).
\begin{lemma}[Properties of the linearised operator $\mathcal{L}$]
The operator $\mathcal{L}: H^2(\RR) \subset L^2(\RR)\rightarrow L^2(\RR)$ satisfies
\begin{enumerate}
    \item (Spectrum) The operator $\mathcal{L}$ has only one negative eigenvalue, $\mathcal{L}Q^3 = -8 Q^3$. Moreover, $\ker{\mathcal{L}} = \{aQ' : a \in \RR \}$ and $\sigma_{ess}(\mathcal{L})=[1,+\infty)$.
    \item (Scaling) $\mathcal{L}\Lambda Q = -2Q$ and $(Q,\Lambda Q)=0 $.
    \item (Coercivity) For any $\phi \in H^1$ holds
    \begin{equation}
        (\phi, Q^3) = (\phi, Q') = 0  \quad\Rightarrow\quad (\mathcal{L}\phi,\phi)\geq \|\phi\|^2_{L^2}.
    \end{equation}
    Moreover, there exists $\nu_0>0$ such that for any $\phi\in H^1$
    \begin{equation}
        (\mathcal{L}\phi,\phi)\geq \nu_0\|\phi\|^2_{H^1}- \frac{1}{\nu_0}[(\phi,Q)^2 + (\phi, y\Lambda Q)^2 + (\phi, \Lambda Q)^2] .
    \end{equation}
    \item (Invertibility) There exists a unique function $R \in \mathcal{Y}$, even and such that 
    \begin{equation}\label{properties_of_R}
        \mathcal{L}R = 5Q^4, \qquad\qquad (Q,R) = -\frac{3}{4}\int Q.
    \end{equation}
    \item (Invertibility (bis))
    There exists a unique function $P \in C^{\infty}(\RR)\cap L^{\infty}(\RR)$ such that $P' \in \mathcal{Y}$ and 
    \begin{equation}
        (\mathcal{L}P)' = \Lambda Q, \qquad  \lim_{y \to -\infty}P(y) = \frac{1}{2}\int Q, \qquad \lim_{y\to+\infty}P(y)=0.
    \end{equation}
    Moreover,
    \begin{equation}\label{relation_on_scalar_prod_P_and_Q}
        (P,Q) = \frac{1}{16}\Big(\int Q \Big)^2>0,  \qquad (P,Q') = 0.
    \end{equation}
\end{enumerate}
\end{lemma}

\begin{remark}
    By $\lim_{+\infty}P=0$ and since $P' \in \mathcal{Y}$, it holds
    \begin{equation}\label{property_on_right_on_P}
       | P(y)| \lesssim e^{-\frac{|y|}{2}}, \quad y>0.
    \end{equation}
\end{remark}

\subsection{Definitions and estimates of localized profiles}
We recall the definition of a one-parameter family of approximate self-similar profiles $b \mapsto Q_b$ with $|b|\ll 1$ which will provide the higher-order deformation of the ground state profile $Q_{b=0}$ in the blow-up setting. 

Introduce a function $\chi \in C^{\infty}(\RR)$ such that
\begin{equation}
    0\leq \chi\leq 1, \quad 0\leq (\chi'')^2 \lesssim \chi' \quad \text{on }\RR, \qquad \chi_{|(-\infty,-2)} \equiv 0 \quad \text{and} \quad \chi_{|(-1,+\infty)} \equiv 1.
\end{equation}

\begin{definition}[Localised profile]
    Let $\gamma = \frac{3}{4}$, the localised profile $Q_b$ is defined by
    \begin{equation}
        Q_b(y) = Q(y) + bP_b(y)
    \end{equation}
    with 
    \begin{equation}\label{def_of_P_b}
        P_b(y) = P(y) \chi_b(y) \quad \text{and} \quad \chi_b(y) = \chi(|b|^{\gamma}y).
    \end{equation}
\end{definition}

We state the properties of $Q_b$ in the following lemma.
\begin{lemma}[Approximate self-similar profiles $Q_b$]
    There exists $b^*>0$ small enough such that for all $|b|<b^*$, the following properties hold:
    \begin{enumerate}
        \item (Estimates on $Q_b$) For all $y\in \RR$,
        \begin{equation}\label{estimate_on_Q_b_and_its_derivs}
        \begin{split}
            \lvert Q_b(y)\rvert \lesssim e^{-|y|}+|b|\Big( \mathds{1}_{[-2,0]}(|b|^{\gamma}y)+e^{-\frac{|y|}{2}}\Big),\\
            \lvert Q^{(k)}_b(y)\rvert \lesssim e^{-|y|}+|b|e^{-\frac{|y|}{2}} + |b|^{1+k\gamma}\mathds{1}_{[-2,-1]}(|b|^{\gamma}y), \quad \forall k \geq 1.
        \end{split}
        \end{equation}

        \item (Equation of $Q_b$) Let the error term be defined by
        \begin{equation}
            -\Psi_b = \big( Q''_b-Q_b+Q^5_b \big)' + b\Lambda Q_b - 2 b^2 \frac{\partial Q_b}{\partial b}.
        \end{equation}
        Then, for all $y\in \RR$,
        \begin{equation}
            \lvert \Psi_b(y) \rvert \lesssim |b|^{1+\gamma} \mathds{1}_{[-2,-1]}(|b|^{\gamma}y)+b^2 \Big( e^{-\frac{|y|}{2}} + \mathds{1}_{[-2,0]}(|b|^{\gamma} y) \Big),
        \end{equation}
        \begin{equation}
            \lvert \Psi^{(k)}_b(y) \rvert \lesssim |b|^{1+(k+1)\gamma}\mathds{1}_{[-2,-1]}(|b|^{\gamma}y) + b^2 e^{-\frac{|y|}{2}}, \qquad \forall k \geq 1.
        \end{equation}

        Moreover,
        \begin{equation}
            \lvert (\Psi_b,\phi)\rvert \lesssim b^2, \qquad \forall \phi \in \mathcal{Y} \qquad\text{and}\qquad \| \Psi^{(k)}_b (y) \|_{L^2_B} \lesssim C_B b^2, \qquad \forall k \geq 0.
        \end{equation}

        \item (Projection of $\Psi_b$ in the direction $Q$)
        \begin{equation}
            \lvert (\Psi_b, Q) \rvert \lesssim |b|^3.
        \end{equation}

        \item (Mass and energy properties of $Q_b$)
        \begin{equation}\label{estimate_mass_of_Qb}
            \bigg\lvert \int Q^2_b - \Big( \int Q^2 + 2b \int P Q \Big) \bigg\rvert \lesssim |b|^{2-\gamma},
        \end{equation}
        \begin{equation}\label{estimate_energy_of_Qb}
            \bigg\lvert E(Q_b) + b \int P Q \bigg\rvert \lesssim b^2.
        \end{equation}
    \end{enumerate}
\end{lemma}
\begin{proof}
The proofs of the results above on the approximate profiles
can be found in \cite[Lemma 2.4]{MMPI} and in \cite[Lemma 3.3]{Martel-Pilod}.
\end{proof}

\subsection{Definition of the approximate solution}
Let $\theta> 1$, $f$ be the solution of \eqref{equation_of_f},  $U$ be the solution of \eqref{gKdV_principal_eq} and set
\begin{equation}\label{Umoinsf}
    v(t,x) = U(t,x)-f(t,x).
\end{equation}

\vspace{0.3cm}
Given $C^1(\RR)$ functions of time $\lambda>0$, $\sigma\in\RR$,
we introduce the rescaled time variable for some $s_0 >1$
\begin{equation}\label{def_of_rescaled_time}
    s= s(t) = s_0 + \int^t_{0}\frac{d s'}{\lambda^3(s')} \quad \text{for} \quad t\in [0,T).
\end{equation}
Consider the inverse function $\tau(s(t)) = t$. We have
\begin{equation}\label{relation_on_diff_of_tau_and_zeta}
    \frac{d s}{dt}(t) = \frac{1}{\lambda^3(t)} \qquad \text{and}\qquad \frac{d \tau}{ds}(s) = \lambda^3(\tau(s)).
\end{equation}
From now on, any time-dependent function can be seen as a function of $s\in [s_0,+\infty)$ or of $t \in [0,T)$, depending on the context. 
The function $f$ was already defined in $t$, its variant in $s$ is then $f(\tau(s),x)$. The parameters $\lambda,\sigma$ will usually be defined in the variable $s$.

We consider the problem \eqref{gKdV_principal_eq}-\eqref{Umoinsf} in rescaled variables, setting
\begin{equation}
    V(s,y):= \lambda^{\frac{1}{2}}(s)v(\tau(s),x), \qquad F(s,y):= \lambda^{\frac{1}{2}}(s)f(\tau(s),x), \qquad F_0(s, y):= \lambda^{\frac{1}{2}}(s)f_0(x),
\end{equation}
with 
\begin{equation}
    y = \frac{x-\sigma(s)}{\lambda(s)}.
\end{equation}
and the functions $\lambda>0$,  $\sigma\in\RR$ defined above.
We define the following quantity
\begin{equation}\label{def_of_error_term_mathcal_E}
        \mathcal{E}(V) = V_s+\partial_y \big( \partial^2_y V -V + (V+F)^5 - F^5 \big) -\frac{\lambda_s}{\lambda}\Lambda V - \Big( \frac{\sigma_s}{\lambda}-1 \Big)\partial_yV.
    \end{equation}
Then,   $U$ is a solution of \eqref{gKdV_principal_eq} if and only if $\mathcal{E}(V) =0 $.
We now concentrate on constructing an approximate solution $W$ of the equation
$\mathcal{E}(V)=0$ of the following form 
\begin{equation}\label{def_of_W_and_of_r}
    W(s,y) = Q_{b(s)}(y)+r(s)R(y), \qquad \text{with} \qquad r(s):=F(s,0)= \lambda^{\frac{1}{2}}(s)f(\tau(s),\sigma(s))
\end{equation}
where $b$ is a $C^1(\RR)$ function of $s$ to determine.
Fix a certain $\rho$ such that 
\begin{equation}\label{condition_on_rho}
0<\rho < \frac{1}{4} \min{\{ 2\beta-1,\, \frac{1}{4}\, ,2-2\beta \}}.
\end{equation}%
We work under the following assumptions on the parameters, 
chosen in view of the formal asymptotics~\eqref{formal_law_of_parameters}
\begin{equation}\tag{BS1}\label{BS1}
    \begin{cases}
        \lvert \sigma(s) - \frac{1}{1-\beta} s^{1-\beta}\rvert \;\leq\; s^{1-\beta-\rho} \,, \\
        \lvert \lambda(s) - s^{-\beta} \rvert \;\leq\; s^{-\beta-\rho} \,,\\
        \lvert b(s) - \beta s^{-1} \rvert \;\leq\; s^{-1-\rho} \,,
    \end{cases}
\end{equation}

We introduce the auxiliary functions (related to the heuristics \eqref{equation_with_l0}
and \eqref{l1})
\begin{equation}\label{definition-g-h}
    g(s)= \frac{b(s)}{\lambda^2(s)}+\frac{4}{\int Q}c_0\lambda^{-\frac{3}{2}}(s)\sigma^{-\theta}(s) \qquad\text{and}\qquad h(s) = \lambda^{\frac{1}{2}}(s) + \frac{1}{\int Q}\frac{2c_0}{\theta-1}\sigma^{1-\theta}(s).
\end{equation}

We also define the following quantities
\begin{equation}\label{def_of_vec_m_and_vecM}
    \vec{m}= \begin{pmatrix}

        \frac{\lambda_s}{\lambda}+b\\
        \frac{\sigma_s}{\lambda}-1
    \end{pmatrix} \qquad \text{and} \qquad \vec{M}= \begin{pmatrix}
        \Lambda\\
        \partial_y
    \end{pmatrix}.
\end{equation}

Some technical estimates
\begin{lemma}\label{lemma_with_all_estimates_on_r_F_etc}
    Under the assumptions \eqref{BS1},
    \begin{equation}\label{assumption_on_sigma_on_t}
    \sigma(s)> \frac{2}{3}\tau(s)+\frac{2}{3}x_0 \quad \text{for all } \,\, s\in [s_0,+\infty)
    \end{equation}
    and for $s_0$ large enough, the following estimates hold for $s \geq s_0 $
    \begin{equation}\label{estimates_on_r}
        \big\lvert r(s) - c_0 \lambda^{\frac{1}{2}}(s) \,\sigma^{-\theta}(s) \big\rvert \lesssim \lambda^{\frac{1}{2}}\sigma^{-\theta-2}\lesssim s^{-3+2\beta},\qquad  |r(s)|\lesssim s^{-1},
    \end{equation}
    \begin{equation}\label{estimate_on_r_s}
         \Big\lvert r_s - c_0\lambda^{\frac{1}{2}}\sigma^{-\theta}\Big(\frac{1}{2}\frac{\lambda_s}{\lambda}- \theta \frac{\sigma_s}{\sigma} \Big)\Big\rvert \lesssim_{\theta} s^{-4+2\beta} + s^{-3+2\beta}|\vec{m}|,
    \end{equation}
    \begin{equation}
        |r_s| \lesssim_{\theta} s^{-2} + s^{-1}|\vec{m}|,
    \end{equation}
    \begin{equation}\label{estimate_on_r_minus_F}
        e^{-\frac{3|y|}{4}}\big\lvert r(s)-F(s,y) \big\rvert \lesssim s^{-2}e^{-\frac{|y|}{4}}, \qquad \forall y \in \RR,
    \end{equation}
    \begin{equation}\label{estimates_on_norms_L2_of_F}
        \|F(s)\|_{L^2_y} \lesssim x^{-\theta + \frac{1}{2}}_0,
        \qquad \|\partial_y F(s) \|_{L^2_y} \lesssim s^{-\beta} x^{-\theta - \frac{1}{2}}_0,
    \end{equation}
    \begin{equation}\label{estimates_on_norms_L_infty_of_F}
        \|F(s) \|_{L^{\infty}_{y}} \lesssim s^{-\frac{\beta}{2}} x^{-\theta}_0, \qquad \|F(s)\|_{L^{\infty}(y>-2|b|^{-\gamma})} \lesssim s^{-1}, \qquad \|\partial_y F\|_{L^{\infty}(y>-2|b|^{-\gamma})} \lesssim s^{-2},
    \end{equation}
    \begin{equation}\label{estimate_on_L_infty_of_dj_F}
        \| \partial^k_y \big( F(s,y) \big) \|_{L^{\infty}_y} \lesssim s^{-\beta(k+\frac{1}{2})}, \qquad \text{for  } k\in \NN\cup\{ 0\},
    \end{equation}
    \begin{equation}\label{estimates_on_norms_L_infty_of_F_with_exp}
        \sup_{y\in \RR}\Big\{e^{-\frac{|y|}{10}}\lvert \partial^k_y F(s,y) \rvert \Big\} \lesssim s^{-(k+1)}, \qquad \text{for  } k \in \NN \cup \{0\},
    \end{equation}
    \begin{equation}\label{estimates_on_scalar_prod_r_minus_F_and_Q}
        \bigg\lvert \big( \partial_y(5 Q^4(r-F)), Q \big) - c_0 \theta \Big( \int Q\Big) \lambda^{\frac{3}{2}} \sigma^{-\theta -1}\bigg\rvert \lesssim s^{-4+2\beta}.
    \end{equation}
\end{lemma}
\begin{proof}
    The first estimate in \eqref{estimates_on_r} comes from Lemma \ref{lemma_on_the_decaying_tail}, the assumption \eqref{assumption_on_sigma_on_t} on $\sigma(s)$ and \eqref{BS1}. Write 
    \begin{equation}
        \lvert r(s)-c_0 \lambda^{\frac 12}(s)\sigma^{-\theta}(s) \rvert = \lvert \lambda^{\frac{1}{2}}(s) \rvert \lvert f(\tau(s),\sigma(s))-c
        _0 \sigma^{-\theta}\rvert \lesssim \lambda^{\frac{1}{2}}(s) \sigma^{-\theta-2}(s) \lesssim s^{-3+2\beta}.
    \end{equation}
    The second estimate comes from the previous one combined with \eqref{BS1}. 

\vspace{0.4cm}
    In order to prove \eqref{estimate_on_r_s}, we prove first the following estimate
    \begin{equation}\label{estimate_on_r_s_another}
        \Big\lvert r_s - \frac{1}{2}\frac{\lambda_s}{\lambda}r + \theta\frac{\sigma_s}{\sigma}r \Big\lvert \lesssim (1+\theta)s^{-2}+(1+\theta)s^{-2}\Big\lvert \frac{\sigma_s}{\lambda}-1 \Big\rvert.
    \end{equation}
    Write 
    \begin{equation}
        r_s - \frac{1}{2}\frac{\lambda_s}{\lambda}r = \lambda^{\frac{7}{2}}(\partial_s f)(\tau(s),\sigma(s))+\lambda^{\frac{1}{2}}\sigma_s (\partial_y f)(\tau(s), \sigma(s)).
    \end{equation}
    Using the relation \eqref{relation_on_diff_of_tau_and_zeta} and  Lemma \ref{lemma_on_the_decaying_tail}, we get
    \begin{equation}
         \lvert \lambda^{\frac{1}{2}}\partial_s f(\tau(s),\sigma(s)) \rvert = \lvert \lambda^{\frac{7}{2}}\partial_1f(\tau(s),\sigma(s)) \rvert \lesssim \lambda^{\frac{7}{2}}\sigma^{-\theta-3}.
    \end{equation}
    Using \eqref{assumption_on_sigma_on_t}, Lemma \ref{lemma_on_the_decaying_tail}, applying the previous estimate \eqref{estimates_on_r} and also writing $|\sigma_s|\leq \lambda\big(\big|\frac{\sigma_s}{\lambda} -1 \big| + 1\big)$, we obtain
    \begin{equation}
        \big\lvert \lambda^{\frac{1}{2}}\sigma_s (\partial_y f)(\tau(s),\sigma(s)) + \theta \frac{\sigma_s}{\sigma}r \big\rvert \lesssim (1+\theta)\lambda^{\frac{3}{2}}\sigma^{-\theta-1}\Big(\Big\lvert \frac{\sigma_s}{\lambda}-1\Big\rvert + 1 \Big).
    \end{equation}
    Combining these estimates and using \eqref{BS1}, we get estimate \eqref{estimate_on_r_s_another}. We deduce the proof of \eqref{estimate_on_r_s} gathering the two estimates \eqref{estimate_on_r_s} and \eqref{estimates_on_r}.
    
    \vspace{0.4cm}
    From the definition of $F$ and $r$, we get the following
    \begin{equation}
        \lvert r(s)-F(s,y) \rvert = \lambda^{\frac{1}{2}}(s)\lvert f(\tau(s),\sigma(s))   -f(\tau(s),\lambda(s)y+ \sigma(s)) \rvert.
    \end{equation}
    
    We divide the study into two cases: $\lambda(s)y>-\frac{1}{4}\sigma(s)$ and $\lambda(s)y<-\frac{1}{4}\sigma(s)$.
    For the first case, it holds $\lambda(s)y + \sigma(s) > \frac{3}{4}\sigma(s) > \frac t2 +\frac{x_0}{2}$, so Lemma \ref{lemma_on_the_decaying_tail} applies. By the mean value theorem, we get the following estimate for such $y$
    \begin{equation}
        \big\lvert r(s) - F(s,y)\big\lvert \lesssim \lambda^{\frac 32}|y|\|\partial_y f(s)\|_{L^{\infty}(>\frac 34 \sigma(s))}
    \end{equation}
    
    In the region $\lambda(s)y<-\frac{1}{4}\sigma(s)$, by assumption \eqref{BS1}, it follows that $y< -\frac 14 \frac{\sigma}{\lambda} < -c s$ for some $c>0$.
    
    Therefore
    \begin{equation}
    \begin{split}
        &e^{-\frac{3}{4}|y|} \big\lvert r(s)-F(s,y) \big\rvert \\
        &\lesssim
        e^{-\frac{|y|}{2}}|y|\lambda^{\frac{3}{2}}(s) \Big( e^{-\frac{|y|}{4}}\|\partial_y f(\tau(s)) \|_{L^{\infty}(>\frac{3}{4}\sigma(s))} + e^{-\frac{\sigma(s)}{16 \lambda(s)}} \| \partial_y f(\tau(s))\|_{L^{\infty}} \Big)\\
        &\lesssim 
        e^{-\frac{|y|}{4}}\lambda^{\frac{3}{2}}(s) \Big( \sigma^{-\theta-1}(s) + e^{-cs}\Big) \lesssim e^{-\frac{|y|}{4}}s^{-2}.
    \end{split}
    \end{equation}

By conservation of the mass, it follows that
\begin{equation}
    \| F(s)\|_{L^2_y} \lesssim \lambda^{\frac{1}{2}}(s)\| f(\tau(s),\lambda(s)y+\sigma(s))\|_{L^2_y} \lesssim \|f(\tau(s))\|_{L^2_x} \lesssim x^{-\theta+\frac{1}{2}}_0.
\end{equation}
By energy conservation and by the Gagliardo-Nirenberg estimate, we get
\begin{equation}
    \|\partial_x f(\tau(s))\|^2_{L^2_x} \lesssim E(f_0) + \int f^6(\tau(s),x)\,dx
    \lesssim x^{-\theta-\frac{1}{2}} + \|\partial_x f(\tau(s))\|^2_{L^2_x}\| f_0\|^4_{L^2_x}.
\end{equation}
Therefore, using \eqref{estimate_on_f0_norm_L2} and \eqref{BS1}, we obtain 
\begin{equation}
    \|\partial_yF(s)\|_{L^2_y} = \lambda(s)\|\partial_x f(\tau(s))\|_{L^2_x} \lesssim s^{-\beta} x^{-\theta-\frac{1}{2}}_0.
\end{equation}

\vspace{0.4cm}

By Sobolev embedding $H^{1}(\RR) \subset C_0(\RR)$ and by \eqref{BS1}, it holds
\begin{equation}
    \|F(s)\|^2_{L^{\infty}} \lesssim \|F(s)\|_{L^2_y}\|\partial_y F(s) \|_{L^2_y} \lesssim s^{-\beta}x^{-2\theta}_0.
\end{equation}

For the estimate on $\{y>-2|b|^{-\gamma}\}$, we observe that for such $y$ by \eqref{BS1} and since $\gamma<1$, for $s_0 \gg 1$, it holds $\lambda(s)y+\sigma(s)> \frac{3}{4}\sigma(s)$.
Therefore, for $y>-2|b|^{-\gamma}$, we get
\begin{equation}
    \lvert F(s,y) \rvert \lesssim \lambda^{\frac{1}{2}}(s)\sigma^{-\theta}(s) \lesssim s^{-1} \quad \text{  and  }\quad \lvert \partial_y F(s,y) \rvert \lesssim \lambda^{\frac{3}{2}}\sigma^{-\theta-1} \lesssim s^{-2}.
\end{equation}
By Sobolev embedding and by the Lemma \ref{lemma_on_the_decaying_tail}, we get the estimate \eqref{estimate_on_L_infty_of_dj_F} 
\begin{equation}
    \| \partial^k_y F(s,y) \|_{L^{\infty}_y}  \lesssim \lambda^{k + \frac{1}{2}}(s)\|f(\tau(s))\|_{H^{k+1}} \lesssim s^{-\beta(k+\frac{1}{2})} \,\delta(x^{-1}_0).
\end{equation}


In order to prove the estimate in \eqref{estimates_on_norms_L_infty_of_F_with_exp}, we proceed as for \eqref{estimate_on_r_minus_F}, separating into two cases for $\lambda(s)y$ comparing with $-\frac{1}{4}\sigma(s)$, then using the Lemma \ref{lemma_on_the_decaying_tail}, \eqref{estimate_on_L_infty_of_dj_F} and \eqref{BS1}, we get, for $y \in \RR$,
\begin{equation}
\begin{split}
    e^{-\frac{|y|}{10}}\lvert \partial^k_y F(s,y) \rvert 
    & = e^{-\frac{|y|}{10}}\lambda^{(\frac 12  + k)}(s)\lvert \partial^k_2 f(\tau(s),\sigma(s)+y\lambda(s)) \rvert \\
    & \lesssim \lambda^{(\frac 12 + k)}(s)\big( \sigma^{-\theta - k}(s)+ e^{-cs} \big) \lesssim s^{-(k+1)}.
\end{split}
\end{equation}
for some $c>0$.

We write
\begin{equation}
        \big( \partial_y(5Q^4 (r-F)), Q \big) = -\int (Q^5)' (r-F) = - \int Q^5 \partial_y F = -\lambda^{\frac{3}{2}} \int Q^{5}(y) (\partial_y f)(\tau(s),\lambda y + \sigma)\,dy.
\end{equation}
Using that $\int Q^5 = \int Q$, we get
\begin{equation}
\begin{split}
    & \bigg\lvert \big( \partial_y(5Q^4(r-F)), Q\big) - c_0 \theta \Big(\int Q\Big)\lambda^{\frac{3}{2}}\sigma^{-\theta-1} \bigg\rvert \\
    & \lesssim \lambda^{\frac{3}{2}} \bigg[ \int Q^5 \big\lvert\partial_y f(\tau(s),\lambda y +\sigma) - f'_{0}(\sigma)\big\rvert +\Big( \int Q \Big)\big\lvert f'_0(\sigma)+ c_0\theta \sigma^{-\theta-1} \big\rvert \bigg].
\end{split}
\end{equation}
By \eqref{assumption_on_sigma_on_t}, the term on the right-hand side is zero.

Separating in two cases for $\lambda y $ compared with $-\frac{1}{4}\sigma$, using \eqref{assumption_on_sigma_on_t}, Lemma \ref{lemma_on_the_decaying_tail}, the mean value theorem and \eqref{BS1}, we obtain
\begin{equation}
    \begin{split}
        &\lambda^{\frac{3}{2}}\int Q^5 \big\lvert \partial_y f(\tau(s),\lambda y +\sigma) -f'_0(\sigma) \big\rvert\\
        & \leq \lambda^{\frac{3}{2}}\int Q^5 \Big[ \big\lvert \partial_y f(\tau(s),\lambda y +\sigma) - f'_0(\lambda y+\sigma) \big\rvert + \big\lvert f'_0(\lambda y +\sigma) - f'_0(\sigma)\big\rvert \Big]\\
        & \lesssim \lambda^{\frac{3}{2}} \big( \sigma^{-\theta-3} + e^{-cs} \big) + \lambda^{\frac{5}{2}}\big( \sigma^{-\theta-2}+e^{-cs} \big)
        \lesssim s^{-4+2\beta}.
    \end{split}
\end{equation}

    
\end{proof}

\begin{lemma}[Mass and energy of $W+F$]
Under the assumption \eqref{BS1}, for $s_0$ large enough, it holds for $s\geq s_0$
\begin{equation}\label{estimate_on_mass_W_plus_F}
    \bigg\lvert \int (W+F)^2-\Big( \int Q^2 + 2b\int P Q + \frac{1}{2}r \int Q \Big) \bigg\rvert \lesssim s^{-2+\gamma}+x^{-2\theta+1}_0,
\end{equation}
\begin{equation}\label{estimate_on_energy_W_plus_F}
 \bigg\lvert   \lambda^{-2}E(W+F) + \frac{1}{16}\Big( \int Q \Big)^2g(s) \bigg\rvert \lesssim s^{-3+2\beta} + \lambda^2 x^{-2\theta-1}_0.
\end{equation}
    
\end{lemma}
\begin{proof}
    First, we write, using $\int QR=-\frac 34 \int Q$,
    \begin{equation}
    \begin{split}
         &\int (W+F)^2 - \int Q^2_b - \frac{1}{2}r \int Q = \\
         &= 2br\int P_b R + 2\int Q(F-r)+2b\int P_b F + r^2 \int R^2 + 2r \int R F + \int F^2.
    \end{split}
    \end{equation}
    From the conservation of the $L^2$ norm, we get
    \begin{equation}
        \|F\|^2_{L^2_y} = \| f\|^2_{L^2_x} = \|f_0\|^2_{L^2} \approx x^{-2\theta+1}_0.
    \end{equation}
    By the following property of $Q$, $|Q^{(k)}(y)|\lesssim e^{-|y|}$ on $\RR$, for $k \geq 0$, the estimates \eqref{estimate_on_r_minus_F}, \eqref{estimates_on_r}, \eqref{BS1} and since $P$ is bounded and $R \in \mathcal{Y}$, we obtain
    \begin{equation}
        \bigg\lvert 2 \int Q(F-r) + 2br\int P_bR + r^2 \int R^2\bigg\rvert \lesssim s^{-2}.
    \end{equation}
    The estimates \eqref{estimates_on_r} and \eqref{estimates_on_norms_L_infty_of_F_with_exp} yield
    \begin{equation}
        \bigg\lvert 2r\int R F \bigg\rvert \lesssim s^{-2}.
    \end{equation}
    Using the property $|P(y)|\lesssim e^{-\frac{|y|}{2}}$ for $y\geq0$, the estimate \eqref{estimates_on_norms_L_infty_of_F} and that $\gamma=\frac{3}{4}$, we get
    \begin{equation}
    \begin{split}
        \bigg\lvert 2b\int P_b F \bigg\rvert 
        & \lesssim 2|b|\| F(s)\|_{L^{\infty}(y>-2|b|^{-\gamma})} \Big( \int^{0}_{-2|b|^{-\gamma}} |P(y)|\,dy + \int_{y>0}e^{-\frac{|y|}{2}}\, dy \Big) \\
        & \lesssim
        s^{-2}(|b|^{-\gamma} + c) \lesssim s^{-2+\gamma}.
    \end{split}
    \end{equation}
    Hence, it follows that
    \begin{equation}
        \bigg\lvert \int (W+F)^2 - \int Q^2_b - \frac{1}{2}r \int Q \bigg\rvert \lesssim s^{-2+\gamma}+x^{-2\theta+1}_0.
    \end{equation}
    The estimate \eqref{estimate_on_mass_W_plus_F} then follows from the mass property \eqref{estimate_mass_of_Qb} of $Q_b$ combined with the previous estimate.

    \vspace{0.4cm}
    We rewrite the energy of $E(W+F)$ as
    \begin{equation}
    \begin{split}
        E(W+F) 
        & = E(Q_b) + E(F) + \int Q'_b\partial_y(rR+F) + r\int R' \partial_yF + \frac{1}{2}r^2\int (R')^2\\
        & - \int Q^5_b (rR+F) - \frac{1}{6}\int \Big((Q_b+rR+F)^6 - Q^6_b - F^6 - 6Q^5_b (rR+F)  \Big).
    \end{split}
    \end{equation}
    Using the equation \eqref{equation_of_Q} of $Q$, the estimates \eqref{estimate_energy_of_Qb}, \eqref{properties_of_R}, \eqref{relation_on_scalar_prod_P_and_Q}, then the estimates \eqref{estimate_on_r_minus_F}, \eqref{estimates_on_r} on $F-r$ and on $r$, the explicit expression of $g$ and \eqref{BS1}, we write
    \begin{equation}
    \begin{split}
        & \Big\lvert E(Q_b) + \int (-Q''-Q^5)(rR+F) + \frac{1}{16}\Big(\int Q\Big)^2 g(s) \Big\rvert \\
        & \lesssim b^2 + \Big\lvert \int QF - \frac{3}{4}r\int Q + \frac{b}{16}\Big(\int Q\Big)^2 - \frac{1}{16}\Big(\int Q \Big)^2 g(s)\Big\rvert\\
        & \lesssim s^{-2} + \frac{1}{4}\Big(\int Q\Big)\Big\lvert r - c_0 \lambda^{\frac{1}{2}}\sigma^{-\theta}\Big\rvert \lesssim s^{-3+2\beta}.
    \end{split}
    \end{equation}
    
    By conservation of the energy, we get
    \begin{equation}
        E(F(s)) = \lambda^2E(f(\tau(s))) = \lambda^2 E(f_0) \approx \lambda^2 x^{-2\theta - 1}_0.
    \end{equation}
    By the properties of $Q$ and of $R$, the decay on the right of $P$, the estimates \eqref{estimates_on_norms_L_infty_of_F_with_exp} on $F$ and $\partial_yF$, the estimate \eqref{estimates_on_r}, the value of $\gamma = \frac{3}{4}$ and \eqref{BS1}, we obtain
    \begin{equation}
    \begin{split}
        & \Big\lvert \int Q'_b \partial_y(rR+F) + \int Q''(rR+F) \Big\rvert
        + \Big\lvert \int Q^5_b(rR+F) - \int Q^5 (rR+F) \Big\rvert \\
        & +\Big\lvert r\int R' \partial_y F + r^2 \int (R')^2 \Big\rvert + \Big\vert \int \Big( (Q_b+rR+F)^6- Q^6_b - F^6 - 6Q^5_b(rR+F) \Big) \Big\rvert \lesssim s^{-2}.
    \end{split}
    \end{equation}
    Combining the last three estimates and the expression of $E(W+F)$, yields the result.
\end{proof}

\begin{lemma}
    The term $\mathcal{E}(W)$ can be rewritten as
    \begin{equation}\label{equation_rewriting_of_error}
        \mathcal{E}(W) = -\vec{m}\cdot\vec{M}Q + \mathcal{R}.
    \end{equation}
    Moreover, for $s_0$ large enough, for $s \geq s_0$, it holds
    \begin{equation}\label{estimate_on_norms_L_2_B_of_R}
        \|\mathcal{R} \|_{L^2_B} + \| \partial_y \mathcal{R}\|_{L^2_B} \lesssim |b_s| +s^{-2} + s^{-1}|\vec{m}|
    \end{equation}
    and
    \begin{equation}\label{estimate_on_scalar_prod_R_Q}
        \bigg\lvert\big( \mathcal{R}, Q\big) -\frac{1}{16}\Big(\int Q\Big)^2\bigg[  b_s+2b^2 - \frac{4c_0}{\int Q}\lambda^{\frac{1}{2}}\sigma^{-\theta}\Big( \frac{3}{2}\frac{\lambda_s}{\lambda} + \theta \frac{\sigma_s}{\sigma}\Big) \bigg] \bigg\rvert \lesssim  |b_s|e^{-\frac{s^{\gamma}}{2}} + s^{-4+2\beta} + s^{-1}|\vec{m}|.
    \end{equation}
\end{lemma}
\begin{proof}
From the definitions of $W$, $Q_b$, $\Psi_b$ of $\vec{m}$, $\vec{M}$ and the equation of $R$, the error term $\mathcal{E}(W)$ can be rewritten as
\begin{equation}
    \mathcal{E}(W) = - \vec{m}\cdot\vec{M}\, Q+\mathcal{R}.
\end{equation}
With 
\begin{equation}
    \begin{split}
        &\mathcal{R} = \mathcal{R}_1 + \partial_y \mathcal{R}_2 + \mathcal{R}_3 - \Psi_b,\\
        &\mathcal{R}_1 = r_s \,R + b\,r\, \Lambda R + 5 \partial_y \big[ Q^4_b(rR+F)-Q^4(rR+r) \big],\\
        & \mathcal{R}_2 = \big( Q_b + rR + F \big)^5 - Q^5_b - F^5 - 5 Q^4_b (rR+F),\\
        &\mathcal{R}_3 = (b_s+2b^2)\,\partial_b Q_b - b\, \vec{m}\cdot \vec{M} P_b - r \,\vec{m}\cdot \vec{M}R.
    \end{split}
\end{equation}


For $k=1,2$, we have the following rewriting 
\begin{equation}
\begin{split}
    &\partial^k_y \big[ Q^4_b(rR+F)-Q^4(rR+r)\big]\\
    &= r \partial^k_y\big( (Q^4_b - Q^4)R\big) + \partial^k_y F\, Q^4_b + 2\big[\partial_y F \,\partial_y (Q^4_b)\big]^{k-1} + F\,\partial^k_y(Q^4_b - Q^4) + (F-r)\partial^k_y (Q^4).
\end{split}
\end{equation}
Therefore, using the estimates \eqref{estimate_on_r_s}, \eqref{estimate_on_r_minus_F}, \eqref{estimates_on_norms_L_infty_of_F}, \eqref{estimates_on_norms_L_infty_of_F_with_exp} from Lemma \ref{lemma_with_all_estimates_on_r_F_etc}, and the definition of $P_b$ in \eqref{def_of_P_b} with the decay on the right of $P$ announced in \eqref{property_on_right_on_P}, we get
\begin{equation}
\begin{split}
    \|\mathcal{R}_1\|_{L^2_B} + \|\partial_y \mathcal{R}_1\|_{L^2_B} &\lesssim s^{-2} + b^4\Big( \int^{0}_{-2b^{-\gamma}} \big( (\partial^2_yF)^2 + (\partial_y F)^2 + F^2\big) e^{\frac yB}\, dy \Big)^{\frac 12}\\
    &\lesssim s^{-2} + |b|^{4- \frac{\gamma}{2}}\|F\|_{L^{\infty}(y> - 2|b|^{-\gamma})}\lesssim s^{-2}.
\end{split}
\end{equation}

Proceeding as previously and using the following rewriting of $\mathcal{R}_2$
\begin{equation}\label{estimate_R_2_in_proof_of_error_term}
    \mathcal{R}_2 = 10\,Q^3(rR+F)^2 + 10\,(Q^3_b - Q^3)(rR+F)^2 + 10\,Q^2_b(rR+F)^3+ 5\,Q_b (rR+F)^4 + (rR+F)^5 - F^5,
\end{equation}
it holds
\begin{equation}\label{estimate_1_in_proof_lemma_error_term}
\begin{split}
    \|\partial_y \mathcal{R}_2\|_{L^2_B} + \|\partial^2_y \mathcal{R}_2\|_{L^2_B} \lesssim s^{-2}
\end{split}
\end{equation}

It remains to estimate the term $\mathcal{R}_3$. The definition  of $L^2_B$ norm yields that the following quantities are finite
\begin{equation}
    \|\mathbf{1}_{[-2b^{-\gamma}, 0]}\|_{L^2_B} +  \|y \mathbf{1}_{[-2b^{-\gamma}, 0]}\|_{L^2_B} \lesssim B^{\frac12}.
\end{equation} 
We have $\partial_b \, Q_b = P_b + \gamma \, y \, (\chi_b)' P$, therefore by the decay on the right of $P$, the property if $P' \in \mathcal{Y}$ and the previous estimate, it yields
\begin{equation}
   \|\mathcal{R}_3\|_{L^2_B}+ \|\partial_y\mathcal{R}_3\|_{L^2_B} \lesssim_B |b_s| + b^2 + |\vec{m}| (|r| + b) \lesssim_B |b_s| + s^{-2} + s^{-1}|\vec{m}|.
\end{equation}

Combining the above estimates, we get the estimate in \eqref{estimate_on_norms_L_2_B_of_R}.

\vspace{0.4cm}
By \eqref{properties_of_R} and by the estimate \eqref{estimate_on_r_s} on $r_s$, we get 
\begin{equation}
    \Big\lvert r_s\big( R,Q\big) + \frac{3}{4}\Big(\int Q\Big)\, c_0 \lambda^{\frac{1}{2}}\sigma^{-\theta}\Big( \frac{1}{2}\frac{\lambda_s}{\lambda} - \theta \frac{\sigma_s}{\sigma} \Big) \Big\rvert \lesssim s^{-4+2\beta}+s^{-3+2\beta}|\vec{m}|.
\end{equation}

The following relation holds true
\begin{equation}
    \Lambda Q + 20PQ^3Q' = (\mathcal{L}P)' + 20Q^3Q'P = \mathcal{L}(P').
\end{equation}
Thus, using the propriety $\mathcal{L}R = 5Q^4$, we write
\begin{equation}\label{estimate_on_br_scalar_prod_in_proof_scalar_prod_of_R_Q}
    \begin{split}
        br \big( \Lambda R, Q\big) = - br \big( R, \Lambda Q \big) = -br\big( R, \mathcal{L}(P') \big) + 20br \big( R, PQ^3Q' \big) = 20br\big( Q^3 P (R+1), Q' \big).
    \end{split}
\end{equation}

The following rewriting holds
\begin{equation}
    \begin{split}
        5 \big( \partial_y\big[ Q^4_b(rR+F)-Q^4(rR+r) \big], Q \big) = -5\big( (Q^4_b - Q^4)(rR+F), Q' \big) - 5\big( Q^4(F-r), Q' \big).
    \end{split}
\end{equation}

By \eqref{estimate_on_r_minus_F} and \eqref{BS1}, the first term in the rewriting can be estimated in the following way
\begin{equation}
\begin{split}
    -5\big( (Q^4_b - Q^4)(rR+F), Q' \big) 
    & = -5r\big( (Q^4_b - Q^4 )(R+1),Q' \big) - 5r \big( (Q^4_b - Q^4)(F-r),Q' \big)\\
    & \lesssim -20rb \big( Q^3P_b(R+1), Q' \big) + s^{-3}.
\end{split}
\end{equation}

And the first term in the estimate can be combined with \eqref{estimate_on_br_scalar_prod_in_proof_scalar_prod_of_R_Q} and the proprieties of $1-\chi_b$
\begin{equation}
    \big| br\big( \Lambda R, Q \big) - 20 br \big( Q^3(R+1)P_bQ' \big)\big| \lesssim s^{-3}.
\end{equation}

From \eqref{estimates_on_scalar_prod_r_minus_F_and_Q}, $|\lambda- \sigma_s| \leq \lambda |\vec{m}|$ and \eqref{BS1}, we estimate the second term on the right-hand side
\begin{equation}
    \bigg\lvert - 5\big( Q^4(F-r), Q' \big) + c_0 \theta \Big( \int Q\Big) \lambda^{\frac{1}{2}}\sigma^{-\theta-1}\sigma_s \bigg\rvert \lesssim s^{-4+2\beta} + \lambda^{\frac{1}{2}}\sigma^{-\theta-1}|\lambda - \sigma_s|\lesssim s^{-4+2\beta}+s^{-2}|\vec{m}|.
\end{equation}

Since $\beta>\frac{1}{2}$, we conclude the estimate for $\mathcal{R}_1$
\begin{equation}
\begin{split}
    & \bigg\lvert \big( \mathcal{R}_1,Q\big) +\frac{3}{4}\Big(\int Q\Big)\, c_0 \lambda^{\frac{1}{2}}\sigma^{-\theta}\Big( \frac{1}{2}\frac{\lambda_s}{\lambda} - \theta \frac{\sigma_s}{\sigma}\Big) + c_0 \theta \Big( \int Q\Big) \lambda^{\frac{1}{2}}\sigma^{-\theta-1}\sigma_s \bigg\rvert\\
    & =
    \bigg\lvert \big( \mathcal{R}_1,Q\big) + \frac{1}{4}c_0 \Big( \int Q\Big) \lambda^{\frac{1}{2}} \sigma^{-\theta} \Big(\frac{3}{2}\frac{\lambda_s}{\lambda} + \theta \frac{\sigma_s}{\sigma} \Big) \bigg\rvert \lesssim s^{-4+2\beta} + s^{-3+2\beta}|\vec{m}|.
\end{split}
\end{equation}

We show now that 
\begin{equation}
    \big\lvert \big( \partial_y \mathcal{R}_2, Q \big) \big\rvert \lesssim s^{-3}.
\end{equation}
From the rewriting \eqref{estimate_R_2_in_proof_of_error_term} and \eqref{estimates_on_norms_L_infty_of_F_with_exp}, \eqref{estimates_on_r}, we get
\begin{equation}
    \big\lvert \big( \partial_y \mathcal{R}_2, Q \big) \big\rvert \lesssim \big|10 \big( Q^3(rR+F)^2, Q' \big) \big| + s^{-3}.
\end{equation}
We write then
\begin{equation}
    10\big( Q^3(rR+F)^2, Q' \big) = -5r^2\big( (R+1)\,\partial_y R\,, Q^4 \big) -5r \big( (F-r)\,\partial_y R\,, Q^4 \big)-5r \big( R\, \partial_y F\,, Q^4 \big) - 5 \big( F\, \partial_y F\,, Q^4 \big).
\end{equation}
By the propriety $\mathcal{L}R = 5Q^4$, we have
\begin{equation}
    -5r^2 \big( (R+1)\,\partial_y R\, , Q^4 \big) = -r^2 \, \big( R - R^2_y\, , R_y \big) = 0.
\end{equation}
Combining with the estimates \eqref{estimates_on_r}, \eqref{estimate_on_r_minus_F}, \eqref{estimates_on_norms_L_infty_of_F_with_exp}, we get
\begin{equation}
    \big\lvert10\big( Q^3(rR+F)^2, Q' \big)\big\rvert \lesssim s^{-3}.
\end{equation}
This concludes the estimate on the term $\mathcal{R}_2$.

We remark that by \eqref{relation_on_scalar_prod_P_and_Q}, the definition of $\chi_b$ and the decay of $Q$, we have
\begin{equation}
\begin{split}
    & (b_s+2b^2) \bigg\lvert \big(\partial_bQ_b,P \big) -\frac{1}{16} \Big( \int Q\Big)^2 \bigg\rvert\\
    & \lesssim (b_s +2b^2)\bigg\lvert \big(Q,P\big) + \big((\chi_b-1)P,Q \big) + \gamma\big(  y P \chi'_b, Q\big) -\frac{1}{16} \Big( \int Q\Big)^2 \bigg\rvert\\
    & \lesssim (b_s + b^2)e^{-\frac{s^{\gamma}}{2}} \lesssim |b_s|e^{-\frac{s^{\gamma}}{2}} + s^{-3}.
\end{split}
\end{equation}
Thus, we conclude for the term $\mathcal{R}_3$
\begin{equation}
    \bigg\lvert \big( \mathcal{R}_3, Q \big) - \frac{1}{16}\Big(\int Q \Big)^2 (b_s + 2b^2) \bigg\rvert \lesssim  |b_s|e^{-\frac{s^{\gamma}}{2}}+ s^{-3} + s^{-1}|\vec{m}|.
\end{equation}
Combining the estimates on the projections of the terms $\mathcal{R}_1$,  $\partial_y\mathcal{R}_2$ and  $\mathcal{R}_3$ on $Q$, and since $\frac{1}{2}<\beta<1$ yields the estimate \eqref{estimate_on_scalar_prod_R_Q}.
\end{proof}

\begin{lemma}[Estimates on $h_s$ and $g_s$]\label{lemma_with_estimates_on_hs_and_gs}
For $s \geq s_0 $, it holds
\begin{equation}\label{estimate_on_gs_with_prod_scal_R_Q}
    \big\lvert c_1 \lambda^2 g_s - \big(\mathcal{R},Q\big) \big\rvert \lesssim |b_s|e^{-\frac{s^{\gamma}}{2}}+s^{-4+2\beta} + s^{-1}|\vec{m}|,
\end{equation}
\begin{equation}\label{estimate_on_hs_with_g}
    \big\lvert \lambda^{-\frac{1}{2}}h_s + \frac{1}{2}\lambda^2 g \big\rvert \lesssim |\vec{m}|.
\end{equation}
\end{lemma}
\begin{proof}
    We compute
    \begin{equation}
        \frac{d}{ds}g(s) = \frac{b_s}{\lambda^2}+2\frac{b^2}{\lambda^2}-2\frac{b}{\lambda^2}\Big( \frac{\lambda_s}{\lambda}+b \Big) - \frac{4\,c_0}{\int Q} \lambda^{-\frac{3}{2}}\sigma^{-\theta}\bigg( \frac{3}{2}\frac{\lambda_s}{\lambda}+\theta\frac{\sigma_s}{\sigma} \bigg).
    \end{equation}
    The estimates \eqref{estimate_on_scalar_prod_R_Q} and \eqref{BS1} yields
    \begin{equation}
        \big\lvert c_1\,\lambda^2 g_s - \big(\mathcal{R},Q \big) \big\rvert \lesssim |b_s|e^{-\frac{s^{\gamma}}{2}}+s^{-4+2\beta} + s^{-1}|\vec{m}|.
    \end{equation}
    For the second estimate, by \eqref{BS1}, we get
    \begin{equation}
        \big\lvert \lambda^{-\frac{1}{2}}h_s + \frac{1}{2}\lambda^2 g \big\rvert = \Big\lvert \frac{1}{2}\Big(\frac{\lambda_s}{\lambda} + b\Big) - \frac{2 \, c_0}{\int Q} \lambda^{\frac{1}{2}}\sigma^{-\theta} \Big( \frac{\sigma_s}{\lambda}-1 \Big)\Big\rvert \lesssim |\vec{m}|+s^{-1}|\vec{m}|\lesssim | \vec{m}|.
    \end{equation}
\end{proof}

\subsection{Modulation and parameters estimates}
Let $U $ be a solution of \eqref{gKdV_principal_eq} defined on the time interval $I $ included in $[0,T)$. We denote $s(I)$ the rescaled time interval following the definition introduced in \eqref{def_of_rescaled_time}.

Consider $f$ defined in \eqref{equation_of_f} and set
\begin{equation}
    v(t,x) = U(t,x) - f(t,x).
\end{equation}
We say \textit{$v$ is close to the approximate blow-up profile on $I$}, if for all $t \in I$ there exist $(\lambda_{\sharp}(t), \sigma_{\sharp}(t))\in (0,+\infty)^2$ and $z \in C(I,L^2(\RR))$ such that
\begin{equation}\label{supposed_decomposition_around_Q}
\begin{split}
    v(t,x)=\lambda^{-\frac{1}{2}}_{\sharp}(t)(Q+r_{\sharp}R+z)(t,y),\\
     y= \frac{x-\sigma_{\sharp}(t)}{\lambda_{\sharp}(t)},\qquad  r_{\sharp(t)}= \lambda^{\frac{1}{2}}_{\sharp}(t)f(t,\sigma_{\sharp}(t)).
\end{split}
\end{equation}
with for all $t \in I$ and for small enough universal constant $\alpha^*>0$
\begin{equation}\label{supposed_decomposition_around_Q_condition}
    \|z(t)\|_{H^1}+ \sigma_{\sharp}^{-1}(t)+\lambda_{\sharp}(t ) \leq \alpha^{*}.\\
\end{equation}

\begin{lemma}\label{lemma_on_decomposition_around_Q}
    Assume $v$ is close to the approximate blow-up profile on $I$ for some $\alpha^{*}$ sufficiently small. There exist unique continuous functions $(\lambda,\sigma,b):I\mapsto (0,+\infty) \times \RR^2$ such that
    \begin{equation}\label{decomposition_of_v_in_lemma}
        \lambda^{\frac{1}{2}}(t)v(t, \lambda(t)y+\sigma(t)) = W(t,y)+ \vare(t,y)
    \end{equation}
    where $W$ is defined in \eqref{def_of_W_and_of_r} and  $\vare$ satisfies for all $t\in I$
    \begin{equation}\label{orthogonal_conditions_in_lemma}
        (\vare(t),y\Lambda Q ) = (\vare(t),\Lambda Q) = (\vare(t), Q) = 0.
    \end{equation}
    Moreover,
    \begin{equation}
        \bigg\lvert \frac{\lambda(t)}{\lambda_{\sharp}(t)} - 1\bigg\rvert + |b(t)| + \bigg\lvert \frac{\sigma(t)-\sigma_{\sharp}(t)}{\lambda_{\sharp}(t)} \bigg\rvert \lesssim \|z(t)\|_{L^2_{sol}}
    \end{equation}
    and
    \begin{equation}\label{estimates_on_norms_of_vares_in_z_in_decomp}
        \|\vare(t)\|_{L^2} \lesssim \|z(t)\|^{\frac 58}_{L^2}, \qquad \|\vare(t)\|_{L^2_{sol}} \lesssim \|z(t)\|_{L^2_{sol}},
    \end{equation}
    \begin{equation}\label{estimates_on_norms_of_partial_vares_in_z_in_decomp}
        \|\partial_y \vare(t)\|_{L^2} \lesssim \|z(t)\|_{L^2_{sol}} + \|\partial_y z(t)\|_{L^2},  \qquad \|\partial_y \vare(t)\|_{L^2_{sol}} \lesssim \|z(t)\|_{L^2_{sol}}+\|\partial_y z(t)\|_{L^2_{sol}}.
    \end{equation}
\end{lemma}

\begin{proof}
    The decomposition is performed first for a fixed $t \in I$. Define a map
    \begin{equation}
        \Theta: (\bar{\lambda}, \bar{\sigma}, \bar{b}, v_1)\in (0,+\infty)\times \RR^2\times H^1 \mapsto \big((\bar{\vare},y\Lambda Q), (\bar{\vare}, \Lambda Q), (\bar{\vare}, Q) \big) \in \RR^3.
    \end{equation}
Where $\bar{\vare}$ is defined by
\begin{equation}
    \begin{split}
        \bar{\vare}(y)= \bar{\vare}_{(\bar{\lambda}, \bar{\sigma}, \bar{b}, v_1)}(y):= \bar{\lambda}^{\frac{1}{2}}v_1(t, \bar{\lambda }y + \bar{\sigma}) - Q_{\bar{b}}(y)-\bar{\lambda}^{\frac{1}{2}} \lambda_{\sharp}^{\frac{1}{2}}(t)f(t,\sigma_{\sharp}(t)+\lambda_{\sharp}(t)\bar{\sigma})R(y).
    \end{split}
\end{equation}
Consider $v_{\sharp}= Q+r_{\sharp}R$ and $\theta_0 = (1,0,0,v_{\sharp})$. We remark that $\bar{\vare}_{\theta_0}= 0$, thus $\Theta(\theta_0)=0$.

We compute the partial derivatives of $\bar{\vare}$ on $\theta_0$
\begin{equation}
    \partial_{\bar{\lambda}} \bar{\vare}\vert_{\theta_0} = \Lambda Q + r_{\sharp}y R',
\end{equation}
\begin{equation}
    \partial_{\bar{\sigma}} \bar{\vare}\vert_{\theta_0} = Q'+ r_{\sharp}R' - \lambda^{\frac{3}{2}}_{\sharp} \partial_x f(t,\sigma_{\sharp})R,
\end{equation}
\begin{equation}
    \partial_{\bar{b}} \bar{\vare}\vert_{\theta_0} = -P.
\end{equation}

Therefore, the Jacobian matrix of $\Theta$ on $\theta_0$ is given by
\begin{equation}
        \Jac_{\Theta}(\theta_0) = 
        \begin{pmatrix}
(\Lambda Q, \Lambda Q) & (Q', y\Lambda Q) + r_{\sharp}(R', y\Lambda Q) & -(P,y\Lambda Q)\\
\|\Lambda Q\|^2_{L^2}+r_{\sharp}(yR',\Lambda Q) & (Q', \Lambda Q)- \lambda^{\frac{3}{2}}\partial_x f(t, \sigma_{\sharp})(R, \Lambda Q) & -(P,\Lambda Q)\\
(\Lambda Q, Q)+ r_{\sharp}(y R', Q) & (Q',Q)- \lambda^{\frac{3}{2}}_{\sharp} \partial_x f(t, \sigma_{\sharp})(R,Q) & -(P,Q)
\end{pmatrix}.
\end{equation}

Assume $\alpha^*$ small enough such that by \eqref{supposed_decomposition_around_Q_condition}$,  \sigma_{\sharp}> \frac{t_0}{2}+ \frac{x_0}{2}$. By Lemma \ref{lemma_on_the_decaying_tail}, we get
\begin{equation}
    \lvert \lambda^{\frac{3}{2}}_{\sharp} \partial_x f(t, \sigma_{\sharp}) \rvert + |r_{\sharp}| \lesssim (\alpha^*)^{\frac{1}{2}+\theta} \lesssim (\alpha^*)^{\frac{1}{2}}.
\end{equation}
By symmetry and by explicit computations, we get
\begin{equation}
    (\Lambda Q, y \Lambda Q ) = (Q', \Lambda Q) = (Q',Q) = (\Lambda Q, Q) = 0 \qquad \text{and} \qquad 
    (Q', y \Lambda Q) = \|\Lambda Q \|^2_{L^2}.
\end{equation}
Using \eqref{relation_on_scalar_prod_P_and_Q} and previous estimates, we obtain
\begin{equation}
    \Jac_{\Theta}(\theta_0) = 
    \begin{pmatrix}
0 & \|\Lambda Q\|^2_{L^2} & -(P,y\Lambda Q)\\ 
\|\Lambda Q\|^2_{L^2} & 0 & -(P,\Lambda Q)\\ 
0 & 0 & -\frac{1}{16}\|Q\|^2_{L^1}
\end{pmatrix} + O\Big( (\alpha^*)^{\frac{1}{2}} \Big).
\end{equation}

Thus for $\alpha^*>0$ small enough, the Jacobian determinant satisfies
\begin{equation}
    \det{\Jac_{\Theta}(\theta_0)} = \frac{1}{16}\|\Lambda Q \|^4_{L^2}\|Q\|^2_{L^1} + O\Big( (\alpha^*)^{\frac{1 }{2}} \Big) >  \frac{1}{32}\|\Lambda Q \|^4_{L^2}\|Q\|^2_{L^1} > 0.
\end{equation}

Therefore, taking $\alpha^*>0$ small enough, by the implicit function theorem, for any $v_1 = v_{\sharp} + z = Q+r_{\sharp}R + z$ with $\|z\|_{H^1}< \alpha^*$, there exists a unique $(\bar{\lambda}, \bar{\sigma}, \bar{b})(v_1)$ close to $(1,0,0)$ in $(0,+\infty)\times \RR^2$, such that $\Theta(\bar{\lambda}, \bar{\sigma}, \bar{b},v_1) =0$. Moreover, the map $v_1 \mapsto (\bar{\lambda}, \bar{\sigma}, \bar{b})(v_1)$ is continuous. 

\vspace{0.3cm}
Let $v$ be \textit{close to the approximate blow-up profile}, and assume $\alpha^*>0$ small enough, such that $\|z(t) \|_{H^1} \leq \alpha^*$, so that $v_1(t,y):=\lambda^{\frac{1}{2}}_{\sharp}(t)v(t,\lambda_{\sharp}(t)y + \sigma_{\sharp}(t))$ is close to $v_{\sharp}(t)$ in $H^1$. Define $(\bar{\lambda}_1,\bar{\sigma}_1,\bar{b}_1)(t) := (\bar{\lambda},\bar{\sigma},\bar{b})(v_1(t))$ and $\bar{\vare}_1(t,y):= \bar{\vare}_{(\bar{\lambda}_1,\bar{\sigma}_1,\bar{b}_1, v_1)(t)}(y)$.
We consider
\begin{equation}
    \lambda(s(t)):=\lambda_{\sharp}(t)\bar{\lambda}_1(t), \qquad \sigma(s(t)):=\lambda_{\sharp}(t)\bar{\sigma}_1(t) + \sigma_{\sharp}(t), \qquad b(s(t)):=\bar{b}_1(t), \qquad \vare(s(t),y):= \bar{\vare}_1(t,y).
\end{equation}
In particular, $\vare$ satisfies the orthogonality conditions \eqref{orthogonal_conditions_in_lemma} and the desired decomposition \eqref{decomposition_of_v_in_lemma} of $v$ holds true.

\vspace{0.4cm}
The equation of $\vare$ can be rewritten as
\begin{equation}\label{equation_of_vare_inside_of_proof_of_decomposition}
    \vare(y) = \bar{\lambda}^{\frac{1}{2}}_1 Q(\bar{\lambda}_1 y + \bar{\sigma}_1)- Q_{\bar{b}_1}(y)+ \bar{\lambda}^{\frac{1}{2}}_1 \lambda^{\frac{1}{2}}_{\sharp}\Big[ f(\sigma_{\sharp})R(\bar{\lambda}_1 y + \bar{\sigma}_1) - f(\sigma_{\sharp}+\lambda_{\sharp}\bar{\sigma}_1)R(y) \Big] + \bar{\lambda}^{\frac{1}{2}}_1 z(\bar{\lambda}_1 y + \bar{\sigma}_1).
\end{equation}

We will project $\vare$ on three directions of orthogonality, $y\Lambda Q$, $\Lambda Q$ and $Q$. 

We will use the following estimate 
\begin{equation}\label{estimate_on_difference_of_Q_in_proof_of_decomposition}
    \bar{\lambda}_1^{\frac{1}{2}} Q(\bar{\lambda}_1 x + \bar{\sigma}_1) - Q(x) = \bar{\sigma}_1 Q'(\bar{\lambda}_1 x) + \frac{1}{2}(\bar{\lambda}_1-1)\bar{\sigma}_1 Q'(\bar{\lambda}_1 x) + (\bar{\lambda}_1 -1)\Lambda Q(x) + \mathcal{O}\big((\bar{\lambda}_1 -1)^2 e^{-\frac{|x|}{2}}\big) + \mathcal{O}\big(\bar{\sigma}_1 e^{-|x|}\big).
\end{equation}
In order to get this estimate, we write
\begin{equation}
    \bar{\lambda}_1^{\frac{1}{2}} Q(\bar{\lambda}_1 x + \bar{\sigma}_1) - Q(x) = \Big[\bar{\lambda}_1^{\frac{1}{2}}Q(\bar{\lambda}_1 x + \bar{\sigma}_1 ) - \bar{\lambda}_1^{\frac{1}{2}}Q(\bar{\lambda}_1 x)\Big] + \Big[\bar{\lambda}_1^{\frac{1}{2}}Q(\bar{\lambda}_1 x)  - \bar{\lambda}_1^{\frac{1}{2}}Q(x)\Big] + (\bar{\lambda}_1^{\frac{1}{2}}-1)Q(x) =: I_1+I_2+I_3.
\end{equation}
By Taylor expansion with the integral form of the remainder, we get
\begin{equation}
    I_1 = \bar{\sigma}_1\, Q'(\bar{\lambda}_1x)+ \frac{1}{2}(\bar{\lambda}_1 -1)\bar{\sigma}_1\,Q'(\bar{\lambda}_1 x) + \mathcal{O}\big(((\bar{\lambda}_1 - 1)^2 + \bar{\sigma}_1^{2}) e^{-|x|}\big),
\end{equation}
\begin{equation}
    I_2 = (\bar{\lambda}_1 -1)x\, Q'(x) + \mathcal{O}\big((\bar{\lambda}_1 - 1)^2 e^{-\frac{|x|}{2}}\big),
\end{equation}
\begin{equation}
    I_3 = \frac{1}{2}(\bar{\lambda}_1 -1) Q(x) + \mathcal{O}\big((\bar{\lambda}_1 - 1)^2 e^{-|x|}\big).
\end{equation}
The three projection yields
\begin{equation}
    \bar{\sigma}_1\big( Q'(\bar{\lambda}_1 x), y \Lambda Q\big) + \mathcal{O}\big( |\bar{\lambda}_1 -1| + |\bar{\sigma}_1| \big) = \mathcal{O}\big(|\bar{b}_1| + \bar{\lambda}_1^{\frac{1}{2}}(\alpha^*)^{\frac{3}{2}}\big( |\bar{\lambda}_1-1| + |\bar{\sigma}_1| \big) + \|z\|_{L^2_{sol}} \big),
\end{equation}
\begin{equation}
    (\bar{\lambda}_1-1)\|\Lambda Q\|^2_{L^2} + \mathcal{O}\big( |\bar{\lambda}_1 -1|^2 + |\bar{\sigma}_1|^2 \big) = \mathcal{O}\big(|\bar{b}_1| + \bar{\lambda}_1^{\frac{1}{2}}(\alpha^*)^{\frac{3}{2}}\big((|\bar{\lambda}_1 -1|+|\bar{\sigma}_1| ) \big) + \|z\|_{L^2_{sol}} \big),
\end{equation}
\begin{equation}
    (\bar{\lambda}_1-1)\|\Lambda Q\|^2_{L^2} + \mathcal{O}\big( |\bar{\lambda}_1 -1|^2 + |\bar{\sigma}_1|^2 \big) = \mathcal{O}\big(\frac{\bar{b}_1}{16}\|Q\|^2_{L^1} + \bar{b}_1^{10} + \bar{\lambda}_1^{\frac{1}{2}}(\alpha^*)^{\frac{3}{2}}\big((|\bar{\lambda}_1 -1|+|\bar{\sigma}_1| ) \big) + \|z\|_{L^2_{sol}} \big).
\end{equation}
Thus yields
\begin{equation}
    |\bar{\lambda}_1 -1| + |\bar{\sigma}_1| \lesssim \|z\|_{L^2_{sol}} + |\bar{b}_1| + \bar{\lambda}_1^{\frac{1}{2}}(\alpha^*)^{\frac{3}{2}}(|\bar{\lambda}_1 -1| + |\bar{\sigma}_1|)
\end{equation}
and
\begin{equation}
    |\bar{b}_1| \lesssim \bar{\lambda}_1^{\frac{1}{2}}(\alpha^*)^{\frac{3}{2}}\big( |\bar{\lambda}_1 -1| + |\bar{\sigma}_1 | \big) + \| z\|_{L^2_{sol}}.
\end{equation}
Choosing $\alpha^*$ small enough, we get
\begin{equation}\label{estimate_on_lambda_sigma_b_in_proof_of_decomp_second_part}
    |\bar{\lambda}_1 -1| + |\bar{\sigma}_1| + |\bar{b}_1| \lesssim \|z\|_{L^2_{sol}}.
\end{equation}

Using the equation \eqref{equation_of_vare_inside_of_proof_of_decomposition} we estimate the $L^2$ and $L^2_{sol}$ norms of $\vare$. 
The estimate \eqref{estimate_on_difference_of_Q_in_proof_of_decomposition} yields 
\begin{equation}
    \|\bar{\lambda}_1 Q(\bar{\lambda}_1 \cdot +\bar{\sigma}_1) - Q(\cdot)\|_{L^2} \lesssim |\bar{\lambda}_1 - 1| + |\bar{\sigma}_1|.
\end{equation}

The definition of $P_{\bar{b}_1}$ in \eqref{def_of_P_b} yields
\begin{equation}
    |\bar{b}_1|\| P_{\bar{b}_1}\|_{L^2} \lesssim |\bar{b}_1|^{\frac{2-\gamma}{2}} = |\bar{b}_1|^{\frac 58}, \qquad |\bar{b}_1| \|P_{\bar{b}_1}\|_{L^2_{sol}} \lesssim |\bar{b}_1|.
\end{equation}

By the mean value theorem, the Lemma \eqref{lemma_on_the_decaying_tail} and since $\sigma_{\sharp} > \frac{t_0}{2} + \frac{x_0}{2}$, it holds
\begin{equation}
    \begin{split}
        & \| f(\sigma_{\sharp}) R(\bar{\lambda}_1\cdot + \bar{\sigma}_1) - f(\sigma_{\sharp} + \lambda_{\sharp}\bar{\sigma}_1)R(\cdot) \|_{L^2}\\
        & \lesssim |\sigma_{\sharp}|^{-\theta}\, \| R(\bar{\lambda}_1\cdot + \bar{\sigma}_1) - R(\cdot) \|_{L^2} + |\lambda_{\sharp}|\,|\sigma_{\sharp}|^{-\theta-1}\, |\bar{\sigma}_1|\\
        & \lesssim |\sigma_{\sharp}|^{-\theta} \, \big( |\bar{\lambda}_1-1| + |\bar{\sigma}_1| \big) + |\lambda_{\sharp}|\,|\sigma_{\sharp}|^{-\theta-1}\, |\bar{\sigma}_1| \lesssim |\bar{\lambda}_1 - 1| + |\bar{\sigma}_1|.
    \end{split}
\end{equation}
By a change of variables and the closeness of $(\lambda_1,\sigma_1)$ to $(1,0)$ in $(0,+\infty) \times \RR$, we get
\begin{equation}
    \bar{\lambda}_1^{\frac 12}\|z(\bar{\lambda}_1 y +\bar{\sigma}_1) \|_{L^2} \lesssim \|z\|_{L^2}, \qquad \bar{\lambda}_1^{\frac 12}\|z(\bar{\lambda}_1 y +\bar{\sigma}_1) \|_{L^2_{sol}} \lesssim \|z\|_{L^2_{sol}}.
\end{equation}
Combining the above estimates with \eqref{equation_of_vare_inside_of_proof_of_decomposition} and \eqref{estimate_on_lambda_sigma_b_in_proof_of_decomp_second_part}, we have
\begin{equation}
    \|\vare\|_{L^2} \lesssim |\bar{\lambda}_1 - 1| + |\bar{\sigma}_1| + |\bar{b}_1|^{\frac 58} + \|z\|_{L^2} \lesssim \|z\|^{\frac 58}_{L^2} \quad \text{and}\quad \|\vare\|_{L^2_{sol}} \lesssim \|z\|_{L^2_{sol}}.
\end{equation}

The proof of the estimate \eqref{estimates_on_norms_of_partial_vares_in_z_in_decomp} is similar to the previous one. Fortunately, the terms with a power of $|\bar{b}_1|$ appear to have a power of more than $1$, thus we do not have a loss of power of $\|z(t)\|_{L^2_{sol}}$ as in the estimate of $\|\vare(t)\|_{L^2}$.
\end{proof}

\begin{lemma}[Equation of $\vare_s$]
The map $s\mapsto (\lambda(s),\sigma(s), b(s))$ is  $C^1$ and $\vare$ verifies the equation
    \begin{equation}\label{equation_of_eps_s}
        \partial_s \vare = \partial_y \Big[ -\partial^2_y\vare + \vare - \Big( (W+F+\vare)^5 - (W+F)^5 \Big) \Big] - \mathcal{E}(W) + \Big(\frac{\lambda_s}{\lambda}+b \Big)\Lambda\vare+\Big(\frac{\sigma_s}{\lambda}-1 \Big)\partial_y\vare - b\Lambda\vare,
    \end{equation}
    where $\mathcal{E}(W)$ is defined in \eqref{def_of_error_term_mathcal_E}.
\end{lemma}
\begin{remark}
    The equation of $\vare$ is obtained by standard computations. The regularity 
    of the parameters is obtained by using the Cauchy–Lipschitz theorem on the differential system of the parameters. 
    See, for example, \cite[Lemma 2.7]{Combet-Martel-18}
    for a similar argument.
\end{remark}

\vspace{0.3cm}
The following lemma states the estimates on $\vare$ induced by the conservation laws.
\begin{lemma}[Mass and energy estimates]
    Under the bootstrap assumption \eqref{BS1}, for $s_0 \gg 1$, the following holds true on $s(I)$
    \begin{equation}\label{estimate_on_L2_norm_of_vare}
        \|\vare(s) \|^2_{L^2} \leq \Big\lvert\int U^2_0 - \int Q^2  \Big\rvert + C s^{-2+\gamma}+ C x^{-2\theta+1}_0,
    \end{equation}
    \begin{equation}\label{estimate_on_L2_norm_of_grad_vare}
        \lambda^{-2}\|\partial_y \vare (s)\|^2_{L^2} \lesssim \lvert E(U_0)\rvert+s^{-3+4\beta}+x^{-2\theta-1}_0 + |g(s)|.
    \end{equation}
\end{lemma}

\begin{proof}[Proof: Mass and energy estimates]
By mass conservation, we get
\begin{equation}
\int U^2_0 = \int U^2 = \|W+F+\vare\|^2_{L^2} = \|W+F\|^2_{L^2} + \|\vare\|^2_{L^2} + 2\int (W+F)\vare.
\end{equation}
Thus, we get
\begin{equation}
    \|\vare\|^2_{L^2} \leq \Big\lvert \int U^2_0 - \int Q^2\Big\rvert + \Big\lvert \int (W+F)^2 - \int Q^2\Big\rvert + 2\Big\lvert \int (W+F)\vare \Big\rvert.
\end{equation}
For the last term on the right-hand side, by the orthogonality relations \eqref{orthogonal_conditions_in_lemma}, the Young inequality, the estimate \eqref{estimates_on_norms_L2_of_F}, the property \eqref{property_on_right_on_P} on $P$ and \eqref{BS1}, yields
\begin{equation}
\begin{split}
    \Big\lvert \int (W+F)\vare \Big\rvert 
    & \lesssim \frac{1}{20}\|\vare\|^2_{L^2} + C\Big(|b|^2\int P^2_b + |r|^2\int R^2 + \int F^2 \Big) \\
    & \leq \frac{1}{20}\|\vare\|^2_{L^2} + C\Big(|b|^{2-\gamma} + |b|^2 + |r|^2 + x^{-2\theta+1}_0 \Big)\\
    & \leq \frac{1}{20}\|\vare\|^2_{L^2} + Cs^{-2+\gamma}+ Cx^{-2\theta+1}_0.
\end{split}
\end{equation}
Combining with \eqref{estimate_on_mass_W_plus_F}, we get the desired estimate on the mass of $\vare$.

\vspace{0.4cm}
The energy conservation yields
\begin{equation}
\begin{split}
    & \lambda^2 E(U_0) =\lambda^2 E(U)= E(W+F+\vare)\\ 
    & =E(W+F) + \int \partial_y(W+F)\partial_y \vare +\frac{1}{2}\int (\partial_y \vare)^2 -\frac{1}{6}\int \big( (W+F+\vare)^6-(W+F)^6 \big).
\end{split}
\end{equation}
We write, using the equation of $-Q''+Q-Q^5 = 0$ and the orthogonality relations \eqref{orthogonal_conditions_in_lemma} on $\vare$,
\begin{equation}\label{proof_of_L2_grad_epsilon_rewriting1}
\begin{split}
    & \int \partial_y (W+F)\partial_y\vare -\frac{1}{6}\int \big((W+F+\vare)^6 - (W+F)^6\big)= b\int \partial_y P_b \; \partial_y \vare + r \int \partial_y R \; \partial_y \vare\\
    &  - \int (\partial^2_y F + F^5 )\vare - \frac{1}{6}\int \Big((W+F+\vare)^6-(W+F)^6- 6Q^5\vare - 6F^5\vare \Big).
\end{split}
\end{equation}
By Cauchy-Schwarz and Young's inequalities, combined with \eqref{estimates_on_r} and \eqref{BS1}, it holds
\begin{equation}
    b \Big\lvert\int\partial_y P_b\; \partial_y \vare \Big\rvert + |r| \Big\lvert\int \partial_y R \; \partial_y \vare \Big\rvert \leq \frac{1}{20}\|\partial_y \vare\|^2_{L^2} + Cs^{-2}.
\end{equation}
Proceeding similarly and then using the estimate \eqref{estimates_on_norms_L2_of_F} on $L^2$ norm of $\partial_yF$ yields
\begin{equation}
    \Big\lvert \int \partial^2_y F \; \vare \Big\lvert \lesssim \frac{1}{20}\|\partial_y \vare\|^2_{L^2} + C s^{-2\beta}x^{-2\theta-1}_0.
\end{equation}
The Hölder's, Young's and Gagliardo-Nirenberg inequalities, then combined with \eqref{estimates_on_norms_L2_of_F} and \eqref{BS2} give
\begin{equation}
\begin{split}
    \Big\lvert \int F^5 \vare \Big\rvert 
    & \lesssim \Big( \int \vare^6\Big)^{\frac{1}{6}}\Big(\int F^6\Big)^{\frac{1}{6}} \lesssim \int \vare^6 +\int F^6 
    \lesssim \|\vare\|^4_{L^2}\|\partial_y\vare\|^2_{L^2} + \|F\|^4_{L^2}\| \partial_y F\|^2_{L^2} \\
    & \leq \frac{1}{20}\|\partial_y \vare \|^2_{L^2} + C s^{-2\beta}x^{-6\theta-1}_0.
\end{split}
\end{equation}
In order to estimate the last term on the right-hand side in \eqref{proof_of_L2_grad_epsilon_rewriting1}, by interpolation, \eqref{BS1}, the estimates \eqref{estimates_on_norms_L_infty_of_F_with_exp}, \eqref{estimates_on_norms_L_infty_of_F}, and finally Gagliardo-Nirenberg and \eqref{BS2}, we have
\begin{equation}
    \Big\lvert \int \Big((W+F+\vare)^6-(W+F)^6- 6Q^5\vare - 6F^5\vare \Big\lvert \leq \frac{1}{20}\|\partial_y \vare\|^2_{L^2} + Cs^{-2}.
\end{equation}
Combining previous estimates with \eqref{estimate_on_energy_W_plus_F}, we get the desired estimate.
\end{proof}

\begin{lemma}[Modulation estimates]\label{lemma_mod_estimates_after_decomp}
    Assume \eqref{BS1} and 
    \begin{equation}
        \sigma(\tau(s)) > \frac{2}{3}\tau(s)+\frac{2}{3}x_0, \quad \forall s\in s(I).
    \end{equation}
    Then, for $s_0\gg1$, the following estimates hold on $s(I)$
    \begin{equation}\label{estimate_on_m_mod_estim}
        |\vec{m}| \lesssim \| \vare \|_{L^2_{sol}} + s^{-2},
    \end{equation}
    \begin{equation}\label{estimation_on_bs_mod_estim}
        |b_s| \lesssim \|\vare\|^2_{L^2_{sol}} + s^{-2},
    \end{equation}
    \begin{equation}\label{estimation_on_gs_mod_estim}
        \lambda^2|g_s| \lesssim s^{-4+2\beta}+\|\vare \|^2_{L^2_{sol}} + s^{-1}\|\vare \|_{L^2_{sol}}.
    \end{equation}
\end{lemma}
\begin{proof}
    We will differentiate with respect to $s$ the three orthogonality relations in \eqref{orthogonal_conditions_in_lemma}.

    Differentiating the orthogonality with $\Lambda Q$, using the equation of $\vare_s$ in \eqref{equation_of_eps_s} and the relation from parity properties $(Q', \Lambda Q) = 0$, we have
    \begin{equation}
        \begin{split}
            \|\Lambda Q\|^2_{L^2}\Big(\frac{\lambda_s}{\lambda}+b \Big) + \big(\vare, \mathcal{L}(\Lambda Q)'\big) 
            & = 
             \big( \vec{m}\cdot \vec{M}\vare, \Lambda Q\big)
            - \big( \mathcal{R}, \Lambda Q\big) - b \big( \Lambda \vare, \Lambda Q \big)\\ 
            & + \big( (W+F+\vare)^5 - (W+F)^5 - 5Q^4 \vare, (\Lambda Q)'\big).
        \end{split}
    \end{equation}
    The estimate \eqref{BS1} yields
    \begin{equation}
        \big\lvert \big( \vec{m}\cdot \vec{M} \vare, \Lambda Q \big) \big\rvert \lesssim |\vec{m}| \| \vare\|_{L^2_{sol}},  \qquad \big\lvert \big(b \Lambda \vare, \Lambda Q \big)\big\rvert \lesssim s^{-2} + \|\vare\|^{2}_{L^2_{sol}}.
    \end{equation}
    By interpolation, \eqref{estimates_on_r}, \eqref{estimates_on_norms_L_infty_of_F_with_exp}, \eqref{BS1}, we have 
    \begin{equation}\label{estimate_1_in_proof_of_mod_estimates}
        \begin{split}
            & \big( (W+F+\vare)^5 - (W+F)^5 - 5Q^4 \vare, (\Lambda Q)'\big)
            \lesssim \int Q^3 |bP_b + rR + F||\vare| (\Lambda Q)'\\
            & + \int (bP_b + rR + F)^4 |\vare| (\Lambda Q)' + \int |(W+F)^3|\vare^2 (\Lambda Q)' + \int |\vare|^5 (\Lambda Q)'
            \lesssim s^{-1}\|\vare\|_{L^2_{sol}} + \|\vare\|^2_{L^2_{sol}}.
        \end{split}
    \end{equation}
    Where the estimate term of order 5 is provided by the following combined with \eqref{estimates_on_norms_of_vares_in_z_in_decomp}
    \begin{equation}
        \int \vare^5 (\Lambda Q)' \lesssim \|\vare \|^3_{L^{\infty}}\int \vare^{2}e^{-\frac{|y|}{2}}\lesssim \|\vare_y\|^{\frac{3}{2}}_{L^2}\|\vare\|^{\frac{3}{2}}_{L^2}\| \vare\|^2_{L^2_{sol}}.
    \end{equation}
    
    Finally combining with the estimate \eqref{estimate_on_norms_L_2_B_of_R} and that $|(\mathcal{R}, \phi)| \lesssim \|\mathcal{R}\|_{L^2_B}$ for $\phi \in \mathcal{Y}$, we get
    \begin{equation}\label{estimate_on_first_component_of_m_in_proof_of_mod_estimates}
        \Bigg\lvert \Big( \frac{\lambda_s}{\lambda} + b\Big) + \frac{\big( \vare,\mathcal{L}(\Lambda Q)' \big)}{\|\Lambda Q\|^2_{L^2}}\Bigg\rvert \lesssim |\vec{m}|\big( s^{-1} + \|\vare\|_{L^2_{sol}} \big) + |b_s| + s^{-2} + \|\vare\|^2_{L^2_{sol}}.
    \end{equation}
     Arguing similarly for the second orthogonality, using in addition the relations $(Q', y \Lambda Q) = \|\Lambda Q\|^2_{L^2}$ and by parity $(\Lambda Q, y \Lambda Q) =0 $, we get
    \begin{equation}\label{estimate_on_second_component_of_m_in_proof_of_mod_estimates}
        \Bigg\lvert \Big( \frac{\sigma_s}{\lambda} -1\Big) + \frac{\big( \vare,\mathcal{L}(y\Lambda Q)' \big)}{\|\Lambda Q\|^2_{L^2}}\Bigg\rvert \lesssim |\vec{m}|\big( s^{-1} + \|\vare\|_{L^2_{sol}} \big) + |b_s| + s^{-2} + \|\vare\|^2_{L^2_{sol}}.
    \end{equation}

    For the third orthogonality we write
    \begin{equation}
    \begin{split}
        \frac{d}{ds}\big( \vare, Q \big)
        & = \big( \vare, \mathcal{L}Q'\big) + \big( \big( (W+F+\vare)^5- (W+F)^5 - 5Q^4\vare \big), Q' \big)\\
        &- \big( \mathcal{E}(W), Q \big) + \big( \vec{m}\cdot\vec{M} \vare, Q \big) - b\big(\Lambda \vare, Q \big).
    \end{split}
    \end{equation}
    Since $\mathcal{L}Q' = 0$, the first term is zero. The second term we estimate in the similar way as for \eqref{estimate_1_in_proof_of_mod_estimates}.
    
    From the equation of $\mathcal{E}(W)$ and since $\big( \Lambda Q, Q \big) = \big( \partial_y Q,Q\big) =0$, we get 
    \begin{equation}
        \big(\mathcal{E}(W),Q \big) = -\big( \mathcal{R}, Q \big).
    \end{equation}
    We also remark that 
    \begin{equation}
        \big( \Lambda \vare, Q\big) = -\big( \vare, \Lambda Q \big) =0 \quad\text{and}\quad \big( \vec{m}\cdot\vec{M}\vare, Q \big) = -\Big( \frac{\sigma_s}{\lambda}-1\Big)\big( \vare, Q' \big).
    \end{equation}

    Thus, we have
    \begin{equation}
        \big(\mathcal{R}, Q \big) = \big( (W+F+\vare)^5 - (W+F)^5 - 5Q^4\vare,
        \, Q' \big) - \Big( \frac{\sigma_s}{\lambda} -1\Big)\big( \vare, Q' \big).
    \end{equation}
    Using this rewriting and the estimates \eqref{estimate_on_scalar_prod_R_Q}, \eqref{estimate_1_in_proof_of_mod_estimates}, we get
    \begin{equation}\label{estimate_on_second_part_in_estimate_on_scalar_prod_of_R_and_Q}
        \bigg\lvert b_s+2b^2 - \frac{4c_0}{\int Q}\lambda^{\frac{1}{2}}\sigma^{-\theta}\Big( \frac{3}{2}\frac{\lambda_s}{\lambda} + \theta \frac{\sigma_s}{\sigma}\Big) \bigg\rvert \lesssim s^{-4+2\beta}+  s^{-1}\big( \|\vare\|_{L^2_{sol}}+|\vec{m}| \big) + |\vec{m}|\|\vare\|_{L^2_{sol}}+\|\vare\|^2_{L^2_{sol}}.
    \end{equation}
    Thus
    \begin{equation}
        |b_s| \lesssim s^{-2} + |\vec{m}|(s^{-1} + \| \vare\|_{L^2_{sol}}) + \|\vare \|^2_{L^2_{sol}}.
    \end{equation}

    The previous estimate on $|b_s|$ and the estimates \eqref{estimate_on_first_component_of_m_in_proof_of_mod_estimates}, \eqref{estimate_on_second_component_of_m_in_proof_of_mod_estimates} and \eqref{estimates_on_norms_of_vares_in_z_in_decomp}, yields
    \begin{equation}
        \begin{split}
            |\vec{m}| \lesssim |\vec{m}|\big( s^{-1} + \|\vare \|_{L^2_{sol}} \big) + s^{-2} + 2\| \vare\|_{L^2_{sol}} + \|\vare \|^2_{L^2_{sol}}.
        \end{split}
    \end{equation}
    Thus, for $s_0 \gg1$, the estimate \eqref{estimate_on_m_mod_estim} holds true.
    Inserting \eqref{estimate_on_m_mod_estim} in the estimate on $|b_s|$ gives the estimate \eqref{estimation_on_bs_mod_estim}. 

    Combining \eqref{estimate_on_second_part_in_estimate_on_scalar_prod_of_R_and_Q} with \eqref{estimate_on_scalar_prod_R_Q}, we have
    \begin{equation}
        \big|\big( \mathcal{R},Q \big)\big| \lesssim s^{-4+2\beta} + |\vec{m}|\big( s^{-1}+\|\vare\|_{L^2_{sol}} \big) + s^{-1}\|\vare\|_{L^2_{sol}} + \|\vare\|^2_{L^2_{sol}}.
    \end{equation}
    Thus, inserting the previous estimate in \eqref{estimate_on_gs_with_prod_scal_R_Q}, yields
    \begin{equation}
       \lvert \lambda^2 g_s \rvert \lesssim s^{-4+2\beta} + s^{-1}\|\vare\|_{L^2_{sol}} + \|\vare\|^2_{L^2_{sol}}.
    \end{equation}
    
\end{proof}

\subsection{Bootstrap estimates}
Fix two smooth functions $\vphi$ and $\psi$ such that
\begin{equation}
    \vphi(y)=\begin{cases}
        e^y, \quad y<-1,\\
        1+y, \quad -\frac{1}{2}<y<\frac{1}{2},\\
        y,  \quad y>2
    \end{cases}
    \qquad \text{and} \qquad \vphi'(y)>0, \qquad \forall y \in \RR,
\end{equation}
\begin{equation}
    \psi(y) = \begin{cases}
        e^{2y}, \quad y<-1,\\
        1, \quad y>-\frac{1}{2}
    \end{cases}
    \qquad \text{and} \qquad \psi'(y)>0, \qquad \forall y \in \RR.
\end{equation}

Let $B>100$ to be fixed later. Define on $\RR$
\begin{equation}
    \vphi_B(y):= \vphi\Big(\frac{y}{B} \Big) \qquad \text{and}\qquad \psi_B(y):= \psi\Big(\frac{y}{B} \Big).
\end{equation}

Consider the following norm on $\vare$
\begin{equation}
    \mathcal{N}_B(s):= \Big(\int\vare_y^2(s,y)\psi_B(y) \,dy + \int\vare^2(s,y)\vphi_B(y) \,dy \Big)^{\frac{1}{2}}.
\end{equation}
From the definition of the $L^2_{loc}$-norm, it holds
\begin{equation}\label{relation_L2loc_with_N_B}
    \|\vare\|^2_{L^2_{sol}}+ \|\vare_y\|^2_{L^2_{sol}} + \int \vphi'_B\,\vare^2  \lesssim \mathcal{N}^2_B
    \qquad\text{and}\qquad \|\vare\|^2_{L^2_{sol}} \leq  C\,B\int \vphi'_B \vare^2 
\end{equation}

In addition of \eqref{BS1}, we will work under the following assumptions
\begin{equation}\tag{BS2}\label{BS2}
        \mathcal{N}_B(s) \;\leq\; s^{-\frac{5}{4}}\,, \qquad \|\vare(s)\|_{H^1} \;\leq\; \alpha^*. 
\end{equation}

\begin{remark}
    In particular, it allows us to control the full $L^{\infty}$ norm of $\vare$ in terms of $\alpha^*$
    \begin{equation}\label{L_infty_norm_of_vare_control}
        \|\vare\|_{L^{\infty}} \lesssim \|\vare\|_{H^1} \lesssim \delta(\alpha^*).
    \end{equation}
\end{remark}

\begin{lemma}[Consequences of bootstrap assumptions]
Assuming $g(s_0)=0$ and \eqref{BS1}, \eqref{BS2} on $s(I)$, it holds
\begin{equation}\label{conseq_of_BS_on_m_on_bs}
    |\vec{m}| \lesssim s^{-\frac{5}{4}} \qquad\text{and}\qquad  |b_s|\lesssim s^{-2},
\end{equation}
\begin{equation}\label{conseq_of_BS_on_gs}
    |g_s| \lesssim s^{2\beta-\frac{9}{4}}+s^{4\beta-4}.
\end{equation}
\begin{equation}\label{conseq_of_BS_on_g_and_hs}
    \lvert g(s) \rvert \lesssim s^{2\beta - 1-4\rho} \qquad \text{and}\qquad \lvert h_s \rvert \lesssim s^{-\frac{\beta}{2}-1-4\rho},
\end{equation}
where $\rho$ satisfies \eqref{condition_on_rho}.
\end{lemma}
\begin{proof}
Combining \eqref{BS2} with the estimates in Lemma \ref{lemma_mod_estimates_after_decomp}, yields \eqref{conseq_of_BS_on_m_on_bs}, \eqref{conseq_of_BS_on_gs}.

The first estimate in \eqref{conseq_of_BS_on_g_and_hs} comes from integrating \eqref{conseq_of_BS_on_gs} and the choice of $\rho$.

From the estimate \eqref{estimate_on_hs_with_g}, the obtained estimate on $g$ and the definitions of $\rho$, we get the estimate on $h_s$.
\end{proof}
\section{Energy estimates}\label{section_energy_estimates}
Assume $\vare$ as in Lemma \ref{lemma_on_decomposition_around_Q} and 
$(\lambda,\sigma, b, \vare)$ satisfies \eqref{BS1}-\eqref{BS2} on some time interval $[s_0,s^*]$ with $s^*\geq s_0$.

We define the mixed energy-virial functional 
    \begin{equation*}
        \mathcal{F} \;=\; \int \Big[ (\partial_y \vare )^2 \psi_B + \vare^2 \varphi_B - \frac{1}{3}\Big( (W+F+\vare)^6-(W+F)^6-6(W+F)^5\vare \Big) \psi_B \Big],
    \end{equation*}
Set $j = \frac 52$ (other values of $j \in (2,3)$ are possible) and let $k$ satisfy
\begin{equation}\label{choice-of-k}
    k > \frac{2(1+j)}{1-\beta}.
\end{equation}
Define $\phi_k$ a smooth non-decreasing function such that
\begin{equation*}
    \phi_k(y)  : =\begin{cases}
        0,  \quad \text{for  } y\leq0,\\
        y^k,  \quad \text{for  } y \geq 1
    \end{cases} \qquad \text{and} \qquad y\phi'_k = k\phi_k \,\text{ sur } \, [0,1].
\end{equation*}
The choice of $k$ is explained later in the proof of the Proposition \ref{Proposition_Monot_formula}. We observe that $k\to +\infty$ corresponds to $\nu \to \frac 12^+$. We also have a lower bound $k > 4(1+j)$, since $\nu < 1$. 
We consider the functional
\begin{equation*}
        \mathcal{H} \;=\; s^j \, \mathcal{F} + \lambda^k \int \vare^2 \phi_k.
    \end{equation*}
\begin{remark}
In the paper \cite{MMPIII}, the functional used at this stage is simply $s^{j}\mathcal{F}$.
However, differentiating this functional makes a delicate term of the form
$s^{-1}\int \vare^2\,\vphi_B$
appear because of the scaling term $-b \Lambda \vare$ in the equation \eqref{equation_of_eps_s}
of $\vare$.
The control of this term in \cite{MMPIII} is based on an additional bootstrap assumption that
is closed using separately a functional of the form $\lambda^k \int \vare^2 \phi_k$
(with a fixed value of $k$, $k=10$).
In the present work, we add directly to the functional $\mathcal{H}$ the term $\lambda^{k}\int \vare^2\,\phi_k$, which simplifies the proof and makes it more flexible.
Note that the term $\lambda^{k}\int \vare^2\,\phi_k$ is essentially critical for scaling
since the $L^2$ norm is invariant by scaling and $\lambda^k \phi_k = (\lambda y) ^k$
for $y\geq 1$.
\end{remark}

\vspace{0.2cm}
\begin{proposition}(Monotonicity formula)\label{Proposition_Monot_formula}\\
The following estimates hold on $[s_0,s^*]:$
\begin{itemize}
    \item Lyapounov control : There exists $0<\mu<1$ such that
    \begin{equation}\label{Lyapounov_control_of_H}
    \begin{split}
        & \frac{d}{ds}\big[\mathcal{H}\big]+s^{j}\frac{\mu}{8}\int\vphi'_B(\vare_y^2+\vare^2) +2s^{j}\int_{y<-\frac B2}\psi'_B\,(\partial^2_y\vare)^2\\
        & +s^{j}\int_{y<-\frac B2} \psi'_B\,(\partial_y \vare)^2 +\frac{\lambda^k}{16}\int \phi_k\,\big(\vare^2 + (\partial_y\vare)^2 \big)
        \lesssim \lambda^ks^{ - \frac 52} + s^{j-4}.
    \end{split}
    \end{equation}
    \item Coercivity of $\mathcal{F} $ :  
    \begin{equation}\label{coercivity_of_mathcal_F}
        \mathcal{F}\;\gtrsim\;\mathcal{N}^2_B.
    \end{equation}
\end{itemize}
\end{proposition}


\begin{remark}
Before beginning the proof, we will present some technical results that will be useful in this section.

By the definitions of $\vphi,\psi$ and $\phi_k$, it holds
\begin{equation}\label{estimates_between_psiB_and_phi_B}
\begin{cases}
    e^{-|y|} + e^{-|y|}\phi_k \lesssim \vphi'_B, \qquad \text{for } y>-\frac{B}{2},\\
    |y|\psi_B \lesssim B \vphi_B, \qquad \text{for } y \in \RR,\\
    \psi_B \lesssim B \vphi'_B, \qquad \text{for } y \in \RR,\\
    e^{-\frac{y}{B}}\,\phi^2_k \leq B^{2k}\,e^{-k}\,k^{4k}\,\phi_k, \qquad \text{for } y >0.\\
\end{cases}
\end{equation}
The following relation will be useful throughout the proof
\begin{equation}\label{relation_on_partial_eps_and_order_5}
\begin{split}
    & \partial_y \Bigg( \frac{1}{6} \Big[(W+F+\vare)^6 - (W+F)^6 - 6 (W+F)^5 \vare \Big] \Bigg) \\
    & = \partial_y (W+F) \Big[ (W+F+\vare)^5 - (W+F)^5 - 5 (W+F)^4 \vare\Big] + \partial_y \vare \Big[ (W+F+\vare)^5 - (W+F)^5 \Big].
\end{split}
\end{equation}
\end{remark}

We will keep track of the order in which the constants are fixed throughout the proof. First, we fix $B$ to be large enough. Then, we set $\alpha^*$ be small enough and smaller than the one appearing in Lemma \ref{lemma_on_decomposition_around_Q}. Finally, we fix the value of $s_0$ to be large enough.

\vspace{0.4cm}
\begin{proof}[Proof of the Proposition \ref{Proposition_Monot_formula}]
First, we show that there exists $0<\mu<1$ such that
\begin{equation}\label{mathcalF_estim_on_all_dds}
\begin{split}
    s^{-j}\frac{d}{ds}\big[s^j \mathcal{F} \big]  + 2\int_{y<-\frac B2} \psi'_B\,(\partial^2_y \vare)^2
    +\frac\mu 8 \int \vphi'_B \big( \vare^2 + (\partial_y \vare)^2 \big) + \int_{y<-\frac B2} \psi'_B\,(\partial_y \vare)^2 \lesssim B s^{-1}\int \vphi_B \,\vare^2 + B^2 s^{-4}.
\end{split}
\end{equation}

Using the equation \eqref{equation_of_eps_s} of $\vare_s$, we remark that
\begin{equation}
    \begin{split}
        s^{-j}\frac{d}{ds}\big[s^j\, \mathcal{F}\big]  
        = & \frac{d}{ds}\big[\mathcal{F}\big]+ \frac js\mathcal{F} = 2\int G_B(\vare)\big( \vare_s - \frac{\lambda_s}{\lambda}\Lambda \vare \big) + 2\frac{\lambda_
        s}{\lambda}\int \Lambda \vare\, G_B(\vare) \\
        & + \frac{j}{s}\mathcal{F} - 2 \int \partial_s(W+F)\, \big[ (W+F+\vare)^5 - (W+F)^5 - 5(W+F)^4 \big] \psi_B \\
        & =:\mathcal{F}_{(1)}+\mathcal{F}_{(2)}+\mathcal{F}_{(3)}+\mathcal{F}_{(4)}+\mathcal{F}_{(5)},
    \end{split}
\end{equation}
where
\begin{equation}
\begin{split}
    \mathcal{F}_{(1)} &= 2 \int G_B(\vare) \, \partial_y\Big[ -\partial^2_{y}\vare + \vare - \big( (W+F+\vare)^5 - (W+F)^5 \big)\Big],\\
    \mathcal{F}_{(2)} &= -2\int G_B(\vare)\mathcal{E}(W),\\
    \mathcal{F}_{(3)}&= 2 \Big(\frac{\sigma_s}{\lambda}-1\Big) \int G_B(\vare) \partial_y \vare,\\
    \mathcal{F}_{(4)} &= 2\frac{\lambda_s}{\lambda}\int G_B(\vare)\, \Lambda \vare + \frac{j}{s}\mathcal{F}, \\
    \mathcal{F}_{(5)} &= -2 \int \psi_B\, (W+F)_s \Big[ (W+F+\vare)^5 - (W+F)^5 - 5(W+F)^4\Big]
\end{split}
\end{equation}
and
\begin{equation}
    G_B(\vare) = -\partial_y(\partial_y \vare \, \psi_B) + \vare \, \vphi_B - \big( (W+F+\vare)^5 - (W+F)^5\big)\psi_B.
\end{equation}

\textit{Estimate on $\mathcal{F}_{(1)}$:} 
We rewrite the term $\mathcal{F}_{(1)}$ in order to obtain a more manageable formula.
The following computations were previously made in Step 3 in the proof of Proposition 3 in \cite{MMPI}.

\vspace{0.4cm}
Note that
\begin{equation}
    \begin{split}
        \mathcal{F}_{(1)} 
        & = 2 \int \psi_B \,\partial_y\Bigl[- \partial^2_y \vare + \vare  - \bigl( (W+F+\vare)^5 - (W+F)^5 \bigr)  \Bigr]\Bigl[ - \partial^2_y \vare + \vare - \bigl( (W+F+\vare)^5 - (W+F)^5 \bigr)  \Bigr]\\
        & + 2 \int \partial_y \Big[ -\partial^2_y \vare+ \vare - \big( (W+F+\vare)^5 - (W+F)^5 \big) \Big]\,\Big[ -\partial_y \vare\, \psi'_B + \vare(\vphi -\psi_B) \Big] =: f^{(1)} + f^{(2)}.
    \end{split}
\end{equation}
By integration by parts, rewrite two terms as follows
\begin{equation}
\begin{split}
    f^{(1)} 
    = &- \int \psi'_B\Big[ - \partial^2_y \vare + \vare - \big( (W+F+\vare)^5 - (W+F)^5\big) \Big]^2\\
    = &-\int \psi'_B \Big[ (-\partial^2_B \vare)^2 - 2 \partial^2_y \vare \, \vare \, + \vare^2 \Big]  \\
    &- \int \psi'_B \Big(\big[ - \partial^2_y \vare + \vare - \big((W+F+\vare)^5 - (W+F)^5 \big)  \big]^2 - \big[ -\partial^2_y \vare + \vare \big]^2 \Big)\\
    = & - \int \psi'_B \Big[ (\partial^2_y \vare)^2 + 2 (\partial_y \vare)^2 \Big] + \int \vare^2 \big( \psi'''_B - \psi'_B \big)\\
    & - \int \psi'_B \Big(\big[ - \partial^2_y \vare + \vare - \big((W+F+\vare)^5 - (W+F)^5 \big)  \big]^2 - \big[ -\partial^2_y \vare + \vare \big]^2 \Big)
\end{split}
\end{equation}
and 
\begin{equation}
    \begin{split}
        f^{(2)} = 
        & -2 \int  \big( - \partial_y^2 \vare + \vare  \big)\big( -\psi'_B \, \partial_y^2\vare - \psi''_B\, \partial_y \vare + (\vphi_B - \psi_B)\, \partial_y \vare + (\vphi'_B - \psi'_B)\,\vare  \big)\\
        & -2 \int (\vphi_B - \psi_B)\, \partial_y \big[ (W+F+\vare)^5 - (W+F)^5 \big]\, \vare  +2\int  \psi'
        _B\, \partial_y \big[ (W+F+\vare)^5 - (W+F)^5 \big]\,\partial_y \vare\\
        & =: f^{(2,1)} + f^{(2,2)} + f^{(2,3)}.
    \end{split}
\end{equation}

Then, we rewrite the three terms as follows
\begin{equation}
    f^{(2,1)} = -2 \Big[\int \psi'_B \, (\partial^2_y \vare )^2  + \int (\partial_y \vare)^2 \,\big( \frac{3}{2}\vphi'_B - \frac{1}{2}\psi'_B - \frac{1}{2}\psi'''_B \big) + \int \vare^2 \big( \frac{1}{2}(\vphi'_B-\psi'_B)-\frac{1}{2}(\vphi'''_B - \psi'''_B) \big) \Big]. 
\end{equation}
By integration by parts and by \eqref{relation_on_partial_eps_and_order_5}, we have
\begin{equation}
\begin{split}
    f^{(2,2)} 
    = &2 \int (\vphi_B - \psi_B)\, \big[ (W+F+\vare)^5 - (W+F)^5 \big]\, \partial_y \vare\\
    &+ 2 \int (\vphi'_B -\psi'_B)\,\big[(W+F+\vare)^5 - (W+F)^5 \big]\, \vare\\
    = & -\frac{1}{3}\int (\vphi'_B - \psi'_B) \big[ (W+F+\vare)^6 - (W+F)^6 - 6 (W+F+\vare)^5 \,\vare \big]\\
    & -2 \int (\vphi_B - \psi_B)\, \partial_y (W+F)\, \big[ (W+F+\vare)^5 - (W+F)^5 - 5 (W+F)^4 \vare \big],
\end{split}
\end{equation}
\begin{equation}
    f^{(2,3)} = 10 \int \psi'_B \, \partial_y \vare \, \Big[ \partial_y (W+F)\, \big( (W+F+\vare)^4 - (W+F)^4 \big) + \partial_y \vare \, (W+F+\vare)^4 \Big].
\end{equation}
Combining those rewritings, we get
\begin{equation}
\begin{split}
    f^{(2)} 
    = & - \int \psi'_B \, \big[ (\partial^2_y \vare)^2 + 2 (\partial_y \vare)^2  \big] + \int \vare^2 (\psi'''_B - \psi'_B)\\
    & - \int \psi'_B \, \Big( \big[-\partial^2_y \vare + \vare - \big((W+F+\vare)^5 - (W+F)^5 \big) \big]^2 - \big[-\partial^2_y \vare + \vare \big]^2 \Big).
\end{split}
\end{equation}
Finally, we have the following rewriting of $\mathcal{F}_{(1)}$
\begin{equation}
    \begin{split}
        \mathcal{F}_{(1)} = 
        & -3 \int \psi'_B \, (\partial^2_y \vare)^2  - \int \big( 3\vphi'_B + \psi'_B - \psi'''_B \big)\, (\partial_y \vare)^2 - \int (\vphi'_B - \vphi'''_B)\, \vare^2\\
        & -\frac{1}{3}\int (\vphi'_B - \psi'_B) \big[ (W+F+\vare)^6 - (W+F)^6 - 6 (W+F+\vare)^5\,\vare \big]\\
        & -2 \int (\vphi_B - \psi_B)\, \partial_y(W+F)\, \big[ (W+F+\vare)^5 - (W+F)^5 - 5 (W+F)^4 \vare \big]\\
        & +10 \int \psi'_B \, \partial_y \vare\, \big[ \partial_y (W+F)\big( (W+F+\vare)^4 -(W+F)^4 \big) + \partial_y \vare \, (W+F+\vare)^4 \big]\\
        & - \int \psi'_B\, \Big( \big[ -\partial^2_y \vare + \vare - \big( (W+F+\vare)^5 -(W+F)^5  \big)\big]^2 - \big[ -\partial^2_y\vare + \vare\big]^2 \Big).
    \end{split}
\end{equation}
We will study $\mathcal{F}_{(1)}$ separately on two domains $\{ y > -\frac{B}{2}\}$ and $\{ y < -\frac{B}{2}\}$.

\textit{Estimate on $\mathcal{F}_{(1)}^{(>)}$:}
We rewrite the term on the domain $\{y>-\frac{B}{2}\}$ as
\begin{equation}
\begin{split}
    &\mathcal{F}_{(1)}^{(>)} = 
     - \int_{y>-\frac B2} \vphi'_B\, \big[ 3(\partial_y\,\vare)^2 + \vare^2 - 5Q^4\,\vare^2 + 20y Q'Q^3\vare^2 \big]\\
    & -\frac{1}{3}\int_{y>-\frac B2} \vphi'_B\, \big[ (W+F+\vare)^6 - (W+F)^6 - 6 (W+F+\vare)^5 \vare + 15 Q^4 \vare^2 \big]  \\
    & +20\int_{y>-\frac B2}  \vphi'_B \, y Q' Q^3 \vare^2- 2\int_{y>-\frac B2} (\vphi_B - \psi_B)\, \partial_y(W+F)\,\big[(W+F+\vare)^5 - (W+F)^5 -5(W+F)^4 \vare \big]\\
    & =:\mathcal{F}_{(1,1)}^{(>)} + \mathcal{F}_{(1,2)}^{(>)}+\mathcal{F}_{(1,3)}^{(>)}.
\end{split}
\end{equation}
For the first term, we rely on the coercivity property of the virial quadratic form under the orthogonality conditions \eqref{orthogonal_conditions_in_lemma}. It is a variant of \cite[Lemma 3.5]{Combet-Martel-17} and of \cite[Lemma 3.4]{MMPI}, based on \cite[Proposition 4]{Martel-Merle-00}. The proof of the following result can be found in the Appendix A.
\begin{lemma}[Localised virial estimate]\label{lemma-localized-virial-estimate}
    There exists $0<\mu<1$ such that 
    \begin{equation}\label{estimate-localized-virial}
        \int_{y>-\frac B2} \big[ 3(\partial_y\vare)^2 + \vare^2 - 5 Q^4 \vare^2 + 20y Q' Q^3 \vare^2\big] \geq \frac{\mu}{2}\int_{y>-\frac B2} \big( \vare^2 + (\partial_y \vare)^2 \big) - \frac{C}{\mu}\frac{1}{B^5}\int \vare^2\,e^{-\frac{|y|}{2}}.
    \end{equation}
\end{lemma}

\vspace{0.4cm}
The virial estimate yields the following estimate
\begin{equation}
    \mathcal{F}_{(1,1)}^{(>)} + \frac \mu 2  \int_{y>-\frac B2} \vphi'_B\, \big( \vare^2 + (\partial_y \vare)^2 \big) \lesssim \frac{1}{B^4}\int \vphi'_B \, \vare^2.
\end{equation}

To estimate the second term, we use the following relation
\begin{equation}
\begin{split}
    & (W+F+\vare)^6+(W+F)^6-6(W+F+\vare)^5\vare + 15Q^4 \vare^2 \\
    & =\big[ (W+F+\vare)^6 - (W+F)^6 - 6(W+F)^5\vare -15(W+F)^4\vare^2 \big]\\
    & 
    -6\vare \big[ (W+F+\vare)^5-(W+F)^5 - 5(W+F)^4\vare \big] - 15\vare^2\big[ (W+F)^4 - Q^4\big].
\end{split}
\end{equation}
Thus, we get 
\begin{equation}
    \begin{split}
        \mathcal{F}_{(1,2)}^{(>)} \lesssim \int_{y>-\frac B2} \vphi'_B  \big[ |W+F|^3|\vare|^3  + \vare^6 + \vare^2 \big( |bP_b+rR + F|^4 + Q^4|bP_b +rR + F| \big) \big]
        \leq \frac{\mu}{2^{10}}\int_{y>-\frac B2} \vphi'_B \vare^2.
    \end{split}
\end{equation}

We rewrite the third term as follows
\begin{equation}
    \begin{split}
        \mathcal{F}_{(1,3)}^{(>)} =
        & -2\int_{y>-\frac B2} (\vphi_B-\psi_B)\,\partial_y(W+F)\,\big[ (W+F+\vare)^5 - (W+F)^5 - 5(W+F)^4\vare -10(W+F)^3\vare^2 \big]\\
        & +20\int_{y>-\frac B2} (y \vphi'_B - \vphi_B + \psi_B) Q'Q^3 \vare^2\\
        & + 20\int_{y>-\frac B2} (\vphi_B-\psi_B)\,\vare^2\,\big[ Q'Q^3 - \partial_y(W+F)(W+F)^3\big] =: \mathcal{F}_{(1,3,1)}^{(>)} + \mathcal{F}_{(1,3,2)}^{(>)} + \mathcal{F}_{(1,3,3)}^{(>)}.
    \end{split}
\end{equation}

The first term in the rewriting of $\mathcal{F}_{(1,3)}^{(>)}$ is estimated using \eqref{estimates_on_norms_L_infty_of_F}, \eqref{estimate_on_Q_b_and_its_derivs}, \eqref{estimates_on_r}, \eqref{L_infty_norm_of_vare_control} and \eqref{estimates_between_psiB_and_phi_B}
\begin{equation}
\begin{split}
    \mathcal{F}_{(1,3,1)}^{(>)} 
    & \lesssim \int_{y>-\frac B2} |\vphi_B-\psi_B||\partial_y(W+F)| \big( |W+F|^2 |\vare|^3 + |\vare|^5 \big)\\
    & \lesssim\delta(\alpha^*)\int_{y>-\frac B2} e^{-|y|}|\vphi_B-1|\vare^2 \lesssim \delta(\alpha^*)\int_{y>-\frac B2} \vphi'_B \vare^2.
\end{split}
\end{equation}

We remark that $\supp(y\vphi'_B -\vphi_B+\psi_B)\cap\{y>-\frac{B}{2}\} = \{y>\frac{B}{2} \}$ and on this domain $y\vphi'_B -\vphi_B +\psi_B \lesssim B\vphi'_B$. Therefore
\begin{equation}
    \mathcal{F}_{(1,3,2)}^{(>)} = 20\int_{y>\frac{B}{2}} Q'Q^3 \vare^2 \lesssim Be^{-2B}\int \vphi'_B \vare^2.
\end{equation}

Using the next rewriting
\begin{equation}
    (W+F)^3\,\partial_y (W+F)-Q^3 Q' = \big( (W+F)^3-Q^3 \big)\,\partial_y(W+F) + Q^3\,\partial_y(W-Q+F)
\end{equation}
and by \eqref{estimates_on_norms_L_infty_of_F}, we get
\begin{equation}
\begin{split}
    \mathcal{F}_{(1,3,3)}^{(>)} 
    & \lesssim \int_{y>-\frac B2} \Big[ |\partial_y (W+F)|\, \big(|bP_b+rR+F|^3 + Q^2|bP_b+ rR+F| \big) + Q^3 |bP_b + rR+F|\Big]\,\vare^2\,|\vphi_B-1|\\
    & \lesssim s^{-1}\int_{y>-\frac B2} |\vphi_B-1|\,\vare^2 \lesssim s^{-1}\int_{y>-\frac B2} \vphi_B \,\vare^2.
\end{split}
\end{equation}

Therefore
\begin{equation}
    \mathcal{F}_{(1,3)}^{(>)} \leq C s^{-1}\int_{y>-\frac B2} \vphi_B\,\vare^2 + \frac{\mu}{2^{10}}\int \vphi_B'\,\vare^2.
\end{equation}

Thus, we conclude the estimate 
\begin{equation}\label{mathcalF_estim_onF1_on_right}
    \mathcal{F}_{(1)}^{(>)} + \frac\mu 4 \int_{y>-\frac B2} \vphi'_B\, \big( \vare^2 + (\partial_y \vare)^2 \big) \lesssim s^{-1}\int \vphi_B\, \vare^2.
\end{equation}

\textit{Estimate on $\mathcal{F}_{(1)}^{(<)}$:}

By \eqref{estimates_between_psiB_and_phi_B}, we have
\begin{equation}
    \begin{split}
        & \mathcal{F}_{(1)}^{(<)} + 3 \int_{y<-\frac B2} \psi'_B \,(\partial^2_y \vare)^2 + \frac{1}{2}\int_{y<-\frac B2} \vphi'_B\, \big( (\partial_y \vare)^2 + \vare^2\big)  + \int_{y<-\frac B2} \psi'_B \, (\partial_y \vare)^2\\
        & \lesssim \int_{y < -\frac{B}{2}} \vphi'_B \big( |W+F|^4 \vare^2 + |\vare|^6 \big) + B \int_{y<-\frac B2} \vphi'_B \, |\partial_y (W+F)|\, \big( |W+F|^3 \vare^2 + |\vare|^5\big)\\
        & + \int_{y<-\frac B2} \psi'_B \, |\partial_y \vare|\, \Big[ |\partial_y (W+F)|\, \big( |W+F|^3 |\vare| + |\vare|^4 \big) + |\partial_y \vare|\big( |W+F|^4 + |\vare|^4 \big)\Big]\\
        & +\int_{y<-\frac B2} \psi'_B \Big( \big[2(-\partial^2_y \vare + \vare) - (W+F+\vare)^5 + (W+F)^5\big]\, \big[ (W+F)^5 - (W+F+\vare)^5 \big] \Big)\\
        & := \mathcal{F}_{(1,1),<} +\mathcal{F}_{(1,2),<} +\mathcal{F}_{(1,3),<} +\mathcal{F}_{(1,4),<}.
    \end{split}
\end{equation}

For the first term, by $\eqref{BS1}$, $\eqref{estimate_on_Q_b_and_its_derivs}$, $\eqref{L_infty_norm_of_vare_control}$, we get
\begin{equation}
\begin{split}
    \mathcal{F}_{(1,1)}^{(<)}  \lesssim \int_{y<-\frac B2} \vphi'_B \,(|W|^4 + |F|^4)\, \vare^2 + \int_{y<-\frac B2} \vphi'_B \,\vare^6 \lesssim (e^{-2B} + s^{-2\beta}+\delta(\alpha^*))\int_{y<-\frac B2} \vphi'_B\,\vare^2.
\end{split}
\end{equation}

Using $\eqref{estimate_on_Q_b_and_its_derivs}$,$\eqref{BS1}$ and choosing firstly $B$ large enough, then $s_0$ small enough, we get 
\begin{equation}
    \begin{split}
        \mathcal{F}_{(1,2)}^{(<)} 
        & \lesssim B \int_{y<-\frac B2} \vphi'_B\, |\partial_y (W+F)|\, \big( |W+F|^3\vare^2 +|\vare|^5 \big)\\
        & \lesssim B\big(e^{-\frac{B}{2}} + s^{-\frac{\beta}{2}} \big) \int_{y<-\frac B2} \vphi'_B \vare^2 \leq \frac{\mu}{2^{10}}\int_{y<-\frac B2} \vphi'_B \vare^2.
    \end{split}
\end{equation}
By the Young inequality and interpolation, combining with \eqref{estimates_between_psiB_and_phi_B}, \eqref{L_infty_norm_of_vare_control}, \eqref{estimate_on_L_infty_of_dj_F} yields
\begin{equation}
    \begin{split}
        \mathcal{F}_{(1,3)}^{(<)} 
        & \lesssim\int_{y<-\frac B2} \psi'_B\big((\partial_y \vare)^2 + \vare^2 \big) \big( |\partial_y (W+F)|^4 + |W+F|^4\big) + \int_{y<-\frac B2} \psi'_B (\partial_y \vare)^2 \vare^4 + \int_{y<-\frac B2} \psi'_B \vare^6\\
        & \lesssim \big( e^{-2B} + s^{-2\beta} + \delta(\alpha^*)\big)\int_{y<-\frac B2} \vphi'_B \big( \vare^2 + (\partial_y \vare)^2\big).
    \end{split}
\end{equation}
Applying the Young inequality, we get
\begin{equation}
    \begin{split}
        \mathcal{F}_{(1,4)}^{(<)} 
        & \lesssim  \int_{y<-\frac B2} \psi'_B \, (\partial^2_y \vare)^2\,\vare^4 + \int_{y<-\frac B2} \psi'_B \,|W+F|^4\,\big((\partial^2_y \vare )^2 + \vare^2\big)\\
        & + \int_{y<-\frac B2} \psi'_B\,\vare^6 + \int_{y<-\frac B2} \psi'_B\,|W+F|^8\,\vare^2 + \int_{y<-\frac B2} \psi'_B \vare^{10}\\
        & \leq \frac{\mu}{2^{10}}\int_{y<-\frac B2} \vphi'_B\, \vare^2 + \frac{\mu}{2^{10}}\int_{y<-\frac B2} \psi'_B \,(\partial_y^2 \vare)^2.
\end{split}
\end{equation}
Therefore
\begin{equation}\label{mathcalF_estim_onF1_on_left}
    \mathcal{F}_{(1)}^{(<)}  \leq - 2 \int_{y<-\frac B2} \psi'_B \,(\partial^2_y \vare)^2 - \frac{\mu}{4}\int_{y<-\frac B2} \vphi'_B\, \big( \vare^2+(\partial_y \vare)^2 \big) -
    \int_{y<-\frac B2} \psi'_B \, (\partial_y \vare)^2.
\end{equation}

\textit{Estimate on $\mathcal{F}_{(2)}$:} 
We write 
\begin{equation}
    \mathcal{F}_{(2)} = 2 \int \vec{m}\cdot \vec{M}Q \, G_B(\vare) - 2 \int \mathcal{R} \, G_B(\vare) = : \mathcal{F}_{(2,1)} + \mathcal{F}_{(2,2)}.
\end{equation}
For the first term in the rewriting, we have
\begin{equation}
    \mathcal{F}_{(2,1)} = 
    2 \Big( \frac{\lambda_s}{\lambda} - b \Big)\int \Lambda Q \, G_B(\vare) + 2 \Big( \frac{\sigma_s}{\lambda}-1 \Big)\int \partial_y Q \, G_B(\vare).
\end{equation}

We will use the following rewriting of the terms
\begin{equation}
    \begin{split}
        \int \Lambda Q \, G_B(\vare) 
        = & \int \psi_B \mathcal{L}(\Lambda Q) \vare  - \int \psi'_B \partial_y(\Lambda Q) \vare + \int (\vphi_B-\psi_B) \Lambda Q \vare - 5 \int \psi_B \Lambda Q \big[(W+F)^4 - Q^4 \big] \vare  \\
        & - \int \psi_B \Lambda Q \big[ (W+F+\vare)^5 - (W+F)^5 - 5(W+F)^4 \vare \big].
    \end{split}
\end{equation}
and since $\mathcal{L}(\partial_y Q) = 0$, we also get
\begin{equation}
    \begin{split}
        \int \partial_y Q\, G_B(\vare) 
        =  & -\int \psi'_B \partial^2_y Q \vare - \int (\vphi_B - \psi_B) \partial_y Q \vare - 5 \int \psi_B \partial_y Q \big[(W+F)^4 - Q^4 \big] \vare \\
        &  - \int \psi_B \partial_y Q \big[ (W+F+\vare)^5 - (W+F)^5 - 5(W+F)^4 \vare \big].
    \end{split}
\end{equation}

Then we estimate the following quantities as follows.

Using $\mathcal{L}(\Lambda Q)  = -2 Q$ and the orthogonality $\int \vare Q = 0$
\begin{equation}
    \Big\lvert \int \mathcal{L}(\Lambda Q)\, \psi_B \, \vare  \Big\rvert = 2 \Big\lvert \int (1-\psi_B)\, Q \, \vare \Big\rvert \lesssim \int_{y < -\frac{B}{2}}\, e^{- \frac{|y|}{2}}\, |\vare| \lesssim e^{-\frac{B}{8}}\|\vare\|_{L^2_{sol}}.
\end{equation}
Since $\supp \psi'_B  = \{ y < -\frac{B}{2} \} $, we get
\begin{equation}
    \Big\lvert \int \psi'_B \, \partial^2_y Q \, \vare \Big\rvert + \Big\lvert \int \psi'_B \, \partial_y \Lambda Q \, \vare \Big\rvert \lesssim \int_{y < -\frac{B}{2}} e^{-\frac{|y|}{2}} |\vare| \lesssim e^{-\frac{B}{8}} \|\vare \|_{L^2_{sol}}.
\end{equation}
From the orthogonalities $\int y \Lambda Q \, \vare  = \int y \partial_y Q \, \vare = 0$, we have
\begin{equation}
        \Big\lvert \int (\vphi_B - \psi_B)\, (\Lambda Q+ \partial_y Q)\, \vare \Big\rvert  = \Big\lvert \int (\vphi_B - \psi_B - \frac{y}{B}) (\Lambda Q + \partial_y Q)\, \vare \Big\rvert  \lesssim e^{-\frac{B}{8}}\|\vare \|_{L^2_{sol}}.
\end{equation}
By \eqref{estimates_on_norms_L_infty_of_F_with_exp}, we get
\begin{equation}
\begin{split}
    \Big\lvert \int \psi_B \big( (W+F)^4 - Q^4 \big)(\Lambda Q + \partial_y Q)\, \vare \Big\rvert \lesssim  s^{-1}\int \psi_B\,  e^{-\frac{3|y|}{4}}\,|\vare| + \int \psi_B \, e^{-\frac{3|y|}{4}} |F|\, |\vare| \lesssim s^{-1}\|\vare \|_{L^2_{sol}}
\end{split}
\end{equation}
By \eqref{estimates_on_norms_L_infty_of_F_with_exp}, \eqref{estimates_on_r}, \eqref{BS1} and \eqref{L_infty_norm_of_vare_control}, we have
\begin{equation}
    \begin{split}
        & \int \psi_B \, \big( |\Lambda Q| + |\partial_y Q| \big)\, \big| (W+F+\vare)^5 - (W+F)^5 - 5 (W+F)^4 \vare \big|\\ 
        & \lesssim \int \psi_B \, e^{-\frac{3|y|}{4}}\, \big( |W+F|^3 \vare^2 + |\vare|^5 \big)
        \lesssim B\int \vphi'_B\,\vare^2 + \delta(\alpha^*)B\int \vphi'_B\,\vare^2.
    \end{split}
\end{equation}

Thus, by \eqref{estimate_on_m_mod_estim}, the Young inequality and \eqref{relation_L2loc_with_N_B}, yields 
\begin{equation}
    \begin{split}
        \mathcal{F}_{(2,1)}
        & \lesssim \big( e^{-\frac{B}{8}} + s^{-1}\big)\|\vare\|^2_{L^2_{sol}} + s^{-4} + s^{-3}\|\vare\|_{L^2_{sol}} + s^{-\frac{5}{4}}B\int \vphi'_B \vare^2\\
        & \lesssim s^{-4} + B \big( e^{-\frac{B}{8}} + s^{-1}\big)\int \vphi'_B\,\vare^2 \leq Cs^{-4} + \frac{\mu}{2^{10}}\int \vphi'_B\,\vare^2.
    \end{split}
\end{equation}

For the second term in the rewriting, we have
\begin{equation}
    \mathcal{F}_{(2,2)}= -2 \int \psi_B\, \partial_y \mathcal{R}\, \partial_y \vare - 2 \int \vphi_B\, \mathcal{R}\, \vare + 2 \int \psi_B \mathcal{R} \big( (W+F+\vare)^5 - (W+F)^5 \big).
\end{equation}
And we estimate these quantities as follows, using the relation \eqref{estimates_between_psiB_and_phi_B}
\begin{equation}
\begin{split}
    \int \psi_B\, \partial_y \mathcal{R}\, \partial_y \vare 
    & \lesssim \Big( \int (\partial_y \mathcal{R})^2\, e^{\frac{y}{B}} \Big)^{\frac{1}{2}}\Big( \int \psi^2_B \,(\partial_y \vare)^2\, e^{-\frac{y}{B}} \Big)^{\frac{1}{2}} \lesssim B^{\frac{1}{2}} \|\partial_y \mathcal{R}\|_{L^2_B} \Big( \int \vphi'_B (\partial_y \vare)^2 \Big)^{\frac{1}{2}}\\
    & \lesssim \frac{1}{2B}\int \vphi'_B (\partial_y \vare)^2 + B^2\| \partial_y \mathcal{R}\|^2_{L^2_B}.
\end{split}
\end{equation}
Similarly, we get
\begin{equation}
    \int \vphi_B \, \mathcal{R}\, \vare \lesssim \frac{1}{2B}\int \vphi'_B \vare^2 + B^2\|\mathcal{R}\|^2_{L^2_B}.
\end{equation}
For the last term in the rewriting of $\mathcal{F}_{(2,2)}$, we have
\begin{equation}
    \int \psi_B \,\mathcal{R}\, \big( (W+F+\vare)^5 - (W+F)^5\big) \lesssim \int \psi_B \,|\mathcal{R}|\big( |W+F|^4 |\vare| + |\vare|^5 \big).
\end{equation}
Since $\|W\|_{L^{\infty}}, \|F\|_{L^{\infty}}$ are finite quantities, it follows
\begin{equation}
     \int \psi_B\, |\mathcal{R}|\,\big( |W|^4 + |F|^4 \big) |\vare| \lesssim \frac{1}{2B}\int \vphi'_B \vare^2 +  B^2 \|\mathcal{R}\|^2_{L^2_B}.
\end{equation}
By \eqref{estimates_between_psiB_and_phi_B} and the Young inequality, we get
\begin{equation}
\begin{split}
    \int \psi_B\, |\mathcal{R}|\, |\vare|^5 
    & \lesssim \delta(\alpha^*)\Big(\int \psi^2_B \,\vare^2 \,e^{-\frac{y}{B}} \Big)^{\frac{1}{2}} \Big( \int R^2 \, e^{\frac{y}{B}}\Big)^{\frac{1}
    {2}}
    \lesssim \delta(\alpha^*)B^{\frac{1}{2}}\Big(\int \vphi'_B \,\vare^2 \Big)^{\frac{1}{2}} \|\mathcal{R}\|_{L^2_B}\\
    & \lesssim \delta(\alpha^*)\Big[\frac{1}{B}\int \vphi'_B\, \vare^2 + B^2 \|\mathcal{R}\|^2_{L^2_B} \Big].
\end{split}
\end{equation}
Therefore, by \eqref{estimate_on_norms_L_2_B_of_R}, we get
\begin{equation}
\begin{split}
    \mathcal{F}_{(2,2)} 
    & \lesssim \frac{1}{B}\int \vphi'_B \,\big( \vare^2 + (\partial_y\vare)^2 \big) + B^2 \big( |b_s|+s^{-1}|\vec{m}| + s^{-2} \big)^2\\
    & \lesssim \frac{1}{B}\int \vphi'_B \,\big( \vare^2+(\partial_y \vare)^2 \big) + B^2 s^{-4}.
\end{split}
\end{equation}

Combining those estimates and using \eqref{conseq_of_BS_on_m_on_bs}, \eqref{relation_L2loc_with_N_B}, \eqref{BS2}, we have
\begin{equation}\label{mathcalF_estim_onF2}
     \mathcal{F}_{(2)}  \leq C B^2 \, s^{-4} + \frac{\mu}{2^{5}}\int \vphi'_B \big( \vare^2 + (\partial_y \vare)^2 \big).
\end{equation}

\textit{Estimate on $\mathcal{F}_{(3)}$:}

In order to estimate $\mathcal{F}_{(3)}$, we write
\begin{equation}
    -\int \partial_y \vare \, \partial_y (\psi_B\, \partial_y \vare) = -\frac{1}{2}\int \psi'_B (\partial_y \vare)^2 \lesssim \int \vphi'_B\, (\partial_y \vare)^2  \qquad \qquad \int \vphi_B\, \partial_y \vare \, \vare = -\frac{1}{2}\int \vphi'_B \, \vare^2.
\end{equation}

Combining the rewriting \eqref{relation_on_partial_eps_and_order_5} with the interpolation, using that $f \in L^{\infty}_t H^2_x$ and \eqref{estimates_between_psiB_and_phi_B}, \eqref{L_infty_norm_of_vare_control}, we get
\begin{equation}
\begin{split}
    - \int \partial_y \vare \, \Big( \psi_B \big[ (W+F+\vare)^5 - (W+F)^5 \big] \Big) 
    &\lesssim  B \int \vphi'_B \Big[ \vare^2 \big( |W+F|^4 + |\partial_y (W+F)|^4 \big) + \vare^6 \Big] \\
    &\lesssim B \, \delta(\alpha^*) \int \vphi'_B \, \vare^2.
\end{split}
\end{equation}

Thus, we conclude that 
\begin{equation}\label{mathcalF_estim_onF3}
     \mathcal{F}_{(3)} \lesssim s^{-\frac{5}{4}} \, B\, \int \vphi'_B \big( \vare^2 + (\partial_y\vare)^2 \big) \leq \frac{\mu}{2^{5}}\int \vphi'_B \big( \vare^2 + (\partial_y\vare)^2 \big).
\end{equation}

\textit{Estimate on $\mathcal{F}_{(4)}$:}

By integration by parts, we get
\begin{equation}
    \int \Lambda \vare\; \partial_y(\psi_B\, \partial_y \vare) = -\int \psi_B\, (\partial_y \vare)^2 + \frac{1}{2}\int y\psi'_B\, (\partial_y \vare)^2,
\end{equation}
\begin{equation}
    \int \Lambda \vare\; \vare \;\vphi_B = -\frac{1}{2}\int y\vphi'_B \,\vare^2.
\end{equation}
Using \eqref{relation_on_partial_eps_and_order_5} yields
\begin{equation}
\begin{split}
\int \Lambda \vare\, \big( (W+F+\vare)^5 - (W+F)^5 \big) 
& = \frac{1}{6}\int (2\psi_B-y\psi'_B)\big[ (W+F+\vare)^6 - (W+F)^6 - 6 (W+F)^5\vare\big]\\
& -\int \psi_B \Lambda(W+F)\big[ (W+F+\vare)^5 - (W+F)^5 -5 (W+F)^4 \vare\big].
\end{split}
\end{equation}

Thus, it yields the following rewriting
\begin{equation}
    \begin{split}
        & \mathcal{F}_{(4)} + \frac{\lambda_s}{\lambda}\int y \vphi'_B\vare^2 - \frac{j}{s}\int \vphi_B \vare^2\\
        & = \Big( 2 \frac{\lambda_s}{\lambda}+\frac{j}{s} \Big)\int \psi_B\, \Big[ (\partial_y\vare)^2
        - \frac{1}{3}\big( (W+F+\vare)^6 - (W+F)^6 - 6 (W+F)^5\vare\big) \Big]\\
        & -\frac{\lambda_s}{\lambda}\int y \psi'_B \Big[ (\partial_y\vare)^2
        - \frac{1}{3}\big( (W+F+\vare)^6 - (W+F)^6 - 6 (W+F)^5\vare\big) \Big]\\
        & +2\frac{\lambda_s}{\lambda}\int \psi_B\,\Lambda(W+F)\, \big( (W+F+\vare)^5 - (W+F)^5 -5 (W+F)^4 \vare\big) \\
        & =: \mathcal{F}_{(4,1)}+\mathcal{F}_{(4,2)}+\mathcal{F}_{(4,3)}.
    \end{split}
\end{equation}
By \eqref{estimates_between_psiB_and_phi_B}, \eqref{estimate_on_m_mod_estim} and \eqref{L_infty_norm_of_vare_control}, we get
\begin{equation}
\begin{split}
    \mathcal{F}_{(4,1)}+\mathcal{F}_{(4,2)}& \lesssim s^{-1}B\int \vphi'_B \big( (\partial_y \vare)^2 + |W+F|^4\vare^2 + \vare^6 \big)\\
    & \lesssim s^{-1}B\int \vphi'_B\big( \vare
    ^2 + (\partial_y \vare)^2\big) .
\end{split}
\end{equation}
For the third term we write
\begin{equation}
    \begin{split}
        \mathcal{F}_{(4,3)} & \lesssim s^{-1}\int \psi_B|W+F|\big( |W+F|^3\vare^2 + |\vare|^5 \big)+s^{-1}\int  y\, \psi_B\, |\partial_y (W+F)|\big( |W+F|^3\vare^2 + |\vare|^5 \big)\\
        & \lesssim s^{-1}B\int \vphi'_B \vare^2 + s^{-1}B\int y 
        \,\vphi'_B \vare^2.
    \end{split}
\end{equation}
Therefore, 
\begin{equation}\label{mathcalF_estim_onF4}
    \mathcal{F}_{(4)} \leq  Cs^{-1}B  \int y\, \vphi'_B\vare^2 + C\frac{j}{s}\int \vphi_B \vare^2 + \frac{\mu}{2^{5}}\int \vphi'_B\big( \vare
    ^2 + (\partial_y \vare)^2\big) .
\end{equation}

\textit{Estimate on $\mathcal{F}_{(5)}$:}
From the definition, we get the following decomposition
\begin{equation}
    \begin{split}
         \mathcal{F}_{(5)} = 
        &-2\,  \int \psi_B\, \big(b_s\,\partial_b(Q_b) + r_s\, R \big)\,\big[ (W+F+\vare)^5-(W+F)^5 - 5 (W+F)^4\vare \big] \\
        & -2\, \int \psi_B\, \partial_s F\, \big[ (W+F+\vare)^5 - (W+F)^5 - 5 (W+F)^4 \vare \big]  =: \mathcal{F}_{(5,1)} + \mathcal{F}_{(5,2)}.
    \end{split}
\end{equation}
In order to estimate the first term, we write
\begin{equation}
    |\partial_b\, Q_b|  = |\chi_b P+\gamma|b|^{\gamma}y\,\chi'(|b|^{\gamma}y)| \lesssim \gamma.
\end{equation}
Using \eqref{estimation_on_bs_mod_estim}, \eqref{estimates_between_psiB_and_phi_B}, \eqref{L_infty_norm_of_vare_control},\eqref{estimate_on_r_s} and \eqref{estimates_on_norms_L_infty_of_F}, we get
\begin{equation}
    \mathcal{F}_{(5,1)} \lesssim s^{-2}\,(B+\delta(\alpha^*))\, \int \vphi'_B \vare^2 .
\end{equation}

For the second term, we write
\begin{equation}
\begin{split}
    & \partial_s\big(F(s,y)\big) = \partial_s \big[ \lambda^{\frac{1}{2}}(s)\,f\big(\tau(s),\lambda(s)y +\sigma(s)\big) \big]\\
    & =\frac{1}{2}\frac{\lambda_s}{\lambda} F(s,y) + \lambda^{\frac{7}{2}}(s)(\partial_1 f)\big(\tau(s), \lambda(s)y+\sigma(s)\big)+ \Big( \frac{\lambda_s}{\lambda}y +\frac{\sigma_s}{\lambda}\Big) \partial_y \big(F(s,y) \big).
\end{split}
\end{equation}

Since $\|f\|_{L^{\infty}_t H^s} \lesssim \delta(x_0^{-1})$, the quantity $\|\partial_t f(t)\|_{L^{\infty}}$ is bounded uniformly on $[0,T)$ 
\begin{equation}
    \|\partial_t f(t)\|_{L^{\infty}} \lesssim \|\partial^3_x f(t) \|_{L^{\infty}} + \|f(t)\|^4_{L^{\infty}} \|\partial_xf(t)\|_{L^{\infty}}.
\end{equation}

By \eqref{estimates_between_psiB_and_phi_B}, \eqref{conseq_of_BS_on_m_on_bs}, \eqref{estimate_on_L_infty_of_dj_F}, we get
\begin{equation}
\begin{split}
    \mathcal{F}_{(5,2)} & \lesssim \int \psi_B  \Big| \frac{\lambda_s}{\lambda}F + \lambda^{\frac{7}{2}} \big(\partial_1 f \big)(\tau(s), \lambda y +\sigma)+ \frac{\sigma_s}{\lambda} \partial_y F \Big|\, \big(|W+F|^3\vare^2 + |\vare|^5 \big)\\
    & + \frac{\lambda_s}{\lambda}\int \psi_B |y|\, |\partial_y F|\, \big(|W+F|^3\vare^2 + |\vare|^5 \big)\\
    & \lesssim
    B s^{-\frac{3\beta}{2}}\int \vphi'_B \vare^2 + Bs^{-\frac{3\beta}{2}} \,s^{-1}\int \vphi_B \vare^2.
\end{split}
\end{equation}
Therefore, we conclude
\begin{equation}\label{mathcalF_estim_onF5}
    \begin{split}
        \mathcal{F}_{(5)} \leq Cs^{-1}\int \vphi_B\,\vare^2 + \frac{\mu}{2^{5}}\int \vphi'_B\,\vare^2.
    \end{split}
\end{equation}

\vspace{0.4cm}
Finally, combining the estimates \eqref{mathcalF_estim_onF1_on_right}, \eqref{mathcalF_estim_onF1_on_left}, \eqref{mathcalF_estim_onF2}, \eqref{mathcalF_estim_onF3}, \eqref{mathcalF_estim_onF4}, \eqref{mathcalF_estim_onF5}, yields the announced estimate \eqref{mathcalF_estim_on_all_dds}.

The presence of the positive term $s^{-1}\int \varphi_B \vare^2$ in the right-hand side
of \eqref{mathcalF_estim_on_all_dds} justifies the introduction of the second
term in the definition of $\mathcal{H}$.

\vspace{0.4cm}
An estimate of the derivative of 
the second term in $\mathcal{H}$ is given by the following lemma
\begin{lemma}(Dynamical control of the tail on the right)\label{lemma_control_of_scaling_term}\\
    For all $s\in[s_0,s^*]$ holds
\begin{equation}\label{scaling_term_estimate}
    \lambda^{-k}\frac{d}{ds}\Big[ \lambda^k \int \phi_k \vare^2 \Big] + \frac{1}{8}\int (\vare^2_y+\vare^2)\phi_k \lesssim s^{-\frac{5}{2}}.
\end{equation}
\end{lemma}
\begin{proof}[Proof of the Lemma \ref{lemma_control_of_scaling_term}]
Using the equation of $\vare_s$, we have
\begin{equation}
        \frac{1}{2}\frac{d}{ds}\int \phi_k\,\vare^2 = \int \phi_k \,\vare\,\vare_s = \mathcal{D}^{(1)} + \mathcal{D}^{(2)} + \mathcal{D}^{(3)} + \mathcal{D}^{(4)}
\end{equation}
where
\begin{equation}
    \begin{split}
        \mathcal{D}^{(1)} = \int \phi_k \vare \,\partial_y\big[ -\partial_y^2 \vare + \vare
        + \frac{\lambda_s}{\lambda}\Lambda \vare \big] ,
        \qquad \qquad
        \mathcal{D}^{(2)} = \Big( \frac{\sigma_s}{\lambda}-1 \Big)\int \phi_k\, \partial_y \vare\, \vare ,\\
        \mathcal{D}^{(3)} = - \int \phi_k\,\vare \, \mathcal{E}(W) , \qquad \qquad
        \mathcal{D}^{(4)} = -\int \phi_k \,\vare\, \partial_y\big[ (W+F+\vare)^5 - (W+F)^5 \big].
    \end{split}
\end{equation}

In order to treat the first term, we observe that, $\phi_k \leq \frac{1}{2}\phi_k$ for $y$ large  enough and $\phi_k$ is defined such that the $y\phi_k = k \phi_k$ on $\RR$. By interpolation, we get
\begin{equation}
    \begin{split}
        \mathcal{D}^{(1)}  & = -\frac{k}{2}\frac{\lambda_s}{\lambda}\int \phi_k\,\vare^2 - \frac{1}{2}\int \phi_k\,\vare^2 - \frac{3}{2}\int \phi_k\,(\partial_y \vare)^2 + \frac{1}{2}\int \phi_k \,\vare^2\\
        & \leq
        -\frac{k}{2} \frac{\lambda_s}{\lambda}\int \phi_k \,\vare^2 - \frac{1}{4}\int \phi_k \, \big( \vare^2 + (\partial_y \vare)^2 \big) + C \int \vphi'_B\,\vare^2.
    \end{split}
\end{equation}

By integration by parts, the second term can be rewritten as follows
\begin{equation}
    \mathcal{D}^{(2)} = -\frac{1}{2}\Big( \frac{\sigma_s}{\lambda}-1 \Big)\int \phi_k \,\vare^2.
\end{equation}
The third term is given by
\begin{equation}
    \mathcal{D}^{(3)} = \int \phi_k\, \vec{m} \cdot \vec{M}Q\,\vare - \int \phi_k \,\mathcal{R}\, \vare.
\end{equation}
By Cauchy-Schwarz inequality and by \eqref{conseq_of_BS_on_m_on_bs} with \eqref{BS2}, we have
\begin{equation}
    \int \phi_k\, \vec{m} \cdot \vec{M}Q\,\vare \lesssim |\vec{m}| \Big( \int e^{-\frac{|y|}{2}}\, \phi^2_k \Big)^{\frac{1}{2}} \|\vare\|_{L^2_{sol}} \leq C(k)\,s^{-\frac{5}{2}}.
\end{equation}
Using \eqref{estimates_between_psiB_and_phi_B}, the Young inequality and \eqref{estimate_on_norms_L_2_B_of_R}, we get
\begin{equation}
    \begin{split}
        \int \phi_k\,\mathcal{R}\,\vare &\lesssim \|\mathcal{R}\|_{L^2_B} \Big( \int \phi^2_k \, e^{-\frac{y}{B}} \vare^2\Big)^{\frac{1}{2}}\lesssim B^{2k+2}\|\mathcal{R}\|^2_{L^2_B} + \frac{1}{B^2}e^{-k}k^{4k}\,\int \phi_k\,\vare^2\\ 
        & \leq CB^{3k}s^{-4} + \frac{1}{2^{10}}\int \phi_k\, \vare^2
        \leq \frac{1}{2^{10}} \Big(s^{-\frac{5}{2}} + \int \phi_k\, \vare^2 \Big).
    \end{split}
\end{equation}
We emphasize that this is the last time that we need to adjust the constant $B$.
Thus $B$ is now fixed.

We rewrite the nonlinear term as follows
\begin{equation}
\begin{split}
    \mathcal{D}^{(4)}  & = \int \phi_k\, \vare\, \big( (W+F+\vare)^5 - (W+F)^5 \big) + \int \phi_k \,\partial_y \vare\, \big( (W+F+\vare)^5 - (W+F)^5 \big)\\
    & \lesssim \int \phi_k'\, |W+F|^4 \, \vare^2 + \int \phi_k \,\vare^6 + \int \phi_k\,|\partial_y(W+F)|\,\big( |W+F|^3\vare^2 + |\vare|^5 \big).
\end{split}
\end{equation}

By \eqref{estimates_on_norms_L_infty_of_F}, \eqref{estimates_on_r} and \eqref{estimate_on_Q_b_and_its_derivs}, we have
\begin{equation}
    \int \phi_k'\, |W+F|^4 \, \vare^2 + \int \phi_k \,\vare^6 \lesssim \int \vphi'_B \, \vare^2 + \big( s^{-2\beta} + \delta(\alpha^*) \big)\int \phi_k \,\vare^2.
\end{equation}

By exponential decay of $Q_b$ and $R$ on the right and by \eqref{estimates_between_psiB_and_phi_B}, we get
\begin{equation}
    \begin{split}
        & \int \phi_k\,|\partial_y(W)|\,\big( |W+F|^3\vare^2 + |\vare|^5 \big) \lesssim \int (\phi_k + r \phi_k)\big( |W+F|^3 \vare^2 + |\vare|^5 \big)\\
        & \lesssim \int \vphi'_B\, \vare^2 + \big( s^{-1} + s^{-\frac{3\beta}{2}} +\delta(\alpha^*)\big)\int\phi_k \,\vare^2.
    \end{split}
\end{equation}

From \eqref{lemma_on_the_decaying_tail} yields 
\begin{equation}
    |\partial_y F| = |\lambda^{\frac{3}{2}} \partial_2 f(s, \lambda y +\sigma)|  \lesssim \lambda^{\frac{3}{2}}|\lambda y + \sigma|^{-\theta - 1} \lesssim \lambda^{\frac{1}{2}}\sigma^{-\theta}y^{-1}.
\end{equation}
Thus, using the above estimate, \eqref{relation_on_partial_eps_and_order_5} and \eqref{BS2}, we get
\begin{equation}
\begin{split}
    \int\phi_k\, |\partial_y(F)|\, \big( |W+F|^3 |\vare|^3 + |\vare|^5\big) \lesssim \lambda^{\frac{1}{2}} \sigma^{-\theta} \int \phi_k\, y^{-1}\, \big( |W+F|^3 \vare^2 + |\vare|^5\big) \lesssim s^{-1}\int \phi_k \, \vare^2.
\end{split}
\end{equation}

Combining the estimates on $\mathcal{D}^{(1)},\mathcal{D}^{(2)},\mathcal{D}^{(3)},\mathcal{D}^{(4)}$ with \eqref{relation_L2loc_with_N_B} and \eqref{BS2}, we conclude the proof of the estimate \eqref{scaling_term_estimate}.
\end{proof}

\textit{End of the estimate of the functional $\mathcal{H}$.}\\
Combining the estimate in Proposition \ref{Proposition_Monot_formula} on the mixed energy-virial functional and in Lemma \ref{lemma_control_of_scaling_term} on the scaling term yields
    \begin{equation}\label{control_of_H_in_proof_1}
        \begin{split}
            & \frac{d}{ds}\big[\mathcal{H}\big]+s^j \Big[ 2\int_{y<-\frac B2 } \psi'_B (\partial^2_y \vare)^2 + \int_{y<-\frac B2}\psi'_B (\partial_y \vare)^2 + \frac{\mu}{4} \int \vphi'_B \big( \vare^2 + (\partial_y \vare)^2\big) \Big] + \frac{\lambda^k}{8}\int \phi_k \big( \vare^2 + (\partial_y \vare)^2\big)\\  
            & \lesssim s^{j-4} + s^{j-1}\int \vphi_B\,\vare^2+ \lambda^ks^{-\frac 52} .
        \end{split}
    \end{equation}

Now, to finish the proof of \eqref{relation_on_partial_eps_and_order_5}, 
we show that the term with a positive sign
$$ s^{j-1} \int \vphi_B\,\vare^2$$
on the right-hand side of \eqref{control_of_H_in_proof_1},
can be controlled, for $s_0$ large enough, using the following two terms
$$   s^j\int\vphi'_B\,\vare^2 \qquad \text{and} \qquad \lambda^k\int \phi_k \,\vare^2 $$
which appear on the left-hand side of \eqref{control_of_H_in_proof_1}.

Using $\varphi_B(y)=y/B$ for $y>2B$ and taking $\vartheta>2B$, we decompose
\begin{equation}\label{control_of_delicate_term_1}
        s^{j-1} \int \vphi_B\,\vare^2 
        =s^{j-1}\frac 1B \int_{y>\vartheta} y \vare^2  + s^{j-1}\int_{y<\vartheta} \vphi_B \vare^2
        \lesssim s^{j-1} \vartheta^{-(k-2)}\int_{y>\vartheta} \vare^2 y^{k-1} + s^{j-1}\int_{y<\vartheta} \vphi_B  \vare^2 .
\end{equation}
We estimate the first term on the right-hand side of \eqref{control_of_delicate_term_1}.
Choosing $\vartheta = s^{\frac{1+\beta}2}$ ($\vartheta>2B$ for $s_0$ large), we have
\begin{equation}
    s^{j-1} \vartheta^{-(k-2)}\int_{y>\vartheta} \vare^2 y^{k-1} 
    = s^{-(1-j+(k-2)\frac{\beta+1}{2})}\int_{y>\vartheta} \vare^2 y^{k-1} .
\end{equation}
Note that
\begin{equation}
     1-j+(k-2)\frac{\beta+1}{2}  
     = \beta k + (1-\beta)
     +\left[ k \frac{1-\beta}{2}- (1+j)\right].
\end{equation}
Thus, the condition on $k$ in \eqref{choice-of-k} implies that
\begin{equation}
     1-j+(k-2)\frac{\beta+1}{2}  
     > \beta k + (1-\beta ) .
\end{equation}
Therefore, using also the definition of the function $\phi_k$
and \eqref{BS1}, choosing $s_0$ sufficiently large, we obtain
\begin{equation}
    s^{j-1} \vartheta^{-(k-2)}\int_{y>\vartheta} \vare^2 y^{k-1} 
    \leq   s_0^{- (1-\beta)} s^{-\beta k} \frac 1k \int_{y>\vartheta} \phi_k' \vare^2 
    \leq \frac{\lambda^k}{2^{10}}\int \phi_k \vare^2  \;.
\end{equation}
We remark that the choice of $s_0$ depends on $k$ and goes to infinity with $k$.

\vspace{0.15cm}
Now, we estimate the second term on the right-hand side of \eqref{control_of_delicate_term_1}.
Note that for any $2<a<+\infty$, it holds $\vphi(y) \lesssim a \vphi'(y)$ on $(-\infty,a]$. Taking $a = \vartheta=s^{\frac{1+\beta}2}>2B$ (for sufficiently large $s_0$), we find
\begin{equation}
    \begin{split}
        s^{j-1}\int_{y<\vartheta}\vare^2\vphi_B \lesssim B\, \vartheta\, s^{j-1} \int_{y<\vartheta}\vare^2\vphi'_B \lesssim B\, s_0^{-\frac{1-\beta}{2}}s^j \int \vare^2 \vphi'_B \leq \frac{\mu}{2^{10}}\, s^j \int\vare^2 \vphi'_B \;.
    \end{split}
\end{equation}
Thus, we have proved the estimate \eqref{Lyapounov_control_of_H}.

\vspace{0.4cm}
\textit{Proof of the coercivity of $\mathcal{F}$:}\\
We decompose the functional as follows
\begin{equation}
    \begin{split}
        \mathcal{F} = 
        & \int \Big[ (\partial_y \vare)^2\psi_B +\vare^2 \vphi_B - 5Q^4 \vare^2\psi_B \Big]\\
        &-\frac{1}{3}\int \psi_B \Big[\big( W+F+\vare \big)^6 - (W+F)^6 - 6(W+F)^5\vare - 15Q^4\vare^2 \Big]
    =: \mathcal{F}_1 + \mathcal{F}_2.
    \end{split}
\end{equation}
The second term is treated as in the proof of the Lemma \ref{lemma-localized-virial-estimate} below. Other possible references are \cite[Appendix A]{Martel-Merle-02} and \cite[Lemma 3.5 ]{Combet-Martel-17}. We deduce that there exists $\nu>0$ such that for $B$ large enough, it holds
\begin{equation}
     \mathcal{F}_1 \geq \nu\, \mathcal{N}^2_B.
\end{equation}

We rewrite the term $\mathcal{F}_2$ as follows
\begin{equation}
    \begin{split}
        \mathcal{F}_2 = 
        & -\frac{1}{3}\int \psi_B \Big[ \big( W+F+\vare\big)^6 - \big( W+F\big)^6 - 6 \big( W+F\big)^5\vare - 15\big( W+F\big)^4\vare^2 \Big]\\
        & +5\int \psi_B\Big[\big( W+F \big)^4 - Q^4 \Big]\vare^2.
     \end{split}
\end{equation}
The collection of estimates \eqref{estimates_on_norms_L_infty_of_F}, \eqref{L_infty_norm_of_vare_control}, \eqref{estimate_on_Q_b_and_its_derivs} \eqref{estimates_on_r} and \eqref{BS1} yields
\begin{equation}
    \big| \mathcal{F}_2 \big| \leq \frac{\nu}{2^{10}}\,\mathcal{N}_B^2.
\end{equation}
This finishes the proof of the coercivity of $\mathcal{F}$.
\end{proof}
\section{Construction of blowing up solutions}\label{section_construction}
As hinted in the strategy of the proof in Section \eqref{S:strategy},
the quantity $h(s)$ defined in \eqref{definition-g-h} is related to one
direction of instability of the exotic blow-up behavior that we want to obtain.
Thus, we need to control $h$ by the choice of the initial data
and a specific continuity argument.
To this aim, we consider the bootstrap assumption
\begin{equation}\tag{BS3}\label{BS3}
    \lvert h(s) \rvert \leq s^{-\frac{\beta}{2} - 2\rho}.
\end{equation}
We follow the theory developed in Sections \ref{section_decomposition} and \ref{section_energy_estimates}. In particular, we prove that for some well-adjusted initial data, the decomposition \eqref{decomposition_of_v_in_lemma} and bootstrap estimates \eqref{BS1}, \eqref{BS2} and \eqref{BS3} hold for the entire time interval $[s_0,+\infty)$, for $s_0$ sufficiently large.

\begin{proposition}\label{Prop_on_blow_up}
Let $x_0$ and $s_0>0$ be large enough and fix $\sigma_0$ such that
\begin{equation}\label{init_cond_sigma_b}
    \Big\lvert \sigma_0 - \frac{1}{1-\beta}s^{1-\beta}_0 \Big\rvert \leq s_0^{1-\beta-2\rho},
\end{equation}
where $\rho $ satisfies \eqref{condition_on_rho}.
Let $\vare_0 \in H^1(\RR)$ be such that 
\begin{equation}\label{conditons_decroissance_on_vare_0}
s_0^{j} \Big [\, \|\varepsilon_0\|^2_{H^1} +  \int_{y>1} \, y \, \varepsilon^2_0(y) dy \Big] < s_0^{-\frac 12}\, , \quad \int_{y>1} y^k\vare^2_0(y)dy <1\,
\end{equation}  
and 
\begin{equation}\label{conditons_ortho_on_vare_0}
    (\varepsilon_0,y\Lambda Q) = (\varepsilon_0, \Lambda Q) = (\varepsilon_0, Q) = 0\,.
\end{equation}
    Then, there exists 
    \begin{equation}\label{init_cond_in_BS}
        \lambda_0 \in [\lambda^-_0,\lambda^+_0], \quad \text{ where } \quad \lambda^{\pm}_0 := \Big( \frac{1}{\int Q}\frac{2 c_0}{1-\theta} \sigma^{1-\theta}_0
        \pm s^{-\frac{\beta}{2}-2\rho}_0 \Big)^2
    \end{equation}
such that the solution $U$ of \eqref{gKdV_principal_eq} evolving from the initial data
\begin{equation}\label{init_U_0_in_the_technical_lemma}
    U_0(x):= \lambda_0^{-\frac{1}{2}}\Big( Q_{b_0}+\lambda_0^{-\frac{1}{2}}f(\tau(s_0),\sigma_0)R + \vare_0 \Big)\Big( \frac{x-\sigma_0}{\lambda_0}\Big) + f(\tau(s_0),x)
\end{equation}
with 
\begin{equation}
    b_0 = -\frac{4c_0}{\int Q}\lambda^{\frac 12}_0 \sigma^{-\theta}_0,
\end{equation}
has a decomposition as in Lemma \ref{lemma_on_decomposition_around_Q} and satisfies \eqref{BS1}, \eqref{BS2} and \eqref{BS3} on the interval $[s_0,+\infty)$.
\end{proposition}

\begin{remark}
As one can see in the statement of Proposition \ref{Prop_on_blow_up}, 
we choose the initial scaling parameter $\lambda_0$ (in a certain small interval) to control the instability related to $h$.
Note also that the initial data $U_0$ chosen in \eqref{init_U_0_in_the_technical_lemma} is already modulated,
in the sense that $(\lambda_0,\sigma_0,b_0)$ are the parameters of the modulation of the solution $U(t)$ at the initial time $t=t_0$.
\end{remark}
\begin{remark}
Observe that the decay condition on the right-hand side  on
$\vare_0$ imposed by \eqref{conditons_decroissance_on_vare_0}
implies that $\vare_0$ is negligible compared to $f_0$ for large $x$. 
To justify this formally, we compute, using the value of $k$ in \eqref{choice-of-k}
and $2\theta = \frac{2-\beta}{1-\beta}$,
\[
\int_{x\gg1} x^k f_0^2 dx
\geq \int_{x\gg1} x^{\frac 2{1-\beta}} x^{-2\theta} dx
= \int_{x\gg1} x^{\frac\beta{1-\beta}} dx = + \infty.
\]
This ensures that the perturbation $\vare$ will not perturb the decaying tail of the function
$f$, even if that one is rapidly decaying for $\nu$ close to $\frac 12$ (case where $\theta$ is large).
\end{remark}

\vspace{0.4cm}
\begin{proof}[Proof of the Proposition \ref{Prop_on_blow_up}] 
By the continuity argument, $U$ has a decomposition as in Lemma \ref{lemma_on_decomposition_around_Q} on some small time interval $[s_0,s_1]$ with $s_1>s_0$.

Recall that 
\begin{equation}\tag{BS1}
        \big\lvert \sigma(s) - \frac{1}{1-\beta} s^{1-\beta}\big\rvert \;\leq\; s^{1-\beta-\rho} \,, \quad
        \lvert \lambda(s) - s^{-\beta} \rvert \;\leq\; s^{-\beta-\rho} \,,\quad
        \lvert b(s) - \beta s^{-1} \rvert \;\leq\; s^{-1-\rho} \,;
\end{equation}
\begin{equation}\tag{BS2}
        \mathcal{N}_B(s) \;\leq\; s^{-\frac{5}{4}}\,, \qquad \|\vare(s)\|_{H^1} \;\leq\; \alpha^*\,;
\end{equation}
\begin{equation}\tag{BS3}
    \lvert h(s) \rvert \leq s^{-\frac{\beta}{2} - 2\rho}\,.
\end{equation}

For $\lambda_0 \in \mathcal{D}_0=[\lambda_0^-,\lambda_0^+]$, we define
\begin{equation}
    s^*(\lambda_0) :=\sup\{s\geq s_0, \text{ Lemma \ref{lemma_on_decomposition_around_Q}  applies and \eqref{BS1}-\eqref{BS2}-\eqref{BS3} hold on $[s_0,s]$} \}.
\end{equation}
For the sake of contradiction, we suppose for all $\lambda_0 \in \mathcal{D}_0$,
it holds $s^*(\lambda_0)< +\infty$.
The supremum $s^*$ is given only by the saturation of the bootstrap assumptions.


\vspace{0.4cm}
First, we strictly improve the bootstrap estimates \eqref{BS1} and \eqref{BS2} on $[s_0,s^*]$. Then, we find a contradiction using a continuity argument on the bootstrap assumption \eqref{BS3}.

\vspace{0.2cm}
\textit{Closing \eqref{BS2}.}\\
We apply \eqref{Lyapounov_control_of_H} with $j= \frac{5}{2}$ 
\begin{equation}
    \frac{d}{ds}\big[ \mathcal{H} \big] \;\lesssim\;  \lambda^k s^{-\frac 52} + s^{-\frac{3}{2}}.
\end{equation}
This choice $j=\frac 52$ is justified on the one hand by the need of having an error term on the 
right-hand side which is integrable in time and on the other hand by the fact that
a larger $j$ gives a better final estimate on $\vare$.
Integrating on $[s_0,s]$ for some $s\leq s^*$ and using \eqref{BS1} on $\lambda$, we get
\begin{equation}
    s^{\frac{5}{2}}\mathcal{F}(s) \;\lesssim\; \mathcal{H}(s)\;\lesssim\; \mathcal{H}(s_0) + s_0^{-\frac 12}.
\end{equation}
By our choice of initial conditions in \eqref{conditons_decroissance_on_vare_0}, the smallness of $\lambda_0$ in \eqref{init_cond_in_BS} and the lower bound on $k$ in \eqref{choice-of-k}, we have 
\begin{equation}
    \mathcal{H}(s_0)  \lesssim s_0^\frac52\, \Big[ \|\vare_0\|^2_{H^1}+ \int_{y>1}y\, \vare^2_0 \,dy \Big]  + \lambda^k_0 \Big[ \int \phi_k\,\vare^2_0 \Big] \lesssim s_0^{-\frac 12}.
\end{equation}
The coercivity property \eqref{coercivity_of_mathcal_F} and the bootstrap condition \eqref{BS2}, yield for $s_0$ large enough
\begin{equation}
    \mathcal{N}^2_B(s) \;\lesssim\; \mathcal{F}(s)  \;\lesssim\; s^{-\frac 52}s_0^{-\frac12}
    \leq \frac12 s^{-\frac 52}.
\end{equation}
Thus, the estimate of $\mathcal{N}_B$ in \eqref{BS2} is strictly improved.
The control of the $H^1$ norm provided by \eqref{estimate_on_L2_norm_of_vare} and \eqref{estimate_on_L2_norm_of_grad_vare} completes the improvement of the
bootstrap estimate \eqref{BS2} for $x_0$ and $s_0$ sufficiently large.

\vspace{0.2cm}
\textit{Closing \eqref{BS1}.} 
\begin{itemize}
    \item $(BS1_{\sigma})$: Using \eqref{BS1},\eqref{BS2} and the relation $(1-\beta)(1-\theta) = -\frac{\beta}{2}$, we get
\begin{equation}\label{lambda_sigma_2_2theta}
    \begin{split}
        \Big\lvert \lambda(s) - \Big( \frac{2}{\int Q}c_0 \frac{1}{\theta -1} \Big)^2 \sigma^{2-2\theta}(s) \Big\rvert \lesssim |h(s)| \Big\lvert \lambda^{\frac{1}{2}}(s) - \frac{2}{\int Q} c_0 \frac{1}{\theta -1} \sigma^{1-\theta}(s)\Big\rvert 
        \lesssim s^{-\beta-2\rho}.
    \end{split}
\end{equation}
    From the definition of $\vec{m}$ in \eqref{def_of_vec_m_and_vecM} and it's control in \eqref{conseq_of_BS_on_m_on_bs}, it holds $|\sigma_s - \lambda| \lesssim s^{-\beta - \frac 54}$. Therefore, we have
\begin{equation}
    \begin{split}
        \Big\lvert \sigma_s - \Big( \frac{2}{\int Q}c_0 \frac{1}{\theta -1} \Big)^2 \sigma^{2-2\theta}(s) \Big\rvert \lesssim \Big(|h(s)| + \big| \sigma_s^{\frac{1}{2}} - \lambda^{\frac{1}{2}} \big|\Big)\, \Big\lvert \sigma_s^{\frac{1}{2}}(s) - \frac{2}{\int Q} c_0 \frac{1}{\theta -1} \sigma^{1-\theta}(s)\Big\rvert 
        \lesssim s^{-\beta-2\rho}.
\end{split}
\end{equation}
Using 
\begin{equation}\label{relation_on_constant_2theta_minus_one}
    \Big( \frac{2}{\int Q} c_0 \frac{1}{\theta -1}\Big)^2 = (1-\beta )^{-\frac{\beta}{1-\beta}} = (2\theta -1)^{2\theta - 2},
\end{equation}

yields
\begin{equation}
    \Big\lvert \big(\sigma^{\frac{1}{1-\beta}}(s)\big)_s - (1-\beta)^{-\frac{1}{1-\beta}} \Big\rvert = \frac{1}{1-\beta} \sigma^{\frac{\beta}{1-\beta}}(s) \Big\lvert \sigma_s - \Big( \frac{2}{\int Q}c_0 \frac{1}{\theta -1} \Big)^2 \sigma^{2-2\theta}(s) \Big\rvert \lesssim s^{-2 \rho}.
\end{equation}

Since $\rho < \frac 12$, the integration on $[s_0,s]$ of the previous estimate yields
\begin{equation}
    \Big|\sigma^{\frac{1}{1-\beta}}(s) - (1-\beta)^{-\frac{1}{1-\beta}}s \Big| \lesssim \Big|\sigma^{\frac{1}{1-\beta}}_0 - (1-\beta)^{-\frac{1}{1-\beta}} s_0 \Big| + s^{1-2\rho}.
\end{equation}

The initial condition \eqref{init_cond_sigma_b} on $\sigma_0$, the mean value theorem for the function $x \mapsto x^{\frac{1}{1-\beta}}$, $\frac{1}{1-\beta} > 2$ and the bound $0<\rho<\frac 12$ imply
\begin{equation}
    \Big|\sigma^{\frac{1}{1-\beta}}_0 - (1-\beta)^{-\frac{1}{1-\beta}} s_0 \Big| \lesssim  s^{\beta}_0 \,(1-\beta)^{-\frac{\beta}{1-\beta}}\,\Big| \sigma_0 - \frac{1}{1-\beta}s^{1-\beta}_0 \Big| \lesssim s^{1-2\rho}_0 \lesssim s^{1-2\rho}
\end{equation}
As previously, by the mean value theorem for $x\mapsto x^{\frac{1}{1-\beta}}$ and $x\mapsto x^{1-\beta}$, we get
\begin{equation}
    \Bigg\lvert  \Big( \frac{(1-\beta)\sigma(s)}{s^{1-\beta}} \Big)^{\frac{1}{1-\beta}} -1  \Bigg\rvert \lesssim s^{-2\rho} \iff \Bigg\lvert  \Big( \frac{(1-\beta)\sigma(s)}{s^{1-\beta}} \Big) -1  \Bigg\rvert \lesssim s^{-2\rho}
\end{equation}
which yields the closing on the parameter $\sigma$.
\item $(BS1_\lambda)$: From \eqref{lambda_sigma_2_2theta} we get
\begin{equation}
    \Big\lvert \lambda(s) - s^{-\beta} \Big\rvert \lesssim s^{-\beta-2\rho}  + \Bigg\lvert  \Bigg( \frac{2}{\int Q}c_0 \frac{1}{\theta -1} \Bigg)^2 \sigma^{2-2\theta}(s) - s^{-\beta} \Bigg\rvert.
\end{equation}
By $(BS1_{\sigma})$, \eqref{relation_on_constant_2theta_minus_one}, the mean value theorem and since $\beta > \frac{1}{2}$, we get the desired estimate
\begin{equation}
    \Big\lvert \lambda(s) - s^{-\beta} \Big\rvert \lesssim s^{-\beta-2\rho} + s^{1-3\beta-2\rho} \lesssim s^{-\beta-2\rho}.
\end{equation}

\item $(BS1_b)$: We write
\begin{equation}
    \begin{split}
        &\Big\lvert \frac{1}{\beta}sb(s) -1 \Big\rvert = \frac{1}{\beta}\lambda^2(s)s \Big\lvert \frac{b(s)}{\lambda^2(s)} - \beta s^{-1}\lambda^{-2}(s)\Big\rvert \\
        &\lesssim \lambda^2(s)s \Bigg( \lvert g(s) \rvert + \Big\lvert \frac{4}{\int Q}c_0 \lambda^{-\frac{3}{2}}(s)\sigma^{-\theta}(s) + \beta s^{-1}\lambda^{-2}(s) \Big\rvert \Bigg) .
    \end{split}
\end{equation}
 By definition of $c_0$ and of $\beta$ we find
 \begin{equation}
 \begin{split}
     &\Big\lvert \frac{4}{\int Q}c_0 \lambda^{-\frac{3}{2}}(s)\sigma^{-\theta}(s) + \beta s^{-1}\lambda^{-2}(s) \Big\rvert  = 2(\theta-1)(2\theta-1)^{-1}\lvert \lambda^{-\frac{3}{2}}(s)\rvert \Bigg\lvert \Big( \frac{\sigma(s)}{2\theta-1} \Big)^{-\theta} - s^{-1}\lambda^{-\frac{1}{2}} \Bigg\rvert\\
     &\lesssim  \lvert \lambda^{-\frac{3}{2}}(s)\rvert \Bigg( \Big\lvert \Big( \frac{\sigma(s)}{2\theta-1} \Big)^{-\theta} - s^{-1+\frac{\beta}{2}} \Big\rvert +  \lvert s^{-1}\lambda^{-\frac{1}{2}}(s)- s^{-1+\frac{\beta}{2}} \rvert \Bigg)\\
     &\lesssim s^{\frac{3}{2}\beta}(s^{-1-2\rho}+ s^{-1-2\beta-2\rho}) \lesssim s^{\frac{3}{2}\beta-1-2\rho}.
\end{split}
 \end{equation}
We have used the next result. Since $2\theta-1 = (1-\beta)^{-1}$, using $(BS1_{\sigma})$ we have
 \begin{equation}
 \begin{split}
     \Bigg\lvert \Big( \frac{\sigma(s)}{2\theta-1} \Big)^{-\theta} - s^{-1+\frac{\beta}{2}} \Bigg\rvert = \Big\lvert \Big( (1-\beta)\sigma(s) \Big)^{-\theta} - \Big(s^{1-\beta}\Big)^{-\theta}\Big\rvert \lesssim s^{-1-2\rho}.
    \end{split}
 \end{equation}
Thus by $(BS1_\lambda)$ we find
\begin{equation}
    \Big\lvert \frac{4}{\int Q}c_0 \lambda^{-\frac{3}{2}}(s)\sigma^{-\theta}(s) + \beta s^{-1}\lambda^{-2}(s) \Big\rvert \lesssim s^{\frac{3}{2}\beta-1-2\rho}.
\end{equation}
We conclude then, using the estimates \eqref{conseq_of_BS_on_g_and_hs} and \eqref{BS3}
\begin{equation}
\begin{split}
    &\Big\lvert \frac{1}{\beta}sb(s)-1 \Big\rvert \lesssim \lambda^2(s)\,s\,\Big(|g(s)| + \Big|\frac{4}{\int Q} c_0\,\lambda^{\frac 32}(s)\,\sigma^{-\theta}(s) + \beta\,s^{-1}\,\lambda^{-2}(s) \Big| \Big)\\ 
    &\lesssim s^{1-2\beta}\Big( s^{2\beta-1-2\rho} + s^{\frac{3}{2}\beta-1-2\rho} \Big) \lesssim s^{-2\rho} .
\end{split}
\end{equation}
\end{itemize}

\vspace{0.2cm}
\textit{Contradiction by a continuity argument.}\\
Define the map
\begin{equation}
    \Phi: \lambda_0 \in [\lambda^-_0,\lambda^+_0] \xmapsto{\Phi_1} \mu_0 \in [-1,1] \xmapsto{\Phi_2} s^*\in [s_0,+\infty)\xmapsto{\Phi_3} h(s^*)(s^*)^{\frac{\beta}{2}+2\rho}\in \{-1,1\}.
\end{equation}
In the following, we prove that the map $\Phi$ is continuous and $\Phi(\lambda_0^-)=-1$, $\Phi(\lambda_0^+)=1$. Given that $\{-1,1\}$ is discrete, it will lead to a contradiction.

The map $\Phi_1$ is linear, continuous and such that $\Phi_1(\lambda^{\pm}_0) = \pm1$. By the definition of $h$, the map $\Phi_3$ is also continuous. 


We define $G(s) = h^2(s)s^{\beta+4\rho}$. For $s_1 \in [s_0,s^*]$ such that $G(s_1) = 1$, the following transversality property holds true
\begin{equation}
\begin{split}
    G'(s_1) 
    &= 2h_s(s_1)h(s_1)(s_1)^{\beta + 4\rho} + (\beta + 4\rho)h^2(s_1)(s_1)^{\beta+4\rho-1}\\
    &= \pm 2 h_s(s_1)(s_1)^{\frac{\beta}{2}+2\rho} + (\beta + 4\rho)\frac{G(s_1)}{s_1}\\
    &\geq -2(s_1)^{-1-2\rho}+\frac{1}{2s_1} \geq -\frac{1}{20s_1}+ \frac{1}{2s_1} >\frac{1}{20s_1} >0.
\end{split}
\end{equation}

In particular, $G(s^*) = 1$ and $G(s_0) = 1$. From the definition of $\lambda^{\pm}_0$, the initial conditions in \eqref{init_cond_sigma_b} and the transversality property on $s_0$, it holds $\Phi_2 \circ \Phi_1 (\lambda^{\pm}_0) = s_0$. Thus, $\Phi(\lambda^{\pm}_0)= \pm1$.

It remains to prove that the map $\Phi_2$ is continuous. Let $\mu_0\in (-1,1)$, so we have $s^* > s_0$. 

By continuity of $G$ and $G'$ we know that for some $\epsilon <s^*-s_0$ and $\delta_\epsilon>0$ sufficiently small holds $G(s)<1-\delta$ for all $s \in [s_0,s^*-\epsilon]$ and $G(s^*+\epsilon)> 1+\delta$. 
By continuous dependence on initial data for parameters and continuity of $\Phi_1$, there exists $\epsilon_\mu$, such that 
\begin{equation}
    \lvert \mu_0 - \tilde{\mu}_0\rvert\leq \epsilon_\mu \Rightarrow \lvert G(s) - \tilde{G}(s) \rvert \leq \frac{\delta}{2} \qquad 
    \text{on }\;[s_0,s^*+\epsilon].
\end{equation}

Let $\tilde{G}$ and $\tilde{s}^*$ be the quantities associated with $\tilde{\mu}_0$. 
We get 
\begin{equation*}
    \forall s \in [s_0,s^*-\epsilon], \quad  \tilde{G}(s) \leq G(s)+\frac{\delta}{2} < 1-\frac{\delta}{2} \quad\Rightarrow\quad \tilde{s}^* \geq s^*-\epsilon
\end{equation*}
and
\begin{equation*}
    \tilde{G}(s^*+\epsilon)\geq G(s^*+\epsilon)-\frac{\delta}{2} >1+\frac{\delta}{2} \quad\Rightarrow\quad \tilde{s}^* \leq s^*+\epsilon.
\end{equation*}
This concludes the continuity of $\Phi_2$ for $\mu_0 \in (-1,1)$. In the case $\mu_0 = \pm1$, we get $s^* = s_0$. By a similar argument, we get the continuity for $\mu_0 = \pm1$.
\end{proof}

\vspace{0.4cm}
\begin{proof}[Proof of the Theorem \ref{Theorem_principal_result}:]

Consider
\begin{equation}
    \tilde{\vare}(t,x) : = U(t,x) - \lambda^{-\frac 12}(s(t))Q(y).
\end{equation}

By a change of variables, the estimates \eqref{estimate_on_L2_norm_of_vare}, \eqref{BS2} and the Lemma \ref{lemma_on_the_decaying_tail}, we have
\begin{equation}
    \|\tilde{\vare}(t)\|_{L^2_x} \leq b\|P_b\|_{L^2} + r\|R\|_{L^2} + \|\vare(s)\|_{L^2} + \|f(t)\|_{L^2}
    \leq \delta(s^{-1}_0,x^{-1}_0) + \Big|\int U^2_0 - \int Q^2 \Big|^{\frac 12}.
\end{equation}
Since $\|U_0 - Q\|_{H^1} \leq \eta$, it holds
\begin{equation}
    \Big|\int U^2_0 - \int Q^2 \Big| \leq \|\tilde{\vare}_0\|_{L^2} + C(b_0 + r(s_0) + \|\vare_0\|_{L^2} + \|f_0\|_{L^2}) \leq \eta^2 + \delta(s^{-1}_0,x^{-1}_0).
\end{equation}
Choosing $s_0,x_0$ be large enough, we get the first estimate in \eqref{main_theorem_estimate_remainder}.

Proceeding as previously and combining \eqref{estimate_on_L2_norm_of_grad_vare} with \eqref{conseq_of_BS_on_m_on_bs} and the bound on $\rho$ in \eqref{condition_on_rho}, we have
\begin{equation}\label{proof_main_thm_estimate_in_tilde_dx_vare}
    \|\partial_x \tilde{\vare}(t)\|_{L^2_x} \leq \lambda^{-1}\big( b\|P'_b\|_{L^2} + r\|R'\|_{L^2}+\|\partial_y \vare (s)\|_{L^2} \big) + \|\partial_x f(t)\|_{L^2}\lesssim \lambda^{-1} s^{-( \frac 12 + 2\rho )}.
\end{equation}
Which yields the limit in \eqref{main_theorem_estimate_remainder}.

From the definition of the rescaled variable, we have
    \begin{equation}
        T-t = \int^{+\infty}_{s(t)} \lambda^3(s')\,ds' \sim \frac{1}{3\beta-1}\,s(t)^{-(3\beta-1)}.
    \end{equation}
Therefore, from \eqref{BS2} and the relation $\nu = \frac{\beta}{3\beta-1}$, we get for $t \uparrow T$
    \begin{equation}\label{proof_main_thm_law_of_lambda_in_t}
        \lambda(s(t)) \sim \big[ (3\beta-1)(T-t) \big]^{\frac{\beta}{3\beta-1}} \sim (3\nu -1 )^{-\nu}(T-t)^{\nu}
    \end{equation}
and 
    \begin{equation}
        \sigma(s(t)) \sim \frac{1}{1-\beta}\big[ (3\beta - 1)(T-t)\big]^{\frac{\beta-1}{3\beta-1}} \sim \frac{(3\nu-1)^{2\nu}}{2\nu-1} (T-t)^{1-2\nu}.
    \end{equation}

The estimate \eqref{proof_main_thm_estimate_in_tilde_dx_vare} implies for $t \uparrow T$ that 
\begin{equation}
    \|\partial_x U(t)\|_{L^2} \sim \lambda^{-1}(s(t))\|Q'\|_{L^2}.
\end{equation}
Therefore, combining with \eqref{proof_main_thm_law_of_lambda_in_t}, it holds for $t\uparrow T$
\begin{equation}
    \|\partial_x U(t)\|_{L^2} \sim (3\nu-1)^{\nu}\|Q'\|_{L^2}\, (T-t)^{-\nu},
\end{equation}
which yields that $\nu$ is the blow up rate of the constructed solution $U$ and finishes the proof of the Theorem \ref{Theorem_principal_result}.
\end{proof}

\appendix
\section{}
\begin{proof}[Proof of the Lemma \ref{lemma-localized-virial-estimate} : Localised virial estimate]
The proof is similar to \cite[Lemma 3.5]{Combet-Martel-17}. First we announce the coercivity property of a virial quadratic form under suitable repulsivity properties.
\begin{proposition}{\cite[Proposition 4]{Martel-Merle-00}}\label{Prop_Martel_Merle00}  There exists $\mu>0$ such that, for all $v\in H^1(\RR)$, it holds
    \begin{equation}
        3\int v^2_y + \int v^2 - 5\int Q^4 v^2 + 20 \int y Q' Q^3 v^2 \geq \mu \int (v^2_y + v^2) - \frac{1}{\mu}\Big( \int v\, y \, \Lambda Q\Big)^2 - \frac{1}{\mu} \Big( \int v\, Q \Big)^2.
    \end{equation}
\end{proposition}

\vspace{0.4cm}
We consider a cut-off function 
\begin{equation}
    \zeta = \begin{cases}
        0, \quad y<-\frac B 2\\
        1, \quad y>-\frac B 4
    \end{cases} \quad , \quad 0\leq \zeta \leq 1 \quad \text{on}\;\ \RR.
\end{equation}
Set $\zeta_B(y) = \zeta\big( \frac{y}{B}\big)$ and $\tilde{\vare}(y) = \vare(y)\,\zeta_B(y)$.
By definition of $\zeta$, we get
\begin{equation}
    \int \tilde{\vare}^2 \leq \int_{y>-\frac B2} \vare^2\,, \qquad \int (\partial_y \tilde{\vare})^2 \leq \int_{y>-\frac B2}(\partial_y \vare)^2 + \int \big( (\zeta'_B)^2 - \frac{1}{2}(\zeta^2_B)'' \big) \, \vare^2 \leq \int_{y>-\frac B2} (\partial_y \vare)^2 + \frac{C}{B^2} \int_{-\frac B2}^{-\frac B 4} \vare^2
\end{equation}
and
\begin{equation}
    -5\int Q^4 \tilde{\vare}^2 + 20\int y\, Q'Q^3 \tilde{\vare}^2 \leq -5\int_{y>-\frac B2} Q^4 \vare^2  + 20 \int_{y>-\frac B2} y\, Q'Q^3 \vare^2  + \frac{C}{B^{10}} \int_{-\frac B2}^{-\frac B 4}\vare^2.
\end{equation}
By the orthogonality condition in \eqref{orthogonal_conditions_in_lemma}, we get
\begin{equation}
    \Big| \int \vare \, \zeta_B \, y \Lambda Q \Big| = \Big| \int \vare\, (1-\zeta_B)\,y \Lambda Q \Big| \leq \frac{C}{B^{10}}\Big(\int_{<-\frac B 4} \vare^2 \, e^{-\frac{|y|}{2}} \Big)^{\frac 12}.
\end{equation}
The Proposition \ref{Prop_Martel_Merle00} imply for $B \gg 1$
\begin{equation}
    (3-\mu)\int_{y>-\frac B2} (\partial_y \vare)^2 + (1 - \frac \mu 2)\int_{y>-\frac B2} \vare^2 - 5\int_{y>-\frac B2} Q^4 \vare^2 + 20 \int_{y>-\frac B2} y \, Q' Q^3 \vare^2  \geq -\frac C\mu \frac{1}{B^{10}}\int \vare^2 \,e^{-\frac{|y|}{2}}.
\end{equation}
\end{proof}
\printbibliography

@article{MMPI,
        author = {Y. Martel and F. Merle and P. Raphael},
        title = {Blow up for the critical gKdV equation I: dynamics near the soliton},
        journal = {Acta Math. 212 },
        pages = {59-140},
        year = {2014}
}

@article{MMPII,
        author = {Y. Martel and F. Merle and P. Raphaël},
        title = {Blow up for the critical gKdV equation. II: Minimal mass dynamics},
        journal = {J. Eur. Math. Soc. (JEMS) },
        volume  = {17},
        number = {No. 8},
        pages = {1855-1925},
        year = {2015}
}

@article{MMPIII,
        author = {Y. Martel and F. Merle and P. Raphael},
        title = {Blow up for the critical gKdV equation III: exotic regimes},
        journal = {Pisa Cl. Sci. (5). },
        volume = {XIV},
        pages = {575-631},
        year = {2015}
}

@article{Martel-Pilod,
        author = {Y. Martel and D. Pilod},
        title = {Full Family of Flattening Solitary Waves for the Critical Generalized KdV Equation},
        journal = {Communications in Mathematical Physics},
        volume  = {378},
        pages = {1011–1080},
        year = {2020}
}

@article{KPV93,
        author = {C.E. Kenig and G. Ponce and L. Vega},
        title = {Well-posedness and scattering results for the generalized korteweg-de vries equation via the contraction principle},
        journal = {Comm. Pure Appl. Math.},
        volume  = {46},
        pages = {527-620},
        year = {1993}
}

@article{Combet-Martel-17,
        author = {V. Combet and Y. Martel},
        title = {Sharp asymptotics for the minimal mass blow up solution of the critical gKdV equation},
        journal = {Bull. Sci. Math.},
        volume  = {141},
        number = {No. 2},
        pages = {20-103},
        year = {2017}
}

@article{Combet-Martel-18,
        author = {V. Combet and Y. Martel},
        title = {Construction of multibubble solutions for the critical gKdV equation},
        journal = {SIAM J. Math. Anal.},
        volume  = {50},
        number = {No. 4},
        pages = {3715-3790},
        year = {2018}
}

@article{Martel-Merle-00,
        author = {Y. Martel and F. Merle},
        title = {A Liouville theorem for the critical generalized Korteweg-de Vries equation},
        journal = {J. Math. Pures Appl.},
        volume  = {(9)79},
        number = {4},
        pages = {339-425},
        year = {2000}
}

@article{Martel-Merle-01,
        author = {Y. Martel and F. Merle},
        title = {Instability of solitons for the critical generalized Korteweg de Vries equation},
        journal = {Geometric And Functional Analysis},
        volume  = {11},
        number = {1},
        pages = {74-123},
        year = {2001}
}

@article{Weinstein-85,
        author = {M.I. Weinstein},
        title = {Modulational stability of ground states of nonlinear Schrödinger equations},
        journal = {SIAM J. Math. Anal.},
        volume  = {16},
        pages = {472-491},
        year = {1985}
}

@article{Martel-Merle-02,
        author = {Y. Martel and F. Merle},
        title = {Stability of blow-up profile and lower bounds for blow-up rate for the critical generalized KdV equation},
        journal = {Ann. Math. (2) 155},
        volume  = {No. 1},
        pages = {235-280},
        year = {2002}
}

@article{Nakanishi-16,
        author = {Y. Martel and F. Merle and K. Nakanishi and  Pierre Raphaël},
        title = {Codimension one threshold manifold for the critical gKdV equation.},
        journal = {Commun. Math. Phys. 342},
        volume  = {No. 3},
        pages = {1075-1106},
        year = {2016}
}

@article{Merle-1993-minmass-nls,
        author = {F. Merle},
        title = {Determination of blow-up solutions with minimal mass for nonlinear Schrödinger equations with critical power.},
        journal = {Duke Math. J.},
        volume  = {69},
        number = {No. 2},
        pages = {427-454},
        year = {1993}
}

@article{LPSS-1988-first-blow-up-nls,
        author = {M.J. Landman and G.C. Papanicolaou and C. Sulem and P.-L. Sulem},
        title = {Rate of blowup for solutions of the nonlinear Schrödinger equation at critical dimension},
        journal = {Phys. Rev. A},
        volume  = {38},
        pages = {3837–3843},
        year = {1988}
}

@article{Perelman-2001-nls-d1,
        author = {G. Perelman},
        title = {On the formation of singularities in solutions of the critical nonlinear Schrödinger equation.},
        journal = {Ann. Henri Poincaré},
        volume  = {2},
        number = {No. 4},
        pages = {605-673},
        year = {2001}
}

@article{Merle-Raphael-crit-nls-2005,
        author = {F. Merle and P. Raphaël},
        title = {The blow up dynamics and upper bound on the blow up rate for the critical nonlinear Schrödinger equation.},
        journal = {Ann. of Math.},
        volume  = 161,
        pages = {157–222},
        year = {2005}
}

@article{Fibich-Merle-Raphael-spec-prop-2006,
        author = {G. Fibich and F. Merle and P. Raphaël},
        title = {Proof of a spectral property related to the singularity formation for the $L^2$ critical nonlinear Schrödinger equation.},
        journal = {Phys. D},
        volume  = {220},
        pages = {1-13},
        year = {2006}
}

@article{Merle-Raphael-profile-quant-2005,
        author = {F. Merle and P. Raphaël},
        title = {Profiles and quantization of the blow up mass for critical nonlinear Schrödinger equation.},
        journal = {Commun. Math. Phys.},
        volume  = {253},
        number = {No. 3},
        pages = {675-704},
        year = {2005}
}

@article{Krieger-Schlag-Tataru-renorm-2008,
        author = {J. Krieger and W. Schlag and D. Tataru},
        title = {Renormalization and blow up for charge one equivariant critical wave maps.},
        journal = {Invent. Math.},
        volume  = {171},
        pages = {543–615},
        year = {2008}
}

@article{Gustafson-Nakanishi-Tsai-asympt-2008,
        author = {S. Gustafson and K. Nakanishi and T.-P. Tsai},
        title = {Asymptotic stability, concentration and oscillations in harmonic map heat flow, Landau Lifschitz and Schrödinger maps on $\RR^2$.},
        journal = {Comm. Math. Phys.},
        volume  = {300},
        pages = {205-242},
        year = {2010}
}

@article{Weinstein-1983,
        author = {M.I. Weinstein},
        title = {Nonlinear Schrödinger equations and sharp interpolation estimates.},
        journal = {Comm. Math. Phys.},
        volume  = {87},
        pages = {567–576},
        year = {1983}
}

@article{Bourgain-Wang-1998,
        author = {J. Bourgain and W. Wang},
        title = {Construction of blowup solutions for the nonlinear Schrödinger equation with critical nonlinearity.},
        journal = {Ann. Sc. Norm. Super. Pisa, Cl. Sci.},
        volume  = {IV, Ser. 25},
        number = {No. 1-2},
        pages = {197-215},
        year = {1997}
}

@article{Martel-Merle-2002,
        author = {Y. Martel and F. Merle},
        title = {Stability of blow-up profile and lower bounds for blow-up rate for the critical generalized KdV equation.},
        journal = {Ann. Math.},
        volume  = {(2) 155},
        number = {No. 1},
        pages = {235-280},
        year = {2002}
}

@article{Martel-Pilod-2024,
        author = {Y. Martel and D. Pilod},
        title = {Finite point blowup for the critical generalized Korteweg-de Vries equation.},
        journal = {Ann. Sc. Norm. Super. Pisa, Cl. Sci.},
        volume  = {(5) 25},
        number = {No. 1},
        pages = {371-425},
        year = {2024}
}

@article{Martel-Pilod-2026,
        author = {Y. Martel and D. Pilod},
        title = {Continuum of finite point blowup rates for the critical generalized Korteweg-de Vries equation.},
        journal = {Preprint, arXiv:2601.20801},
        year = {2026}
}

@article{Krieger-Schlag-Tataru-2009,
        author = {J. Krieger and W. Schlag and D. Tataru},
        title = {Slow blow-up solutions for the $H^1(\RR^3)$ critical focusing semilinear wave equation.},
        journal = {Duke Math. J.},
        volume  = {147},
        number = {No. 1},
        pages = {1-53},
        year = {2009}
}

@article{Krieger-Schlag-2014,
        author = {J. Krieger and W. Schlag},
        title = {Full range of blow up exponents for the quintic wave equation in three dimensions.},
        journal = {J. Math. Pures Appl.},
        volume  = {(9) 101},
        number = {No. 6},
        pages = {873-900},
        year = {2014}
}

@article{Donninger-Huang-Krieger-Schlag-2014,
        author = {R. Donninger and M. Huang and J. Krieger and W. Schlag},
        title = {Exotic blowup solutions for the $u^5$ focusing wave equation in $\RR^3$.},
        journal = {Mich. Math. J.},
        volume  = {63},
        number = {No. 3},
        pages = {451-501},
        year = {2014}
}

@article{Krieger-Schlag-2009,
        author = {J. Krieger and W. Schlag},
        title = {Non-generic blow-up solutions for the critical focusing NLS in $1-D$.},
        journal = {J. Eur. Math. Soc.},
        volume  = {11},
        number = {No. 1},
        pages = {1-125},
        year = {2009}
}

@article{Jendrej-2017,
        author = {J. Jendrej},
        title = {Construction of type II blow-up solutions for the energy-critical wave equation in dimension 5.},
        journal = {J. Funct. Anal.},
        volume  = {272},
        number = {No. 3},
        pages = {866-917},
        year = {2017}
}

@article{Kato-1983,
        author = {T. Kato},
        title = {On the Cauchy problem for the (generalized) Korteweg–de Vries equation.},
        journal = {Stud. Appl. Math.},
        volume  = {8},
        pages = {93–128},
        year = {1983}
}

@article{Jendrej-Lawrie-Rodriguez-2022,
        author = {J. Jendrej and A. Lawrie and C. Rodriguez},
        title = {Dynamics of bubbling wave maps with prescribed radiation.},
        journal = {Ann. Sci. Éc. Norm. Supér.},
        volume  = {(4) 55},
        number = {No. 4},
        pages = {1135-1198},
        year = {2022}
}

@article{Raphael-2005,
        author = {P. Raphael},
        title = {Stability of the log-log bound for blow up solutions to the critical nonlinear Schrödinger equation},
        journal = {Math. Ann.},
        volume  = {331},
        number = {No. 3},
        pages = {577-609},
        year = {2005}
}

@article{Yang-Roudenko-Zhao-blowup-dynamics-2018,
        author = {K. Yang and S. Roudenko and Y. Zhao},
        title = {Blow-up dynamics and spectral property in the $L^2$-critical nonlinear Schrödinger equation in high dimensions},
        journal = {Nonlinearity},
        volume  = {31},
        number = {No. 9},
        pages = {4354-4392},
        year = {2018}
}

@article{Kim-Kihun-et-al-2024,
        author = {K. Kim and S. Kwon and S.-J. Oh},
        title = {Blow-up dynamics for radial self-dual Chern-Simons-Schrödinger equation with prescribed asymptotic profile},
        journal = {Preprint, arXiv:2409.13274},
        year = {2024}
}

@book{Linares-Ponce,
  author = {F. Linares and G. Ponce},
  year = {2015},
  title = {Introduction to Nonlinear Dispersive Equations},
  publisher = {Springer},
  edition = {2}
}
\end{document}